\documentclass[review,onefignum,onetabnum]{siamart190516}

\usepackage{verbatim}
\usepackage{lipsum}
\usepackage{amsfonts}
\usepackage{graphicx}
\usepackage{subfigure}
\usepackage{epstopdf}
\usepackage{algorithmic}
\ifpdf
  \DeclareGraphicsExtensions{.eps,.pdf,.png,.jpg}
\else
  \DeclareGraphicsExtensions{.eps}
\fi

\renewcommand\theequation{\thesection.\arabic{equation}}
\theoremstyle{plain}
\newtheorem{thm}{Theorem}[section]

\newtheorem{remark}{\textbf{Remark}}[section]
\newcommand{\eps}{\epsilon}

\newcommand{\bm}{\boldsymbol}
\newcommand{\bu}{{\mathbf u}}
\newcommand{\bv}{{\mathbf v}}

\newcommand{\Grad}[1]{\nabla #1}

\newcommand{\be}{\begin{equation}}
\newcommand{\ee}{\end{equation}}

\newcommand{\bse}{\begin{subequations}}
\newcommand{\ese}{\end{subequations}}
\def\benl{\begin{eqnarray*}}
\def\eenl{\end{eqnarray*}}

\def\bu{\bm{u}}

\def\be{\bm{e}}

\def\bx{\bm{x}}

\def\bmu1{\bm{\mu_1}}

\def\ln{{\rm ln}}

\newcommand{\ben}{\begin{eqnarray}}
\newcommand{\een}{\end{eqnarray}}
\newcommand{\beq}{\begin{equation}}
\newcommand{\eeq}{\end{equation}}
\newcommand{\bea}{\begin{array}}
\newcommand{\eea}{\end{array}}
\newcommand{\bef}{\begin{figure}[H]}
\newcommand{\eef}{\end{figure}}

\newsiamremark{hypothesis}{Hypothesis}
\crefname{hypothesis}{Hypothesis}{Hypotheses}
\newsiamthm{claim}{Claim}

\title{The  generalized scalar auxiliary variable approach (G-SAV)  for gradient flows}

\author{Qing Cheng\thanks{Department of Applied Mathematics, Illinois Institute of Technology, Chicago, IL 60616, USA 
  (\email{qcheng4@iit.edu)}. }}
\usepackage{amsopn}

\begin{document}
\bibliographystyle{plain}
\graphicspath{ {Figures/} }
\maketitle

% REQUIRED
\begin{abstract}
We  establish a general framework for developing, efficient energy stable numerical schemes for gradient flows  and develop three classes of   generalized scalar auxiliary variable approaches (G-SAV).   Numerical schemes  based on the G-SAV approaches  are as efficient as the original  SAV schemes \cite{SXY19,cheng2018multiple}  for gradient flows, i.e., only require solving linear  equations with constant coefficients at each time step,   can be unconditionally energy stable.   But G-SAV approaches  remove  the definition restriction that auxiliary variables can only be square root function.  The definition form of auxiliary variable is applicable to any reversible function for G-SAV approaches .  Ample numerical results for phase field models are presented  to validate the effectiveness and accuracy of the proposed G-SAV numerical schemes.
\end{abstract}

% REQUIRED
\begin{keywords}
 gradient flow, G-SAV approach,  SAV approach, energy stability, phase-field
\end{keywords}

% REQUIRED
%\begin{AMS}
 % 68Q25, 68R10, 68U05
%\end{AMS}

\section{Introduction}
Gradient flows have been widely used  in science and engineering in the last few decades.
Due to the second law of  thermodynamics,  the common characteristic of gradient flows  is to satisfy the energy dissipative law in time.  Correspondingly,  tremendous efforts have been devoted to the construction of efficient and accurate numerical methods preserving the energy  dissipative law  at the discrete level for gradient flows.   We refer the reader  to the related  review papers  \cite{Du.F19,SXY19,She.Y19}  and convex splitting method \cite{elliott1993global,eyre1998unconditionally,baskaran2013convergence}, stabilized method \cite{qiao2012stability,du2018stabilized}, Average Vector Field method \cite{celledoni2012preserving,quispel2008new},  the newly developed  IEQ method \cite{yang2016linear,YANG2017_3phase},  SAV method (cf. \cite{SXY19,cheng2018multiple}), Lagrange multiplier methods \cite{guillen2013linear,cheng2019new}, RK-SAV method \cite{akrivis2019energy} and gPAV method \cite{Dong2020}  which have received much attention recently due to their efficiency, flexibility and accuracy.

The key point  to achieve the property of preserving energy decay  for  IEQ  and SAV approaches  lies in introducing auxiliary variables. By taking time derivative with respect to auxiliary variables,  the original PDE system of  gradient flows can be transformed  into an equivalent form.  After treating the new auxiliary variables implicitly and nonlinear terms explicitly, several classes of energy stable numerical schemes can be developed (cf.  \cite{SXY19,yang2016linear}).  But the form of the auxiliary variable  is only limited  to the square root function which is critical to the proof of the discrete energy stability.  Meanwhile,  the square root function brings inconvenience to the terms  that energy can not be bounded below \cite{cheng2019highly,li2016characterizing}.  
 In our new  framework  the square root function for auxiliary variable  is not essential to devise energy stable schemes.  For our G-SAV approach,  the definition of scalar  auxiliary variable can be any invertible function with respect to energy functional.  Energy stable numerical schemes can be constructed for a large class of dissipative systems by using our 
 generalized-SAV approach.

The goal of this paper is to establish  a general framework to develop efficient and accurate  time discretization for gradient flows.  We propose three numerical approaches to construct schemes which enjoy all advantages of the SAV schemes,  but also remove the
definition restriction that auxiliary variables can only be square root function.   In this paper three different G-SAV  approaches will be considered: (i)  in the first approach, we  define the auxiliary variable as any invertible function with respect to  energy functional and derive  a large class of energy stable schemes which only need to solve  linear, constant coefficients equation.  The small price to pay is that we need to solve  a nonlinear algebraic equation whose cost is negligible. (ii) In the second approach, we treat the new auxiliary variable explicitly and derive a class of linear, no-iterative, energy stable schemes. But our numerical simulations for BCP mode \cite{avalos2016frustrated,avalos2018transformation}  in section $6$ indicate that the first approach is more robust than the second approach, even though the second approach is much more easy to implement.  (iii) In the third approach, we derive G-SAV  schemes which preserve original energy dissipative law instead of modified energy.  Actually the third approach is equivalent with the new Lagrange multiplier approach \cite{cheng2019new}.  In summary,  there are several  advantages for our new three numerical approaches:
\begin{itemize}
\item For the first approach: the auxiliary variable can be defined by any invertible function
 which eliminates  the constraint of free energy to be bounded below.  The small price to pay  is to solve a nonlinear algebraic system whose cost is negligible.

\item For the second  approach:  by treating the auxiliary explicitly, the numerical solution and auxiliary variable can be expressed explicitly.  It inherits all the advantages of the first approach, but do not need to solve a  nonlinear algebraic equation at each time step.  Furthermore, the auxiliary variable can be guaranteed to be positive by choosing  $\tanh$ function or exponential function which IEQ and SAV approaches can not preserve.
 While its drawback is that it may not  be robust as the first approach in dealing with models with extremely stiff terms, for example coupled Cahn-Hilliard model in Section $6$.

\item For the third  approach: compared with IEQ and original SAV approaches, it preserves original energy dissipative law instead of modified energy. 
The nonlinear part of free energy do not need  to be bounded from below as it is required  in the SAV approach. 

\end{itemize}

The reminder of this paper is structured as follows.  In Section 2,  we  present a general framework  for gradient flows  by using G-SAV approach with single component.   In Sections 3,  we apply the G-SAV approach for gradient flows with multiple components. In Section 4,  we introduce the second  approach for gradient flows. In Section 5, 
We introduce the third approach which preserves original energy dissipative law and the technique of stabilization and adaptive  time stepping strategy.   In Section 6,  we present several  numerical simulations for  Allen-Cahn and Cahn-Hilliard equations to show the validation of G-SAV schemes.  Some numerical experiments will also  be shown for coupled Cahn-Hilliard (BCP) model by using the first approach.  Some concluding remarks are given in Section 7.

\section{The first approach}
\label{sec:main}
We present in this section  a general methodology to develop energy stable numerical schemes  for  gradient flows.  To simplify the presentation,  we consider here single-component models for gradient flows.   The first  generalized G-SAV  approaches developed here will be extended to problems with multi-components models  in the subsequent sections. 

To fix the idea,  we consider a system with total free energy in the form 
\begin{equation}\label{orienergy}
 E(\phi)=\int_\Omega \frac12\mathcal{L}\phi\cdot \phi +F(\phi) d\bx,
\end{equation}
where $\mathcal{L}$ is certain linear positive operator,  $F(\phi)$ is a nonlinear potential.  Then a general gradient flow with the above free energy  takes the following form
\begin{equation}\label{grad:flow}
\begin{split}
&\phi_t=-\mathcal{G}\mu,\\
&\mu=\mathcal{L}\phi+F'(\phi).
\end{split}
\end{equation}
Where  $\mathcal{G}$  is  a  positive operator describing the relaxation process of the system. 
The boundary conditions can be either one of the following two type
\begin{eqnarray}
&&(i)\mbox{ periodic; or } (ii)\,\,\partial_{\bf n} \phi|_{\partial\Omega}=\partial_{\bf n}\mu|_{\partial\Omega}=0,
\end{eqnarray}
where $\bf n$ is the unit outward normal on the boundary $\partial\Omega$. Taking the inner products of the first two equations with $\mu$ and $-\phi_t$ respectively,   taking integration by part,  summing up these two  results,  we obtain the following energy dissipation law:
\begin{equation}
 \frac{d }{d t} E(\phi)=-(\mathcal{G}\mu,\mu),
\end{equation}
where, and in the sequel, $(\cdot, \cdot)$ denotes the inner product in $L^2(\Omega)$. We shall also denote the $L^2$-norm by $\|\cdot\|$. 

Below we shall introduce  the first G-SAV approach which preserves  energy stability  while retaining all essential advantages of the original  SAV \cite{shen2018scalar} approach.

\subsection{ The first G-SAV approach with single nonlinear potential}
We rewrite the original energy \eqref{orienergy} as 
\begin{equation}\label{modenergy}
 E(\phi)=\int_\Omega \frac12\mathcal{L}\phi\cdot \phi d\bx +G^{-1}\{G(
\int_{\Omega}F(\phi)d\bx)\} ,
\end{equation}
where  $G$ is an  invertible function.

Firstly we  define new scalar auxiliary  variable $r=G(\int_{\Omega}F(\phi)d\bx)$.
Taking  derivative of $r$ with respect to time,  we derive 
\begin{equation}\label{eq:1}
r_t=G'(\int_{\Omega}F(\phi)d\bx)(F'(\phi),\phi_t), 
\end{equation}
and notice the equality 
\begin{equation}\label{eq:2}
(G^{-1})'\{G(
\int_{\Omega}F(\phi)d\bx)\}=(G^{-1})'(r) =\frac{1}{G'(\int_{\Omega}F(\phi)d\bx)}.
\end{equation}

Using  two equalities  \eqref{eq:1}-\eqref{eq:2},  we shall rewrite  system \eqref{grad:flow}  
\begin{eqnarray}
&&\partial_t\phi=-\mathcal{G}\mu,\label{gsav:1}\\
&&\mu=\mathcal{L}\phi+\frac{r}{G(\int_{\Omega}F(\phi)d\bx)}F'(\phi)\label{gsav:2},\\
&&\frac{d}{dt}G^{-1}(r)=(G^{-1})'(r) r_t=\frac{r}{G(\int_{\Omega}F(\phi)d\bx)}(F'(\phi),\phi_t). \label{gsav:3}
\end{eqnarray}
There are various SAV approaches by choosing different functions $G$, for example 
\begin{itemize}
\item  The monotone  polynomial SAV approach: for example  $G=x^{\frac 13}$,  $x^{\frac 15}$ or $x^3$;
\item The square root  SAV approach and  original SAV approach \cite{shen2018scalar}: $G=\sqrt{x+C}$;
\item The mapped  exponential SAV approach: $G=e^{\frac{x}{C}}$;
\item The mapped  $\tanh$ SAV approach: $G=\tanh(\frac{x}{C})$.
\end{itemize}
Where $C$ is a positive constant.
Taking the inner products of the first two equations with $\mu$ and $-\phi_t$ respectively,   summing up the results along with the third  equation,  we obtain the following energy dissipation law:
\begin{equation}
\frac{d }{dt} \tilde E(\phi)=-(\mathcal{G}\mu,\mu),
\end{equation}
where $\tilde E(\phi)=\int_\Omega \frac12\mathcal{L}\phi\cdot \phi d\bx +G^{-1}(r) $ is a modified energy.
\subsection{Schemes  based on the G-SAV approach with single nonlinear potential}
A first-order  numerical scheme based on  the G-SAV approach for system \eqref{gsav:1}-\eqref{gsav:3}  is 
\begin{eqnarray}
&&\frac{\phi^{n+1}-\phi^n}{\delta t}=-\mathcal{G}\mu^{n+1},\label{scheme:gsav:1}\\
&&\mu^{n+1}=\mathcal{L}\phi^{n+1}+\frac{r^{n+1}}{G(\int_{\Omega}F(\phi^n)d\bx)}F'(\phi^n)\label{scheme:gsav:2},\\
&&\frac{(G^{-1}(r))^{n+1}-(G^{-1}(r))^n}{\delta t}=\frac{r^{n+1}}{G(\int_{\Omega}F(\phi^n)d\bx)}(F'(\phi^n),\frac{\phi^{n+1}-\phi^n}{\delta t}).\label{scheme:gsav:3}
\end{eqnarray}
Taking the inner products of \eqref{scheme:gsav:1} with $\mu^{n+1}$ and of \eqref{scheme:gsav:2} with $-\frac{\phi^{n+1}-\phi^n}{\delta t}$,  summing up the results and taking into account \eqref{scheme:gsav:3},  we have the following:
\begin{thm}\label{stable:1}
 The  scheme \eqref{scheme:gsav:1}-\eqref{scheme:gsav:3} is unconditionally energy stable in the sense that
 $$\tilde E(\phi^{n+1})-\tilde E(\phi^{n}) \le -\Delta t(\mathcal{G}\mu^{n+1},\mu^{n+1}),$$
 where $\tilde E(\phi^k)=\int_\Omega \frac12\mathcal{L}\phi^k\cdot \phi^k d\bx +(G^{-1}(r))^k$.
\end{thm}

Only the special case $G=x^{\frac 12}$,  we can derive linear,  no-iterative, original SAV schemes for system \eqref{sav:1}-\eqref{sav:3} in Remark \ref{sav:re}. Usually scheme \eqref{scheme:gsav:1}-\eqref{scheme:gsav:3}  is nonlinear and below we introduce how to solve it efficiently.
Setting $\xi^{n+1}=\frac{r^{n+1}}{G(\int_{\Omega}F(\phi^n)d\bx)}$ and writing 
\begin{equation}\label{sol}
\phi^{n+1}=\phi_1^{n+1}+\xi^{n+1}\phi_2^{n+1},\; \mu^{n+1}=\mu_1^{n+1}+\xi^{n+1}\mu_2^{n+1},\; 
 \end{equation}
in the above, we find that $(\phi_i^{n+1},\mu_i^{n+1})\;(i=1,2)$ can be determined  as follows:
\begin{eqnarray}
&&\frac{\phi_1^{n+1}-\phi^n}{\delta t}=-\mathcal{G}\mu_1^{n+1},\label{sav:1c}\\
&&\mu_1^{n+1}=\mathcal{L}\phi_1^{n+1},\label{sav:2c}
\end{eqnarray}
and
\begin{eqnarray}
&&\frac{\phi_2^{n+1}}{\delta t}=-\mathcal{G}\mu_2^{n+1},\label{sav:1d}\\
&&\mu_2^{n+1}=\mathcal{L}\phi_2^{n+1}+F'(\phi^n).\label{sav:2d}
\end{eqnarray}
Once  $(\phi_1,\phi_2)$ are solved, we plug $\phi^{n+1}=\phi_1^{n+1}+\xi^{n+1}\phi_2^{n+1}$ and $r^{n+1}=\xi^{n+1}G(\int_{\Omega}F(\phi^n)d\bx)$ into equation \eqref{scheme:gsav:3} to obtain $\xi^{n+1}$ by solving a nonlinear algebraic equation where Newton iterator solver with initial guess $(\xi^{n+1})^0=1$ should be implemented.
Since $\xi^{n+1}$ is a first-order approximation for $1$.  Finally solution $\phi^{n+1}$ can be updated by equation \eqref{sol}.

In summary,  we can then determine solution $\phi^{n+1}$ for scheme \eqref{scheme:gsav:1}-\eqref{scheme:gsav:3}  as follows:
 \begin{itemize}
\item  Solve  linear constant coefficient equations  \eqref{sav:1c} and  \eqref{sav:2c}  to obtain $\phi_{1}^{n+1}$,  equations  \eqref{sav:1d}   and  \eqref{sav:2d} to obtain $\phi_{2}^{n+1}$;
\item  Solve $\xi^{n+1}$  from  equations \eqref{scheme:gsav:3}  by plugging $\phi^{n+1}$ and $r^{n+1}$ into  it;
\item  Update $\phi^{n+1}$ from \eqref{sol}. 
\end{itemize}

\subsubsection{Square root SAV approach and Original SAV approach}
For example,  if we adopt  invertible function $G=x^{\frac 12}$, then $(G^{-1})=x^2$ and $(G^{-1})'(r) r_t=2rr_t$, then equations \eqref{gsav:1}-\eqref{gsav:3} are reformulated as 
 \begin{eqnarray}
&&\partial_t\phi=-\mathcal{G}\mu,\label{osav:1}\\
&&\mu=\mathcal{L}\phi+\frac{r}{G(\int_{\Omega}F(\phi)d\bx)}F'(\phi)\label{osav:2},\\
&&\frac{r}{G(\int_{\Omega}F(\phi)d\bx)}(F'(\phi),\phi_t)=2rr_t=\frac{d}{dt}r^2 .\label{osav:3}
\end{eqnarray}

Then, a first-order SAV scheme for the above system is
\begin{eqnarray}
&&\frac{\phi^{n+1}-\phi^n}{\delta t}=-\mathcal{G}\mu^{n+1},\label{linear:osav:1b}\\
&&\mu^{n+1}=\mathcal{L}\phi^{n+1}+\frac{r^{n+1}}{\sqrt{\int_{\Omega}F(\phi^n)d\bx+C_0}}F'(\phi^n),\label{linear:osav:2b}\\
&&\frac{(r^2)^{n+1}-(r^2)^n}{\delta t}\nonumber\\&&=\frac{r^{n+1}}{\sqrt{\int_{\Omega}F(\phi^n)d\bx+C_0}}({F'(\phi^n)},\frac{\phi^{n+1}-\phi^n}{\delta t}).\label{linear:osav:3b}
\end{eqnarray}
Taking the inner products of \eqref{linear:osav:1b} with $\mu^{n+1}$ and of \eqref{linear:osav:2b} with $-\frac{\phi^{n+1}-\phi^n}{\delta t}$,  summing up the results and taking into account \eqref{linear:osav:3b},  we have the following:
\begin{thm}
 The  scheme \eqref{linear:osav:1b}-\eqref{linear:osav:3b} is unconditionally energy stable in the sense that
 $$\tilde E(\phi^{n+1})-\tilde E(\phi^{n}) \le -\Delta t(\mathcal{G}\mu^{n+1},\mu^{n+1}),$$
 where $\tilde E(\phi^k)=\int_\Omega \frac12\mathcal{L}\phi^k\cdot \phi^k d\bx +(r^k)^2$.
\end{thm}
\begin{remark}\label{sav:re}
 If we can cancel $r$ on both sides from equation \eqref{osav:3} and obtain
\begin{eqnarray}
&&\partial_t\phi=-\mathcal{G}\mu,\label{sav:1}\\
&&\mu=\mathcal{L}\phi+\frac{r}{G(\int_{\Omega}F(\phi)d\bx)}F'(\phi)\label{sav:2},\\
&&r_t=\frac{1}{2}\frac{1}{G(\int_{\Omega}F(\phi)d\bx)}(F'(\phi),\phi_t). \label{sav:3}
\end{eqnarray}
Notice that $G=x^{\frac 12}$,  then it is observed that  the system \eqref{sav:1}-\eqref{sav:3} is exactly  the original SAV approach in \cite{SXY19,shen2018scalar}. By treating $r$ implicitly, a large class of linear numerical schemes have been  proposed in \cite{SXY19,shen2018scalar,cheng2019highly,jing2019second}.
\end{remark}
\subsubsection{The mapped $\tanh$ SAV approach}
In order to  make  $r=\tanh(\int_{\Omega}F(\phi)d\bx)$ not too close to $\pm 1$ numerically since $\tanh^{-1}(r=\pm 1 ) \rightarrow \infty$.  Then  we consider a mapped function  $G=\tanh(\frac{x}{C})$, then $G^{-1}=C\tanh^{-1}(x)$, and $C$ is a large positive constant.  We define a new variable $r=\tanh(\frac{\int_{\Omega}F(\phi)d\bx}{C})$ ,   then equations \eqref{gsav:1}-\eqref{gsav:3} are reformulated as 
 \begin{eqnarray}
&&\partial_t\phi=-\mathcal{G}\mu,\label{tanh:sav:1}\\
&&\mu=\mathcal{L}\phi+\frac{r}{\tanh(\frac{\int_{\Omega}F(\phi)d\bx}{C})}F'(\phi)\label{tanh:sav:2},\\
&&\frac{d}{dt}(C\tanh^{-1}(r))=\frac{r}{\tanh(\frac{\int_{\Omega}F(\phi)d\bx}{C})}(F'(\phi),\phi_t).\label{tanh:sav:3}
\end{eqnarray}

Then, a first-order SAV scheme for the above system is
\begin{eqnarray}
&&\frac{\phi^{n+1}-\phi^n}{\delta t}=-\mathcal{G}\mu^{n+1},\label{linear:tansav:1b}\\
&&\mu^{n+1}=\mathcal{L}\phi^{n+1}+\frac{r^{n+1}}{\tanh(\frac{\int_{\Omega}F(\phi^n)d\bx}{C})}F'(\phi^n),\label{linear:tansav:2b}\\
&&\frac{(C\tanh^{-1}(r))^{n+1}-(C\tanh^{-1}(r))^n}{\delta t}\nonumber\\&&=\frac{r^{n+1}}{\tanh(\frac{\int_{\Omega}F(\phi^n)d\bx}{C})}({F'(\phi^n)},\frac{\phi^{n+1}-\phi^n}{\delta t}).\label{linear:tansav:3b}
\end{eqnarray} 
\begin{remark}
The $\tanh$-SAV scheme \eqref{linear:tansav:1b}-\eqref{linear:tansav:3b} takes a  big advantage of  dealing with nonlinear potential $F(\phi^n)$ which is  not bounded up and below or the singular potential.  For example  the logarithmic (singular) functions: $F=\alpha \phi +\beta\ln(\frac{1+\phi}{1-\phi})$, $F=\phi+(1-\phi)\log(1-\phi)$ \cite{israel2012long,cherfils2011cahn} and  MBE model without slope selection \cite{cheng2019highly}. Because the domain of function  $\tanh$ is $(-\infty,\infty)$ while the range  is $(-1,1)$.

\end{remark}

Taking the inner products of \eqref{linear:tansav:1b} with $\mu^{n+1}$ and of \eqref{linear:tansav:2b} with $-\frac{\phi^{n+1}-\phi^n}{\delta t}$,  summing up the results and taking into account \eqref{linear:tansav:3b},  we have the following:
\begin{thm}
 The  scheme \eqref{linear:tansav:1b}-\eqref{linear:tansav:3b} is unconditionally energy stable in the sense that
 $$\tilde E(\phi^{n+1})-\tilde E(\phi^{n}) \le -\Delta t(\mathcal{G}\mu^{n+1},\mu^{n+1}),$$
 where $\tilde E(\phi^k)=\int_\Omega \frac12\mathcal{L}\phi^k\cdot \phi^k d\bx +(\tanh^{-1}(r))^k$.
\end{thm}

\subsection{Schemes  based on the MG-SAV approach with multiple  nonlinear potentials}

If we consider gradient flows with multiple nonlinear potentials,  where the total free energy is 
\begin{equation}\label{m:energy}
 E(\phi)=\int_\Omega \frac12\mathcal{L}\phi\cdot \phi +\sum\limits_{i=1}^mF_i(\phi) d\bx,
\end{equation}
then the multiple generalized scalar auxiliary variable approach (MG-SAV) may achieve better accuracy for numerical simulation  which has been observed in  \cite{cheng2018multiple}.  The gradient flow with multiple nonlinear potentials is formulated as the following form
\begin{equation}\label{m:grad:flow}
\begin{split}
&\phi_t=-\mathcal{G}\mu,\\
&\mu=\mathcal{L}\phi+\sum\limits_{i=1}^mF'_i(\phi), 
\end{split}
\end{equation}
where $F'_i(\phi)$ is the variational derivative of $F_i$.  By using the MG-SAV approach, 
we can also rewrite the energy \eqref{m:energy} as 
\begin{equation}\label{m:mod:energy}
 E(\phi)=\int_\Omega \frac12\mathcal{L}\phi\cdot \phi d\bx +\sum\limits_{i=1}^m G_i^{-1}\{G_i(\int_{\Omega}F_i(\phi)d\bx)\},
\end{equation}
where $G_i$ are various invertible functions for $i=1,2,\cdots,m$.
Setting new variables to be 
\begin{equation}
r_i=G_i(\int_{\Omega}F_i(\phi)d\bx).
\end{equation}
Now we shall derive an equivalent form 
\begin{eqnarray}
&&\partial_t\phi=-\mathcal{G}\mu,\label{mgsav:1}\\
&&\mu=\mathcal{L}\phi+\sum\limits_{i=1}^m\frac{r_i}{G_i(\int_{\Omega}F_i(\phi)d\bx)}F_i'(\phi)\label{mgsav:2},\\
&&\frac{d}{dt}G_i^{-1}(r_i) =\frac{r_i}{G_i(\int_{\Omega}F_i(\phi)d\bx)}(F_i'(\phi),\phi_t).\label{mgsav:3}
\end{eqnarray}
Then a second-order  BDF2  scheme can be constructed for equations \eqref{mgsav:1}-\eqref{mgsav:3}
\begin{eqnarray}
&&\frac{3\phi^{n+1}-4\phi^n+\phi^{n-1}}{2\delta t}=-\mathcal{G}\mu^{n+1},\label{scheme:mgsav:1}\\
&&\mu^{n+1}=\mathcal{L}\phi^{n+1}+\sum\limits_{i=1}^m\frac{r_i^{n+1}}{G_i(\int_{\Omega}F_i(\phi^{\dagger,n})d\bx)}F_i'(\phi^{\dagger,n})\label{scheme:mgsav:2},\\
&&\frac{3G_i^{-1}(r_i^{n+1}) -4G_i^{-1}(r_i^{n})+G_i^{-1}(r_i^{n-1})}{2\delta t} \nonumber\\&&\hskip 1cm =\frac{r_i^{n+1}}{G_i(\int_{\Omega}F_i(\phi^{\dagger,n})d\bx)}(F_i'(\phi^{\dagger,n}),\frac{3\phi^{n+1}-4\phi^n+\phi^{n-1}}{2\delta t}).\label{scheme:mgsav:3}
\end{eqnarray}
Where $G_i$ with  $i=1,2,\cdots ,m$ are invertible functions and $g^{\dagger,n}= 2g^n-g^{n-1}$  for any sequence $\{g^n\}$. 

Taking inner product of equation \eqref{scheme:mgsav:1} with $\mu^{n+1}$, of  equation \eqref{scheme:mgsav:2} with $\frac{3\phi^{n+1}-4\phi^n+\phi^{n-1}}{\delta t}$ and combining with equations \eqref{scheme:mgsav:3},  we derive the following energy dissipative law for scheme \eqref{scheme:mgsav:1}-\eqref{scheme:mgsav:3}.

\begin{thm}
 The  scheme \eqref{scheme:mgsav:1}-\eqref{scheme:mgsav:3} is unconditionally energy stable in the sense that
 $$\tilde E(\phi^{n+1})-\tilde E(\phi^{n}) \le -\delta t(\mathcal{G}\mu^{n+1},\mu^{n+1}),$$
 where $\tilde E(\phi^k)=\int_\Omega \frac14(\mathcal{L}\phi^k\cdot \phi^k +\mathcal{L}(2\phi^k-\phi^{k-1})\cdot (2\phi^k-\phi^{k-1}))d\bx+ \sum\limits_{i=1}^m (\frac 32G_i^{-1}(r_i^k)-\frac 12G_i^{-1}(r_i^{k-1}))$.
\end{thm}
Below we show how to solve the scheme \eqref{scheme:mgsav:1}-\eqref{scheme:mgsav:3}.
Define new variables 
\begin{equation}\label{m:xi}
\xi_i^{n+1}=\frac{r_i^{n+1}}{G_i(\int_{\Omega}F_i(\phi^{\dagger,n})d\bx)},
\end{equation} 
for $i=1,2,\cdots,m$, then $\phi^{n+1}$ and $\mu^{n+1}$ can be expressed as 
\begin{equation}\label{m:sol}
\phi^{n+1}=\phi_1^{n+1}+\sum\limits_{i=1}^m\xi_i^{n+1}\phi_{i,2}^{n+1},
\end{equation}
 and  
\begin{equation}
\mu^{n+1}=\mu_1^{n+1}+\sum\limits_{i=1}^m\xi_i^{n+1}\mu_{i,2}^{n+1}. 
\end{equation}  
Where  $(\phi_1^{n+1},\mu_1^{n+1})$ and $(\phi_{i,2}^{n+1},\mu_{i,2}^{n+1})$  can be determined as follows:
\begin{eqnarray}
&&\frac{3\phi_1^{n+1}-4\phi^n+\phi^{n-1}}{2\delta t}=-\mathcal{G}\mu_1^{n+1},\label{m:solution:1}\\
&&\mu_1^{n+1}=\mathcal{L}\phi_1^{n+1},\label{m:solution:2}
\end{eqnarray}
and 
\begin{eqnarray}
&&\frac{3\phi_{i,2}^{n+1}-4\phi^n+\phi^{n-1}}{2\delta t}=-\mathcal{G}\mu_{i,2}^{n+1},\label{m:solution:3}\\
&&\mu_{i,2}^{n+1}=\mathcal{L}\phi_{i,2}^{n+1} +  F_i'(\phi^{\dagger,n})\label{m:solution:4}. 
\end{eqnarray}
Once $(\phi_1^{n+1},\mu_1^{n+1})$ and $(\phi_{i,2}^{n+1},\mu_{i,2}^{n+1})$  are known, we plug equations \eqref{m:xi} and \eqref{m:sol} into equations \eqref{scheme:mgsav:3} to obtain a coupled  $m\times m$ nonlinear algebraic system of  $\xi^{n+1}_i$ for $i=1,2,\cdots,m$. Finally, we update $\phi^{n+1}$ from \eqref{m:sol}. 

In summary,  we can then determine solution $\phi^{n+1}$ for scheme \eqref{scheme:mgsav:1}-\eqref{scheme:mgsav:3}  as follows:
 \begin{itemize}
\item  Solve  linear constant coefficient equation  from equation  \eqref{m:solution:1} and equation \eqref{m:solution:2}  to obtain $\phi_{1}^{n+1}$,  equation  \eqref{m:solution:3}   and equation \eqref{m:solution:4} to obtain $\phi_{i,2}^{n+1}$;
\item  Solve $\xi_i^{n+1}$ for $i=1,2,\cdots,m$ from  equation \eqref{scheme:mgsav:3}  by plugging equations \eqref{m:xi} and \eqref{m:sol} into  it;
\item  Update $\phi^{n+1}$ from \eqref{m:sol}. 
\end{itemize}

\section{Schemes  based on the G-SAV approach with multiple components}
In this section, we consider multiple-components phase field models  which  play an important role in describing  three material  components \cite{boyer2011numerical,yang2017numerical}.  We introduce $\phi_i$ $(i=1,2,\cdots,m)$ be the $i-th$ phase variables . 
A special form of  total free energy for multi-phase system  is formulated as
\begin{equation}\label{En:three}
E(\phi_1,\phi_2,\cdots,\phi_m)=\sum\limits_{i,j=1}^{m}\int_{\Omega}\mathcal{L}\phi_i\cdot \phi_jd\bx +\int_{\Omega}F(\phi_1,\phi_2,\cdots,\phi_m) d\bx,
\end{equation}
where $\mathcal{L}$ is a self-adjoint nonnegative linear operator, and the constant matrix $d_{ij}$ is symmetric positive definite.   $F(\phi_1,\phi_2,\cdots,\phi_m)$ is nonlinear potential. 
We derive the multi-components  phase field model  by taking  variational derivative  with respect to  \eqref{En:three},
\begin{equation}\label{three:CH}
\begin{split}
&\partial_t\phi_i=M\mathcal{G}\mu_i, \quad i=1,2,\cdots,m\\
&\mu_i=\mathcal{L}\phi_i +f_i,
\end{split}
\end{equation}
with some suitable boundary condition,  $M$ is mobility constant.  
\begin{thm}
The multi-components phase field  equation \eqref{three:CH} satisfies the following energy dissipative law:
\begin{equation}
\frac{d}{dt}E(\phi_1,\phi_2,\cdots,\phi_m)=-M\sum\limits_{i=1}^m(\mathcal{G}\mu_i,\mu_i).
\end{equation}
\end{thm}
\begin{proof}
Taking inner product of  equation \eqref{three:CH} with $\mu_i$ and $\partial_t\phi_i$ respectively for $i=1,2,\cdots,m$,  taking integration by parts and  summing up these three equalities.  The desired energy dissipative law is obtained.
\end{proof}
We rewrite energy \eqref{En:three} to the following form
\begin{equation}\label{mod:En:three}
E(\phi_1,\phi_2,\cdots,\phi_m)=\sum\limits_{i=1}^{m}\int_{\Omega}\mathcal{L}\phi_i\cdot \phi_id\bx  +G^{-1}\{G(\int_{\Omega}F(\phi_1,\phi_2,\cdots,\phi_m) d\bx)\}.
\end{equation}
Introducing a new SAV 
\begin{equation}
r(t)=G(\int_{\Omega}F(\phi_1,\phi_2,\cdots,\phi_m) d\bx),
\end{equation}
then the system \eqref{three:CH} is reformulated to be
\begin{equation}\label{mod:three:CH}
\begin{split}
&\partial_t\phi_i=M\mathcal{G}\mu_i, \quad i=1,2,\cdots,m\\
&\mu_i=\mathcal{L}\phi_i +\frac{r}{G(\int_{\Omega}F(\phi_1,\phi_2,\cdots,\phi_m)d\bx)}f_i,\\
&\frac{d}{dt}G^{-1}(r) =\frac{r}{G(\int_{\Omega}F(\phi_1,\phi_2,\cdots,\phi_m)d\bx)}\sum\limits_{i=1}^3(f_i,\partial_t\phi_i).
\end{split}
\end{equation} 
Then an iterative  Crank-Nicolson scheme for system \eqref{mod:three:CH}  is
\begin{equation}\label{scheme:mod:three:CH}
\begin{split}
&\frac{\phi_i^{n+1}-\phi_i^n}{\delta t}=M\mathcal{G}\mu_i^{n+\frac 12}, \quad i=1,2,3,\\
&\mu^{n+\frac 12}_i=\mathcal{L}\phi_i^{n+\frac 12} +\frac{r^{n+\frac 12}}{G(\int_{\Omega}F(\phi_1^{\star,n},\phi_2^{\star,n},\cdots,\phi_m^{\star,n}) d\bx)}f_i^{\star,n},\\
&\frac{(G^{-1}(r))^{n+1}-(G^{-1}(r))^n}{\delta t}=\frac{r^{n+\frac 12}}{G(\int_{\Omega}F(\phi_1^{\star,n},\phi_2^{\star,n},\phi_3^{\star,n})d\bx)}\sum\limits_{i=1}^3(f_i^{\star,n},\frac{\phi_i^{n+1}-\phi_i^n}{\delta t}),
\end{split}
\end{equation}
where $f_i^{*,n}=\frac 12(3f_i^n-f_i^{n-1})$ and $\phi^{*,n}=\frac 12(3\phi^n-\phi^{n-1})$.
It is easy to see  that a no-iterative Crank-Nicolson scheme can be developed similarly by treating $r$ explicitly. And it can also be solved efficiently as  one component case. 

Taking inner product of equation \eqref{scheme:mod:three:CH} with $\mu_i^{n+\frac 12}$,  $\frac{\phi_i^{n+1}-\phi_i^n}{\delta t}$ for $i=1,2,3$,  summing up these three equalities and combining the third equation of  \eqref{scheme:mod:three:CH},  we derive the following energy dissipative law for scheme \eqref{scheme:mod:three:CH}.

\begin{thm}
 The  scheme \eqref{scheme:mod:three:CH}- is unconditionally energy stable in the sense that
 $$\tilde E(\phi_1^{n+1},\phi_2^{n+1},\cdots,\phi_m^{n+1})-\tilde E(\phi_1^{n},\phi_2^{n},\cdots,\phi_m^{n})) \le -M\sum\limits_{i=1}^m(\mathcal{G}\mu_i^{n+\frac 12},\mu_i^{n+\frac 12}),$$
 where $\tilde E(\phi^k)=\int_\Omega \sum\limits_{i=1}^m\frac12\mathcal{L}\phi_i^k\cdot \phi_i^k d\bx+  G^{-1}(r^k)$.
\end{thm}
The numerical scheme \eqref{scheme:mod:three:CH} can also be efficiently solved. Setting
$$\xi^{n+\frac 12}=\frac{r^{n+\frac 12}}{G(\int_{\Omega}F(\phi_1^{\star,n},\phi_2^{\star,n},\cdots,\phi_m^{\star,n}) d\bx)}.$$ Obviously $\xi^{n+\frac 12}$ is a second-order approximation for  $1$.
writing 
\begin{equation}\label{sol:three}
\phi_i^{n+1}=\phi_{i,1}^{n+1}+\xi^{n+\frac 12}\phi_{i,2}^{n+1},\; \mu_i^{n+1}=\mu_{i,1}^{n+1}+\xi^{n+\frac 12}\mu_{i,2}^{n+1},\; 
 \end{equation}
in the above, we find that $(\phi_i^{n+1},\mu_i^{n+1})\;(i=1,2,\cdots,m)$ can be determined  as follows:
\begin{eqnarray}
&&\frac{\phi_{i,1}^{n+1}-\phi_i^n}{\delta t}=M\mathcal{G}\mu^{n+\frac 12}_{i,1},\label{three:sav:1c}\\
&&\mu_{i,1}^{n+\frac 12}=\mathcal{L}\phi^{n+\frac 12}_{i,1}, \label{three:sav:2c}
\end{eqnarray}
and
\begin{eqnarray}
&&\frac{\phi_{i,2}^{n+1}}{\delta t}=M\mathcal{G}\mu^{n+\frac 12}_{i,2},\label{three:sav:1d}\\
&&\mu_{i,2}^{n+\frac 12}=\mathcal{L}\phi^{n+\frac 12}_{i,2} +f_i^{\star,n}.\label{three:sav:2d}
\end{eqnarray}
Once  $(\phi_{i,1}^{n+1},\phi_{i,2}^{n+1})$ are solved, we plug $\phi_i^{n+1}=\phi_{i,1}^{n+1}+\xi^{n+\frac 12}\phi_{i,2}^{n+1}$  and  
\begin{equation}\label{rn}
r^{n+\frac 12}=\xi^{n+\frac 12}G(\int_{\Omega}F(\phi_1^{\star,n},\phi_2^{\star,n},\cdots,\phi_m^{\star,n}) d\bx),
\end{equation}
 into the third  equation of \eqref{scheme:mod:three:CH}  to obtain $\xi^{n+\frac 12}$ by solving a nonlinear algebraic equation where Newton iterator solver should be implemented with initial guess $(\xi^{n+\frac 12})^0=1$.  Finally solution $\phi^{n+1}$ can be updated by equation \eqref{sol:three}.
 
In summary,  we can then determine solution $\phi_i^{n+1}$ for $i=1,2,\cdots,m$  as follows:
 \begin{itemize}
\item  Solve  linear constant coefficient equation  from equations  \eqref{three:sav:1c}-\eqref{three:sav:2c} to obtain $\phi_{i,1}^{n+1}$ and equations  \eqref{three:sav:1d}-\eqref{three:sav:2d} to obtain $\phi_{i,2}^{n+1}$;
\item  Solve $\xi^{n+\frac 12}$ from the third equation \eqref{scheme:mod:three:CH} by plugging equation \eqref{rn} into it;
\item  Update $\phi_i^{n+1}=\phi_{i,1}^{n+1}+\xi^{n+\frac 12}\phi_{i,2}^{n+1}$  for $i=1,2,\cdots,m$.
\end{itemize}
Hence, the above scheme can be  implemented very efficiently.

\section{The second approach}
We can treat scalar variable $r$ explicitly in equation \eqref{gsav:2}, and derive 
an  no-iterative,   first-order  numerical scheme of G-SAV for system \eqref{gsav:1}-\eqref{gsav:3}  is 
\begin{eqnarray}
&&\frac{\phi^{n+1}-\phi^n}{\delta t}=-\mathcal{G}\mu^{n+1},,\label{scheme2:gsav:1}\\
&&\mu^{n+1}=\mathcal{L}\phi^{n+1}+\frac{r^{n}}{G(\int_{\Omega}F(\phi^n)d\bx)}F'(\phi^n)\label{scheme2:gsav:2},\\
&&\frac{(G^{-1}(r))^{n+1}-(G^{-1}(r))^n}{\delta t}=\frac{r^{n}}{G(\int_{\Omega}F(\phi^n)d\bx)}(F'(\phi^n),\frac{\phi^{n+1}-\phi^n}{\delta t}).\label{scheme2:gsav:3}
\end{eqnarray}
Taking the inner products of \eqref{scheme2:gsav:1} with $\mu^{n+1}$ and of \eqref{scheme2:gsav:2} with $-\frac{\phi^{n+1}-\phi^n}{\delta t}$,  summing up the results and taking into account \eqref{scheme2:gsav:3},  we have the following:
\begin{thm}
 The  scheme \eqref{scheme2:gsav:1}-\eqref{scheme2:gsav:3} is unconditionally energy stable in the sense that
 $$\tilde E(\phi^{n+1})-\tilde E(\phi^{n}) \le -\Delta t(\mathcal{G}\mu^{n+1},\mu^{n+1}),$$
 where $\tilde E(\phi^k)=\int_\Omega \frac12\mathcal{L}\phi^k\cdot \phi^k d\bx +(G^{-1}(r))^k$.
\end{thm}

We now show that the above G-SAV scheme \eqref{scheme2:gsav:1}-\eqref{scheme2:gsav:3} can be efficiently implemented.
Writing equation \eqref{scheme2:gsav:1} as 
\begin{equation}
\frac{\phi^{n+1}}{\delta t}+\mathcal{G}\mathcal{L}\phi^{n+1}=\frac{\phi^n}{\delta t}
-\mathcal{G}\{\frac{r^{n}}{G(\int_{\Omega}F(\phi^n)d\bx)}F'(\phi^n)\}.
\end{equation}
Defining Linear operator $\chi=\frac{1}{\delta t}+\mathcal{G}\mathcal{L}$, then $\phi^{n+1}$ can be solved by
\begin{equation}\label{step:1}
\phi^{n+1}=\chi^{-1}\{\frac{\phi^n}{\delta t}
-\mathcal{G}\{\frac{r^{n}}{G(\int_{\Omega}F(\phi^n)d\bx)}F'(\phi^n)\}\}.
\end{equation}
$r^{n+1}$ can be updated by plugging $\phi^{n+1}$ into equation \eqref{scheme2:gsav:3}
\begin{equation}\label{step:2}
(G^{-1}(r))^{n+1}=(G^{-1}(r))^n+\delta t\frac{r^n}{G(\int_{\Omega}F(\phi^n)d\bx)}(F'(\phi^n),\frac{\phi^{n+1}-\phi^n}{\delta t}).
\end{equation}
Finally  $r^{n+1}=G(G^{-1}(r))^{n+1}$.

In summary,  we can then determine solution $\phi^{n+1}$  as follows:
 \begin{itemize}
\item  Solve a linear constant coefficient equation  from \eqref{step:1};
\item  Update $(G^{-1}(r))^{n+1}$ from equation \eqref{step:2};
\item  Obtain $r^{n+1}=G(G^{-1}(r))^{n+1}$.
\end{itemize}
Hence, the above scheme can be  implemented very efficiently.  The second-order numerical schemes based on BDF2 or Crank-Nicolson can be constructed similarly. 

\begin{remark}
If we choose function $G=e^x$ as the exponential function, which is  the just  so-called E-SAV approach  developed in \cite{liu2019exponential}, some numerical simulations are also presented  in \cite{liu2019exponential} to validate the efficiency of this approach.
\end{remark}

\section{The third  approach (Lagrange multiplier approach)} 
Both the first and second approaches preserve modified energy dissipative law in discrete level. In this section,  we develop SAV approach which preserves original energy dissipative  law instead of  modified energy dissipative law.  Since the auxiliary variable is defined as  $r=G(\int_{\Omega}F(\phi)d\bx$, then we obtain  $\int_{\Omega}F(\phi)d\bx=G^{-1}(r)$.  The system \eqref{grad:flow}  can be reformulated  as 
\begin{eqnarray}
&&\partial_t\phi=-\mathcal{G}\mu,\label{ori:gsav:1}\\
&&\mu=\mathcal{L}\phi+\frac{r}{G(\int_{\Omega}F(\phi)d\bx)}F'(\phi)\label{ori:gsav:2},\\
&&\frac{d}{dt}\int_{\Omega}F(\phi)d\bx=\frac{d}{dt}G^{-1}(r)=\frac{r}{G(\int_{\Omega}F(\phi)d\bx)}(F'(\phi),\phi_t). \label{ori:gsav:3}
\end{eqnarray}
A second-order scheme for \eqref{ori:gsav:1}-\eqref{ori:gsav:3} is constructed as 
\begin{eqnarray}
&&\frac{3\phi^{n+1}-4\phi^n+\phi^{n-1}}{2\delta t}=-\mathcal{G}\mu^{n+1},\label{scheme:ori:gsav:1}\\
&&\mu^{n+1}=\mathcal{L}\phi^{n+1}+\frac{r^{n+1}}{G(\int_{\Omega}F(\phi^{\dagger,n})d\bx)}F'(\phi^{\dagger,n})\label{scheme:ori:gsav:2},\\
&&\frac{\int_{\Omega}3F(\phi^{n+1})-4F(\phi^n)+F(\phi^{n-1})d\bx}{2\delta t}\nonumber\\&& \hskip 1cm=\frac{r^{n+1}}{G(\int_{\Omega}F(\phi^{\dagger,n})d\bx)}(F'(\phi^{\dagger,n}),\frac{3\phi^{n+1}-4\phi^n+\phi^{n-1}}{2\delta t}). \label{scheme:ori:gsav:3}
\end{eqnarray}

Taking the inner products of \eqref{scheme:ori:gsav:1} with $\mu^{n+1}$ and of \eqref{scheme:ori:gsav:2} with $-\frac{3\phi^{n+1}-4\phi^n+\phi^{n-1}}{2\delta t}$,  summing up the results and taking into account \eqref{scheme:ori:gsav:3},  we have the following:
\begin{thm}\label{third:decay}
 The  scheme \eqref{scheme:ori:gsav:1}-\eqref{scheme:ori:gsav:3}  is unconditionally energy stable in the sense that
 $$\tilde E(\phi^{n+1})-\tilde E(\phi^{n}) \le -\Delta t(\mathcal{G}\mu^{n+1},\mu^{n+1}),$$
 where $\tilde E(\phi^k)=\int_\Omega \frac14(\mathcal{L}\phi^k\cdot \phi^k +\mathcal{L}(2\phi^k-\phi^{k-1})\cdot (2\phi^k-\phi^{k-1}))d\bx +\frac 32\int_{\Omega}F(\phi^k)d\bx-\frac 12\int_{\Omega}F(\phi^k)d\bx$.
\end{thm}

\begin{remark}
From Theorem \ref{third:decay} we find that scheme \eqref{scheme:ori:gsav:1}-\eqref{scheme:ori:gsav:3} satisfies original energy dissipative law. 
It is also  observed that the scheme \eqref{scheme:ori:gsav:1}-\eqref{scheme:ori:gsav:3} is exactly the same  with the Lagrange multiplier approach \cite{cheng2019new} by treating  $$\eta^{n+1}=\frac{r^{n+1}}{G(\int_{\Omega}F(\phi^{\dagger,n})d\bx)}.$$   Ample numerical simulations are shown in \cite{cheng2019new} to validate the efficiency of Lagrange multiplier approach. 
\end{remark}

\subsection{Stabilized-G-SAV approach and adaptive time stepping}
For problems with stiff nonlinear terms, one may have to use very small time steps to obtain accurate results with G-SAV  above. In order to allow larger time steps while achieving desired accuracy, we may add suitable stabilization and use  adaptive time stepping. 
\subsubsection{Stabilization}
Instead of solving \eqref{grad:flow}, we consider a perturbed  system with two additional stabilization terms
\begin{equation}\label{grad:flow2}
\begin{split}
&\phi_t=-\mathcal{G}\mu,\\
&\mu=\mathcal{L}\phi+\epsilon_1 \phi_{tt}+\epsilon_2 \mathcal{L}\phi_{tt}+F'(\phi),
\end{split}
\end{equation}
where $\epsilon_i,\;i=1,2$ are two small stabilization constants whose choices will depend on how stiff are the nonlinear terms. It is easy to see that the above system is a gradient flow with a perturbed free energy $E_\epsilon(\phi)=E(\phi)+\frac{\epsilon_1}2(\phi_t,\phi_t)+\frac{\epsilon_2}2(\mathcal{L}\phi_t,\phi_t)$ and satisfies the following energy law:
\begin{equation}
 \frac d{dt}E_\epsilon(\phi)=-(\mathcal{G}\mu,\mu).
\end{equation}
The schemes presented before for \eqref{grad:flow} can all be easily extended for  \eqref{grad:flow2} while keeping the same simplicity. For example, a second order scheme based on the second approach is:
\begin{eqnarray}
&&\frac{\phi^{n+1}-\phi^n}{\delta t}=-\mathcal{G}\mu^{n+1/2},\label{lsav:1c2}\\
&&\mu^{n+1/2}=\mathcal{L}\phi^{n+1/2}+ \frac{\epsilon_1}{(\delta t)^2}(\phi^{n+1}-2\phi^n+\phi^{n-1})
\nonumber\\&&\hskip 1cm +\frac{\epsilon_2}{(\delta t)^2}\mathcal{L}(\phi^{n+1}-2\phi^n+\phi^{n-1}) +\frac{r^{n+\frac 12}}{G(\int_{\Omega}F(\phi^{\star,n})d\bx)}F'(\phi^{\star,n}),\label{lsav:2c2}\\
&&\frac{G^{-1}(r^{n+1}) -G^{-1}(r^{n})}{\delta t} =\frac{r^{n+\frac 12}}{G(\int_{\Omega}F(\phi^{\star,n})d\bx)}(F'(\phi^{\star,n}),\frac{\phi^{n+1}-\phi^n}{\delta t})\label{lsav:3c2},
\end{eqnarray}
where $f^{n+1/2}=\frac 12(f^{n+1}+f^n)$ and $f^{*,n}=\frac 12(3f^n-f^{n-1})$  for any sequence $\{f^n\}$. 

Taking the inner products of \eqref{lsav:1c2} with $\mu^{n+1/2}$ and  of 
\eqref{lsav:2c2} with $-\frac{\phi^{n+1}-\phi^n}{\delta t}$, summing up the results along with \eqref{lsav:3c2} and dropping some unnecessary terms, we immediately derive the following results:
\begin{thm}
The scheme \eqref{lsav:1c2}-\eqref{lsav:3c2} is unconditionally energy stable in the sense that
$$E_\epsilon^{n+1}-E_\epsilon^n\le -\delta t(\mathcal{G}\mu^{n+1/2},\mu^{n+1/2}),$$
where $E_\epsilon^k=E(\phi^k)+\frac{\epsilon_1}{2}(\frac{\phi^k-\phi^{k-1}}{\delta t},\frac{\phi^k-\phi^{k-1}}{\delta t})+\frac{\epsilon_2}{2}(\mathcal{L}\frac{\phi^k-\phi^{k-1}}{\delta t},\frac{\phi^k-\phi^{k-1}}{\delta t}) $ with $E(\phi)$ being the original free energy defined in \eqref{orienergy}.
\end{thm}
It is clear that the above scheme can be efficiently implemented as the scheme 
\eqref{scheme:gsav:1}-\eqref{scheme:gsav:3}.

\subsubsection{Adaptive time stepping}
To improve the efficiency of G-SAV approach, one can combine them with
an adaptive time stepping method.   Many numerical examples have been provided  for the original SAV approaches \cite{cheng2018multiple,shen2018scalar,cheng2019new}.  Similarly,  we can also apply an adaptive time stepping strategy for G-SAV approach since all the schemes using G-SAV are energy stable.   For example,  we can  combine an  adaptive time stepping method with G-SAV scheme \eqref{lsav:1c2}-\eqref{lsav:3c2}  and  achieve  a second-order adaptive Cank-Nicolson scheme, see  \cite{cheng2018multiple}.  

\section{Numerical results} 
In this section some numerical experiments will be provided  to validate their stability and convergence rates for the first and second G-SAV  approaches, since enough numerical simulations are shown in \cite{cheng2019new} for the third approach.  We will compare the performance of  different G-SAV approaches by choosing various function $G$. In all numerical examples below, we assume  periodic boundary conditions and use a Fourier Spectral  method in space. The default  computational domain is $[-\pi,\pi)^d$ with $d=2$.

\subsection{Validation and comparison}
We consider famous  Allen-cahn \cite{allen1979microscopic} and  Cahn-Hilliard \cite{cahn1958free,cahn1959free} equations  and use $128$ modes in each direction in our Fourier Spectral method so that the spatial discretization errors are negligible compared with time discretization error. The total free energy of Allen-Cahn and Cahn-Hiliard equation is 
$$E_{tot}=\int_{\Omega}\frac 12|\Grad\phi|^2+F(\phi)d\bx,$$
where $F(\phi)=\frac{1}{4\eps^2}(\phi^2-1)^2$ is double well potential. The form of  chemical potential in \eqref{grad:flow} is 
$$\mu=-\Delta\phi+F'(\phi),$$  the $\mathcal{G}=I$ for Allen-Cahn equation and $\mathcal{G}=\Delta$ for Cahn-Hiliiard equation.

\subsubsection{\bf Comparison of  various G-SAV  approaches}
We first investigate the performance of  various  G-SAV  approaches proposed in Section 2.  We consider the 2D Cahn-Hilliard equation  and  choose  a  random initial condition
\begin{equation}\label{ini_rand}
\phi(x,y)=0.03+0.001\,{\rm rand}(x,y),
\end{equation}
where ${\rm rand}(x,y)$ represents random data between $[-1,1]^2$.

\begin{figure}[htbp]
\centering
\subfigure[$G=\tanh(\frac{x}{c}) $]{
\includegraphics[width=0.22\textwidth,clip==]{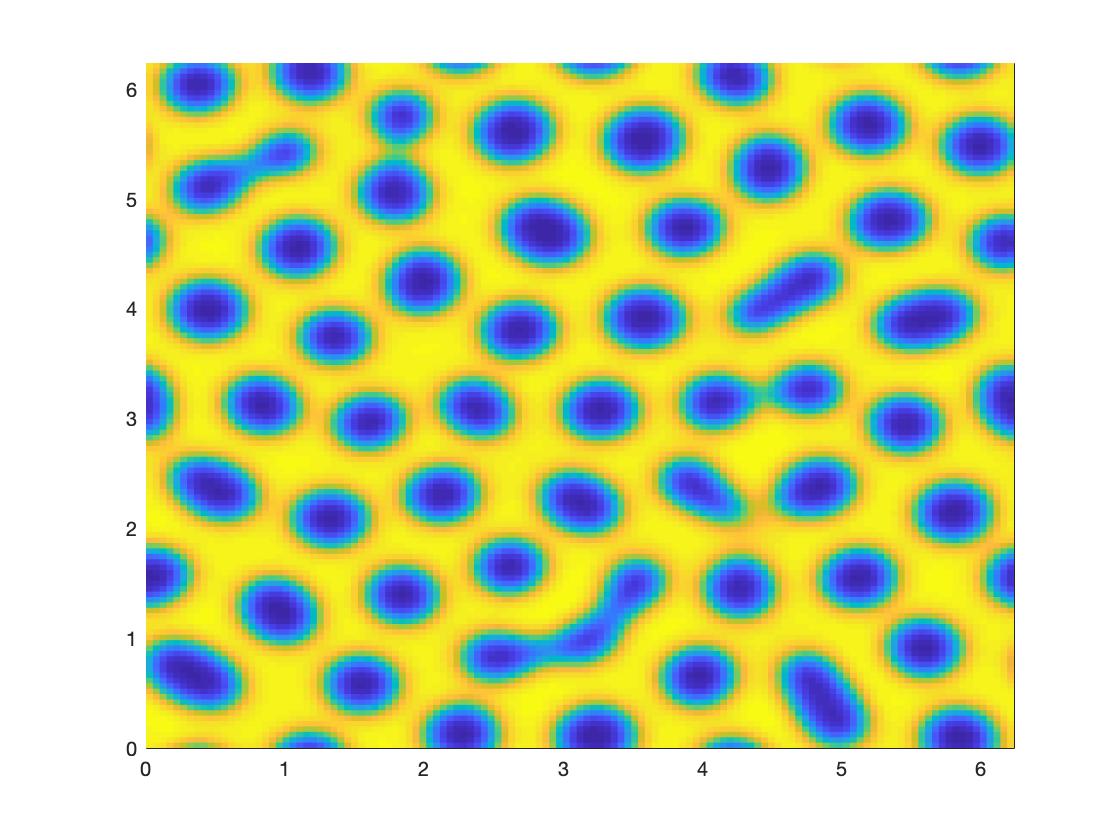}}
\subfigure[$G=\tanh(\frac{x}{c}) $ ]{
\includegraphics[width=0.22\textwidth,clip==]{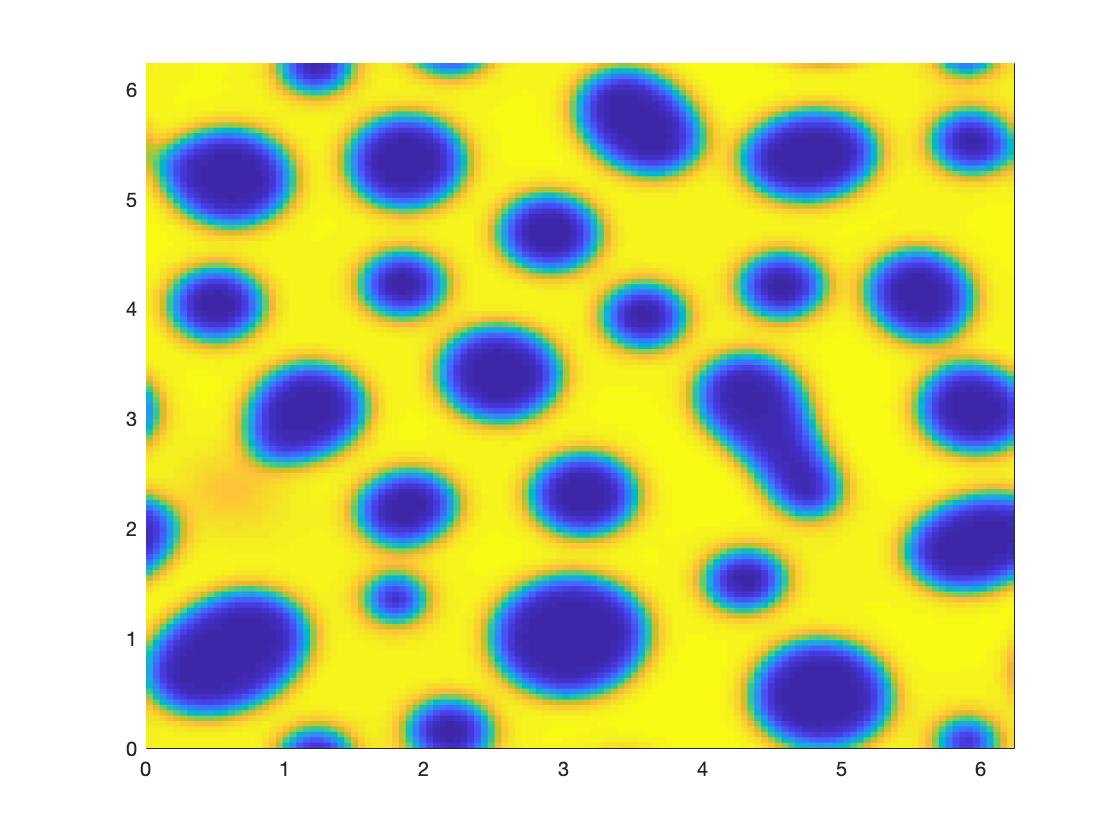}}
\subfigure[$G=\tanh(\frac{x}{c}) $ ]{
\includegraphics[width=0.22\textwidth,clip==]{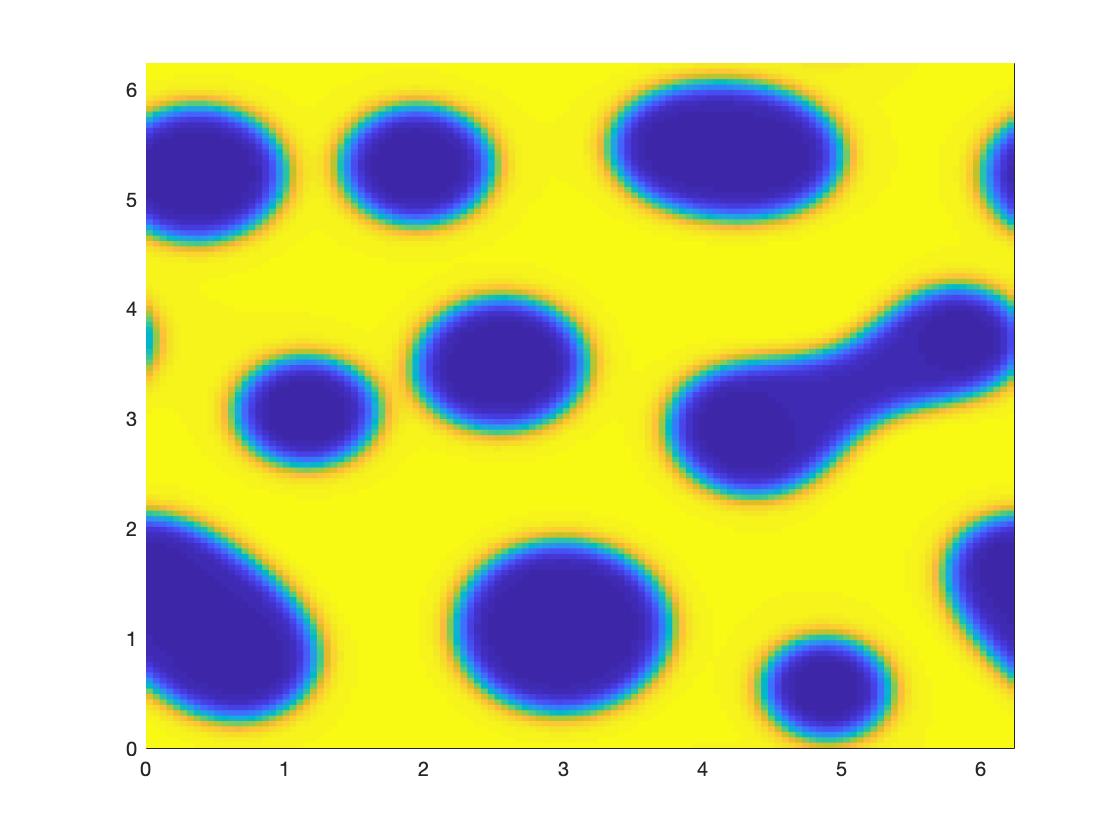}}
\subfigure[$G=\tanh(\frac{x}{c}) $ ]{
\includegraphics[width=0.22\textwidth,clip==]{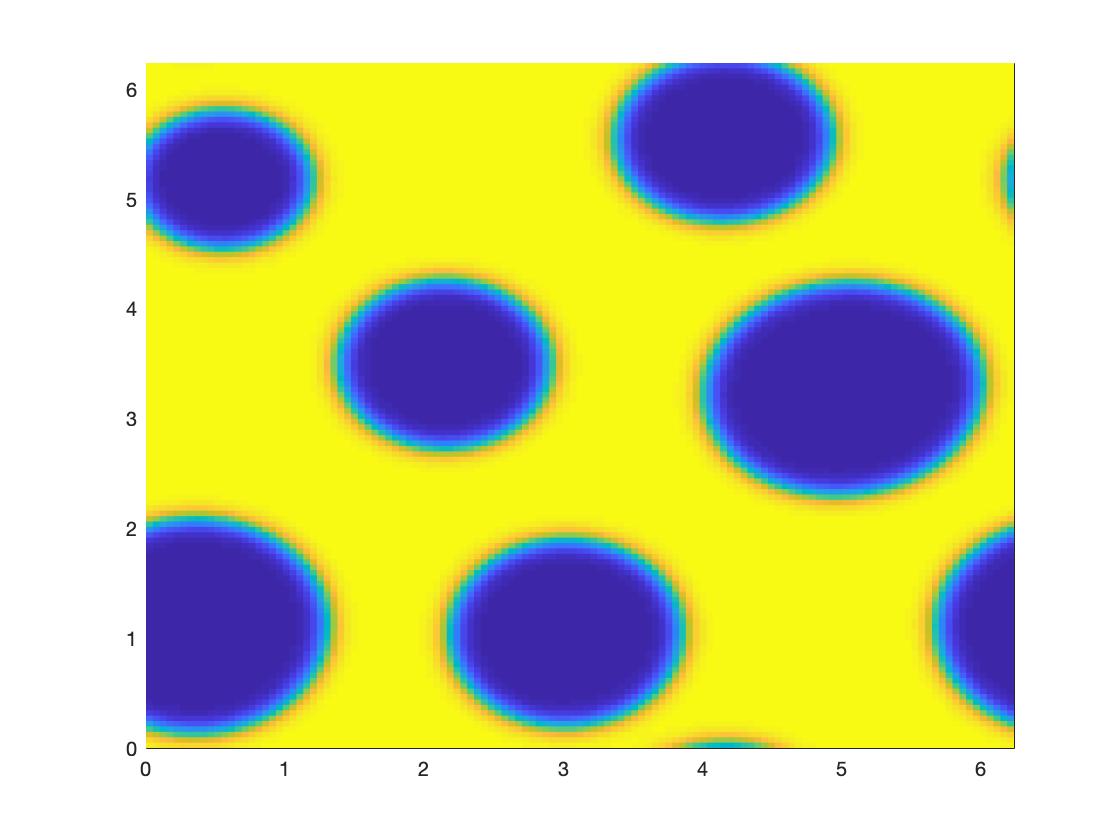}}
\subfigure[$G=x^{\frac 13}$]{
\includegraphics[width=0.22\textwidth,clip==]{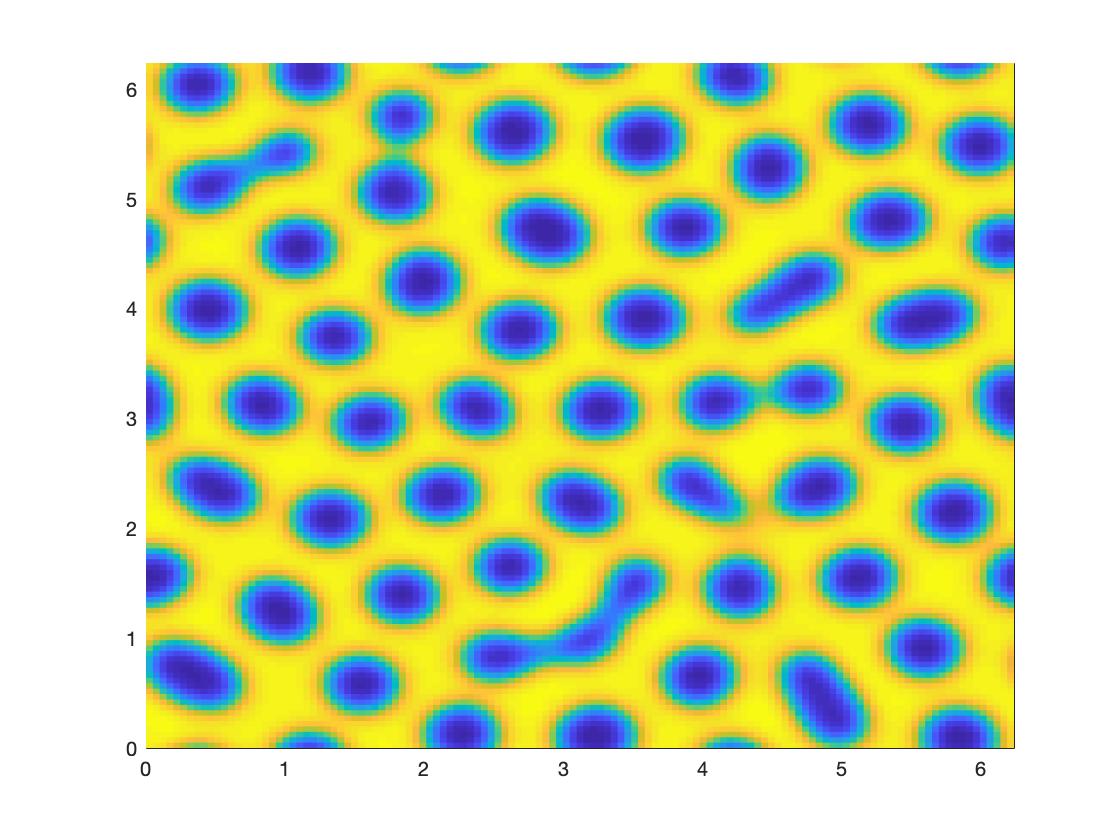}}
\subfigure[$G=x^{\frac 13} $ ]{
\includegraphics[width=0.22\textwidth,clip==]{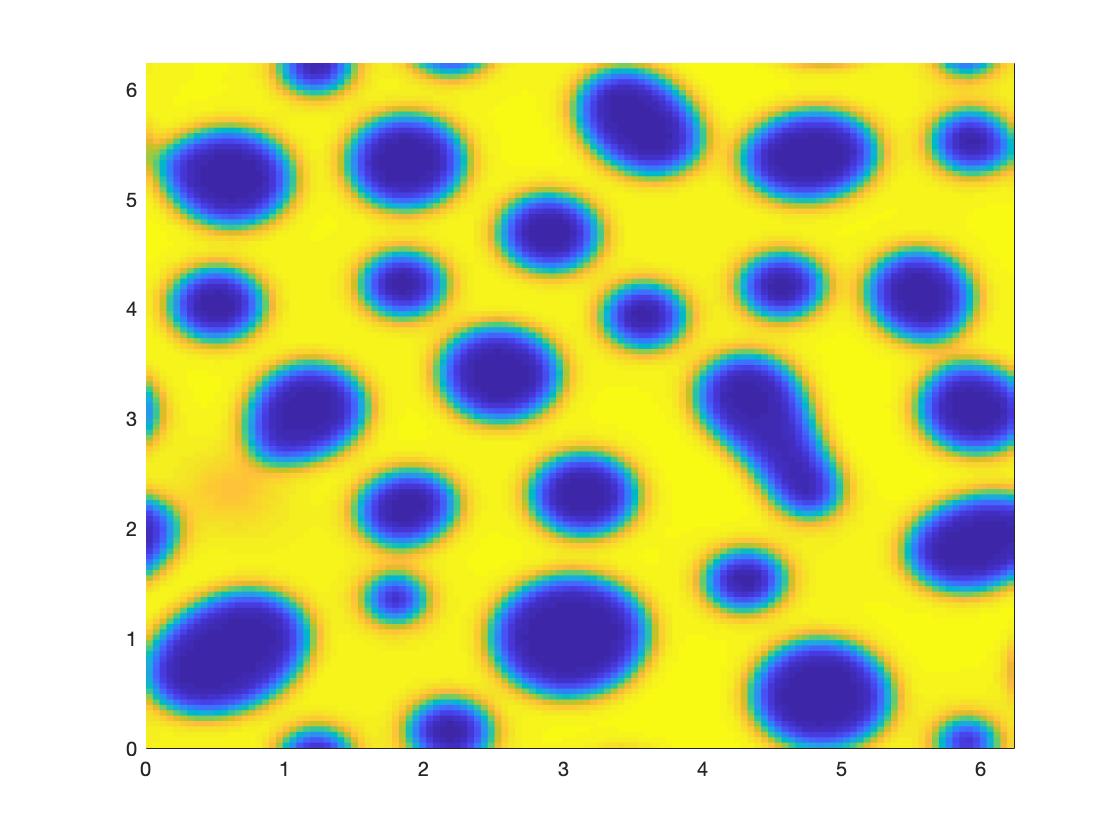}}
\subfigure[$G=x^{\frac 13} $ ]{
\includegraphics[width=0.22\textwidth,clip==]{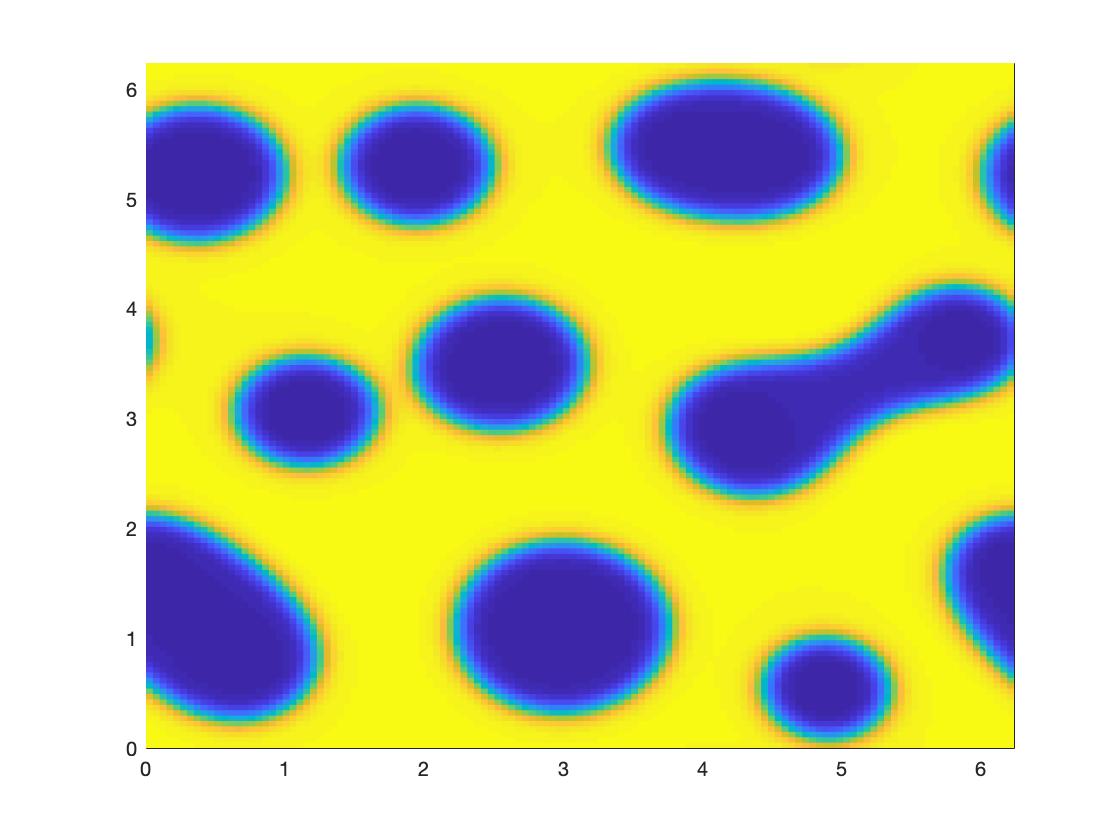}}
\subfigure[$G=x^{\frac 13} $ ]{
\includegraphics[width=0.22\textwidth,clip==]{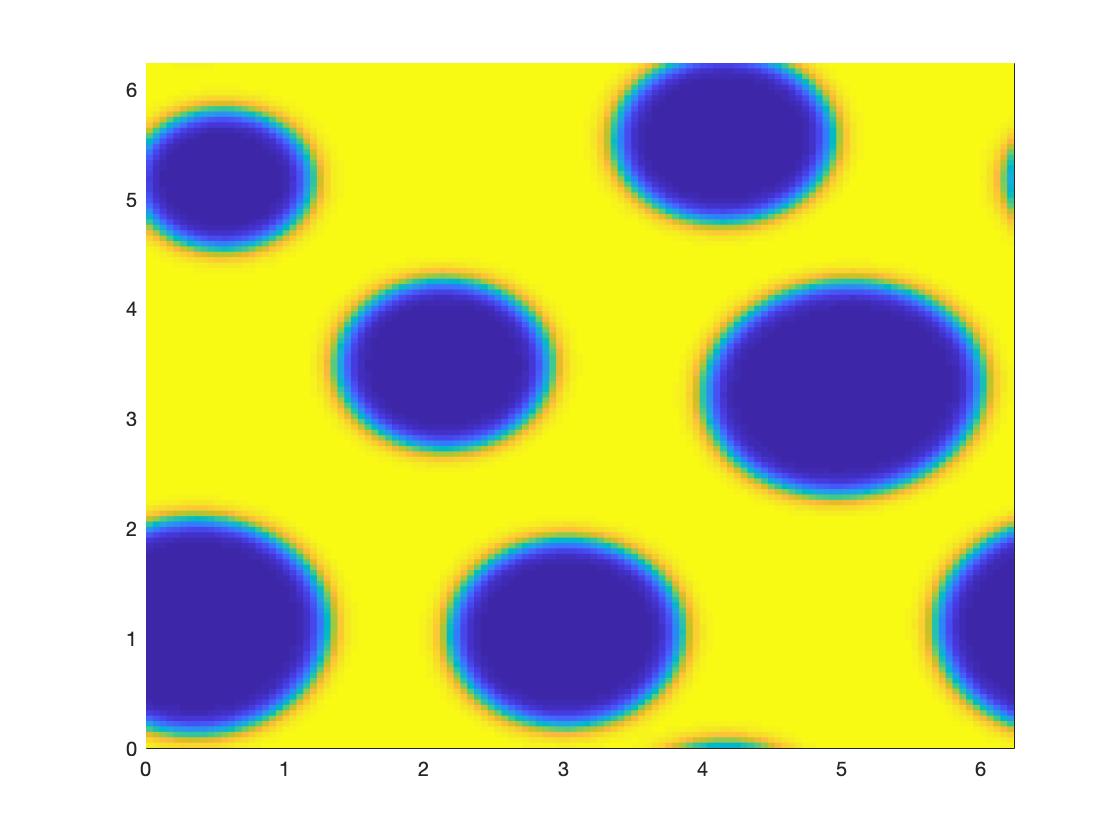}}
\subfigure[$G=x^3$]{
\includegraphics[width=0.22\textwidth,clip==]{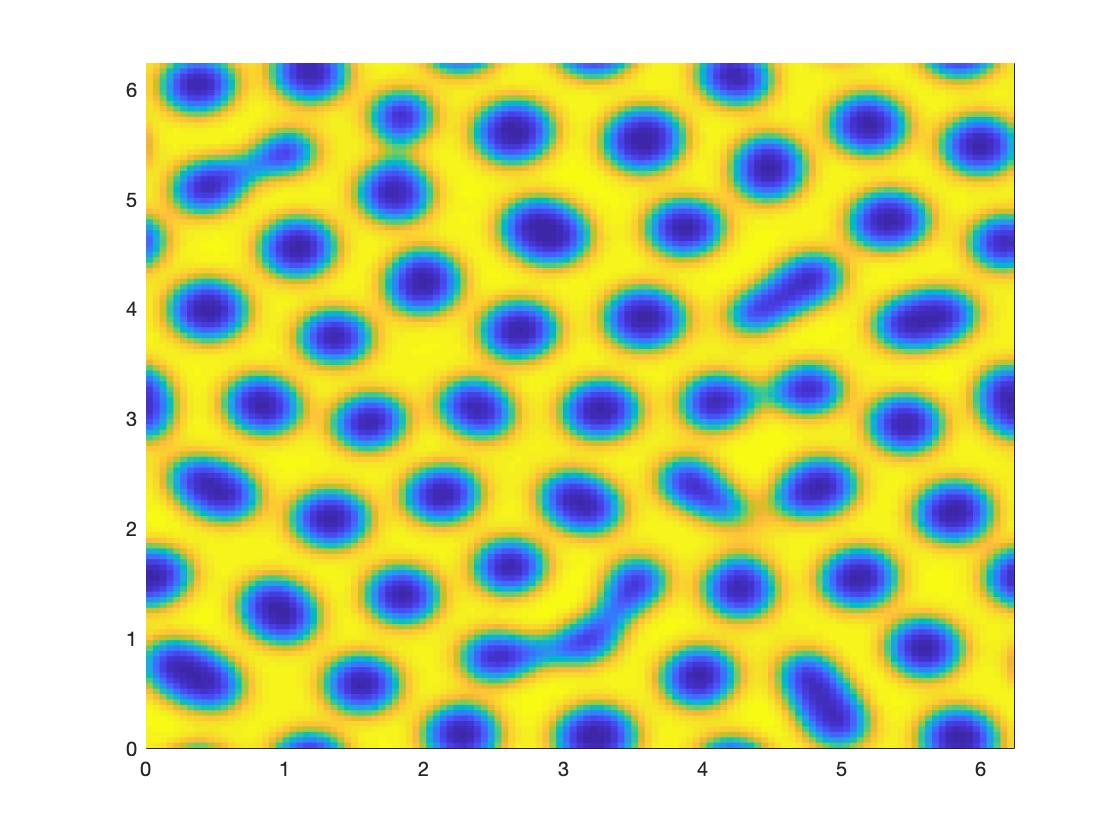}}
\subfigure[$G=x^3 $ ]{
\includegraphics[width=0.22\textwidth,clip==]{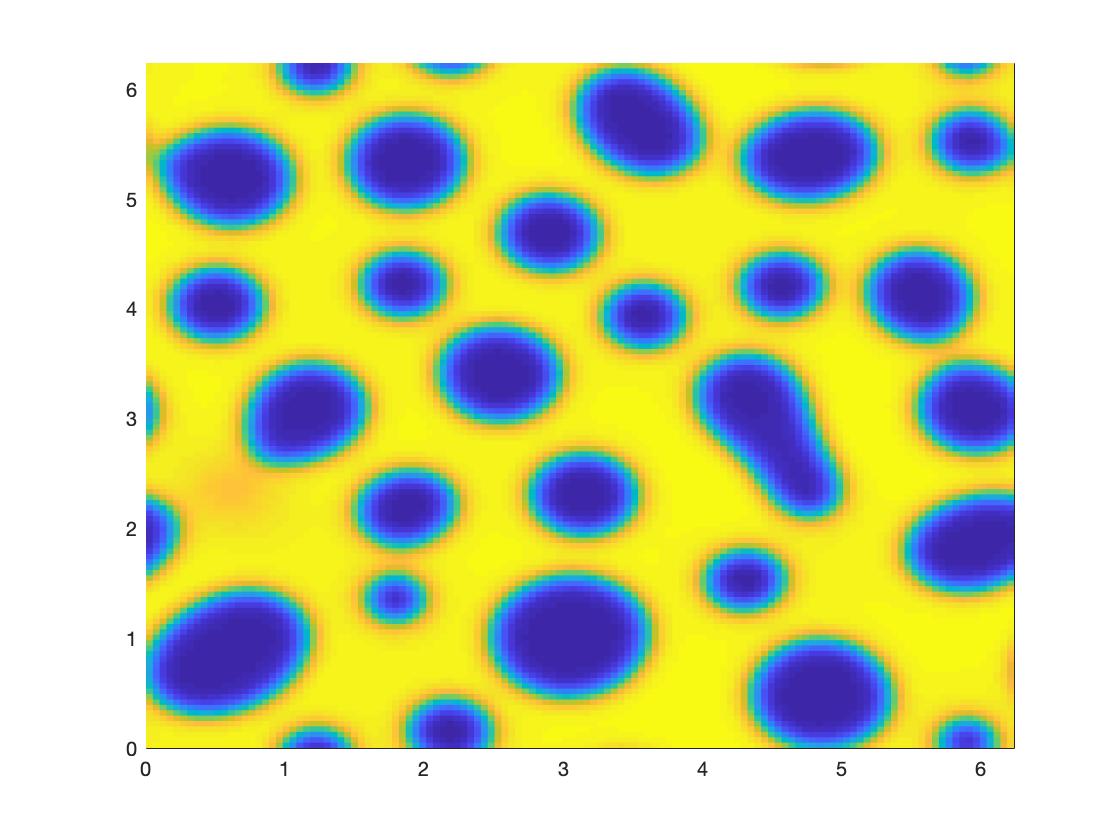}}
\subfigure[$G=x^3 $ ]{
\includegraphics[width=0.22\textwidth,clip==]{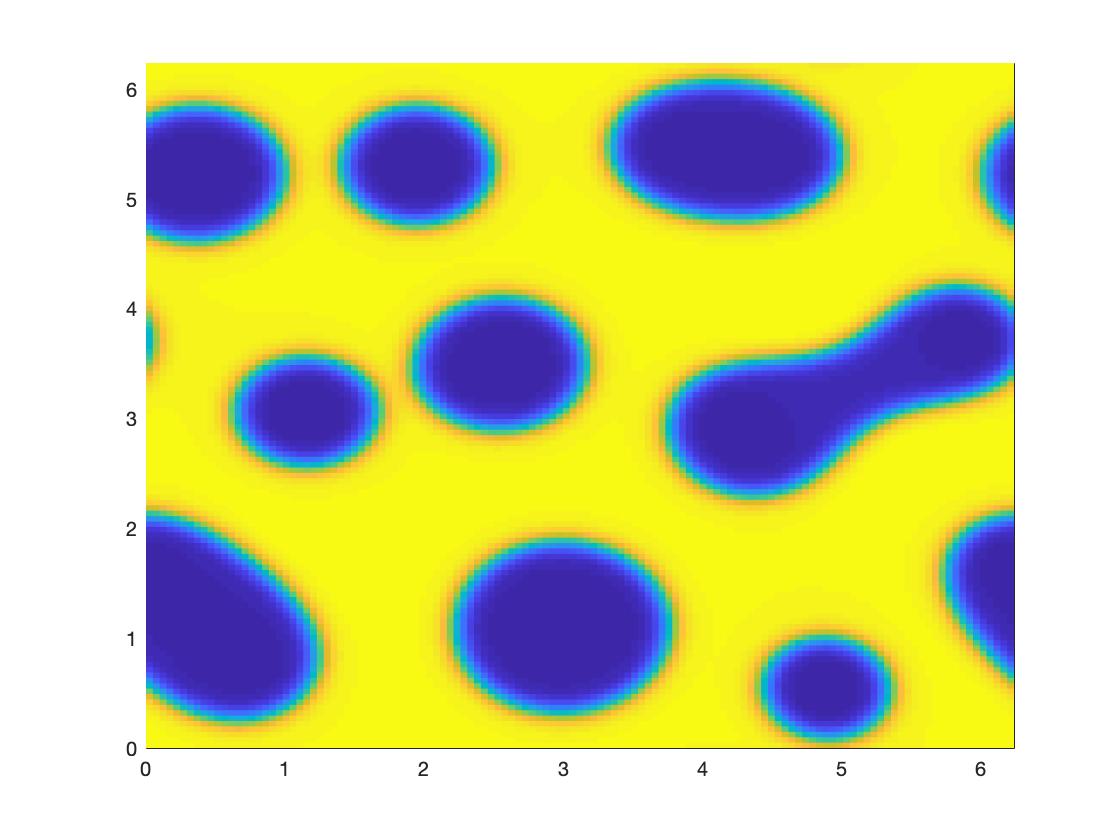}}
\subfigure[$G=x^3 $ ]{
\includegraphics[width=0.22\textwidth,clip==]{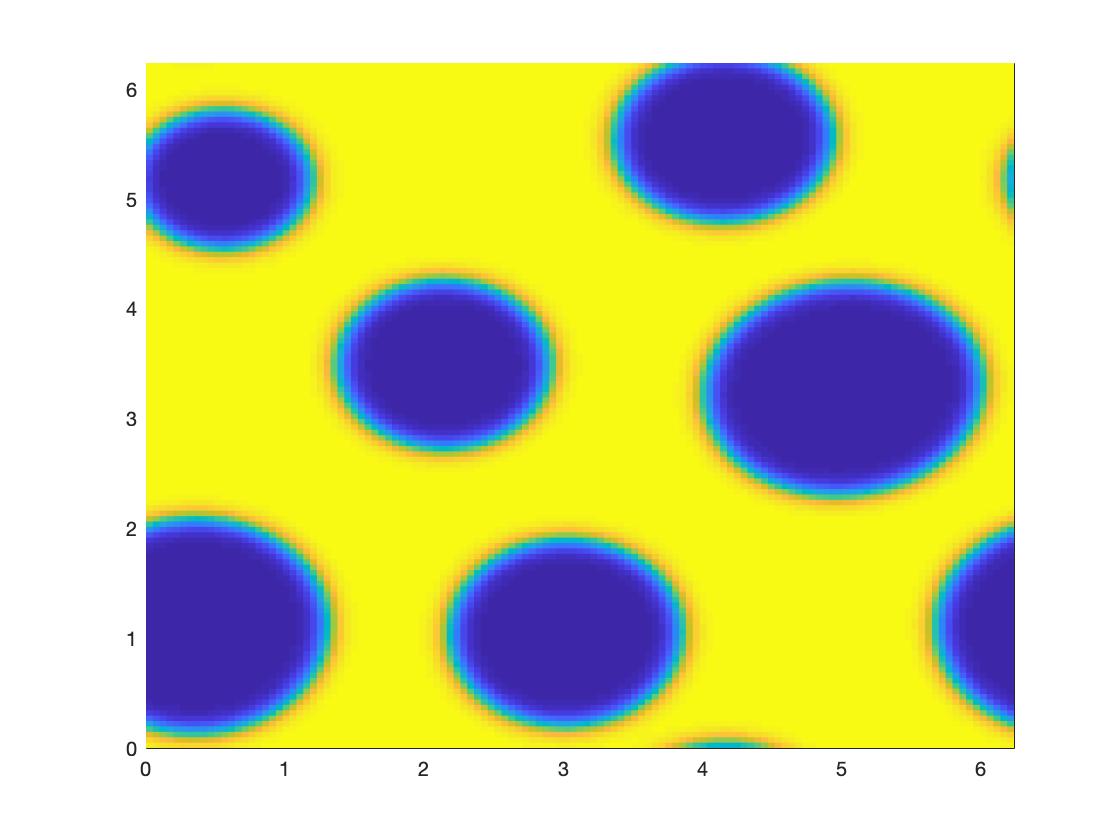}}
\subfigure[$G=e^{\frac{x}{c}} $]{
\includegraphics[width=0.22\textwidth,clip==]{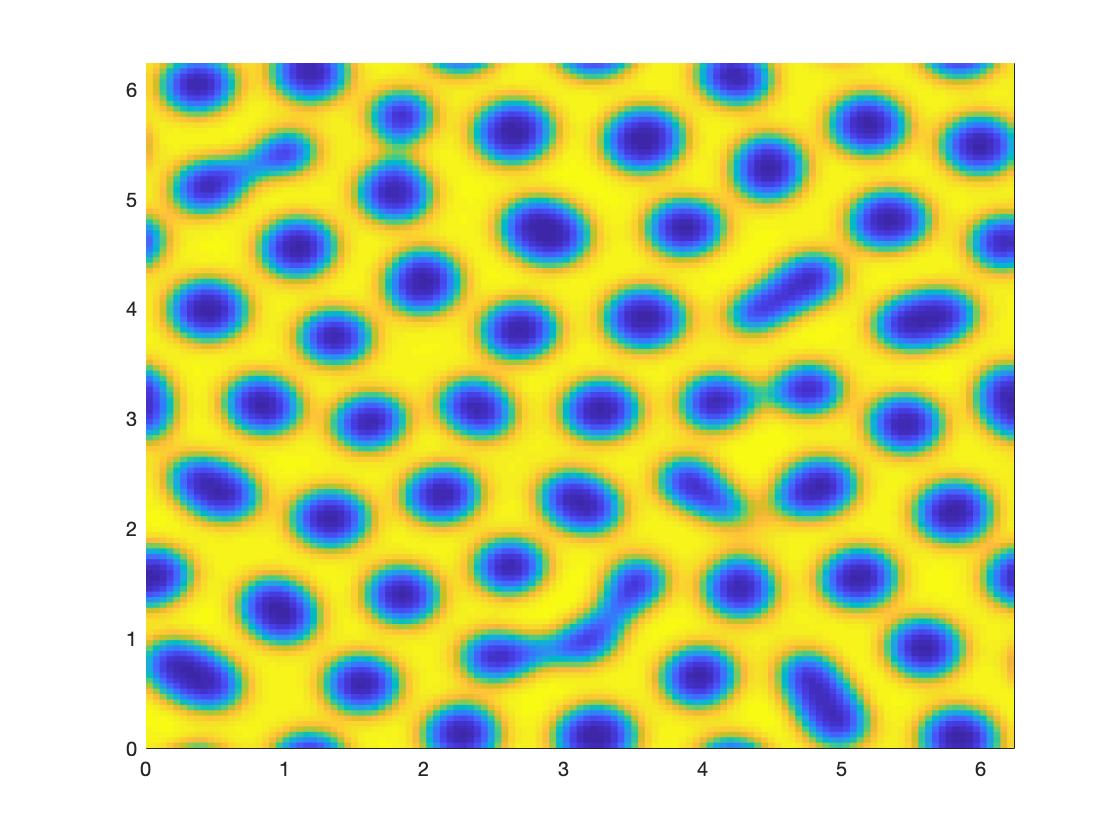}}
\subfigure[$G=e^{\frac{x}{c}}  $ ]{
\includegraphics[width=0.22\textwidth,clip==]{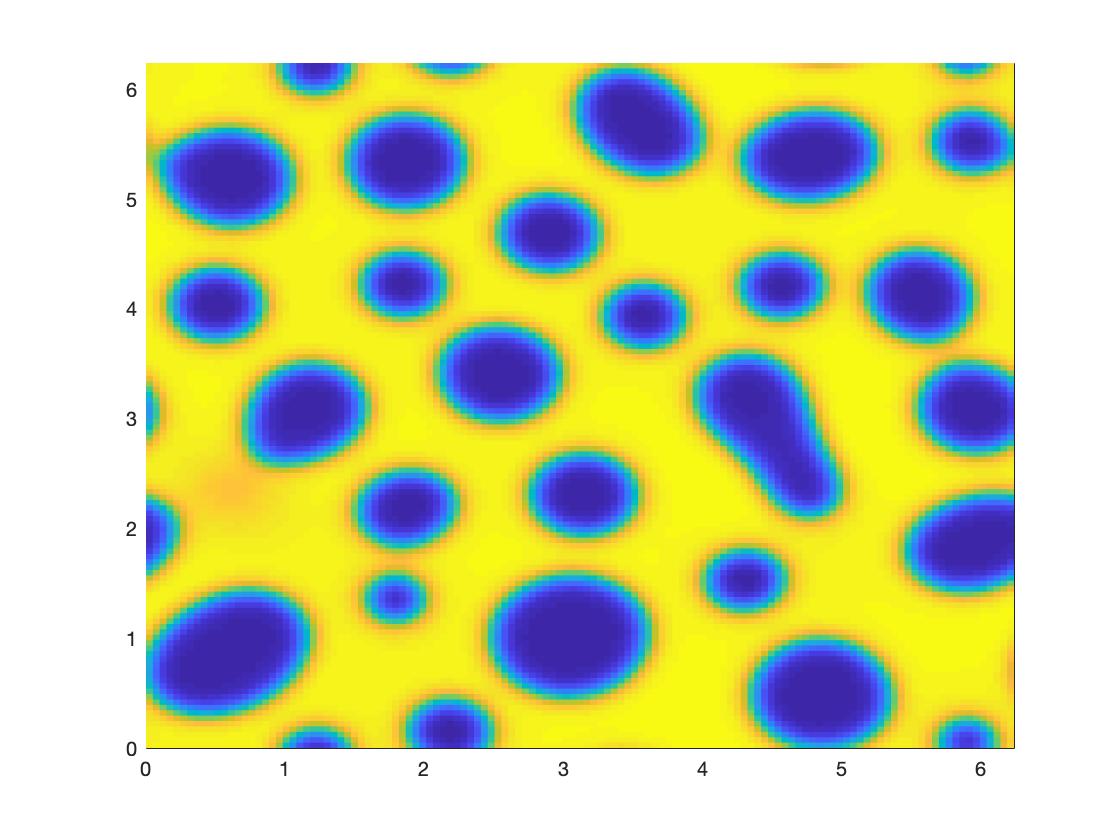}}
\subfigure[$G=e^{\frac{x}{c}}  $ ]{
\includegraphics[width=0.22\textwidth,clip==]{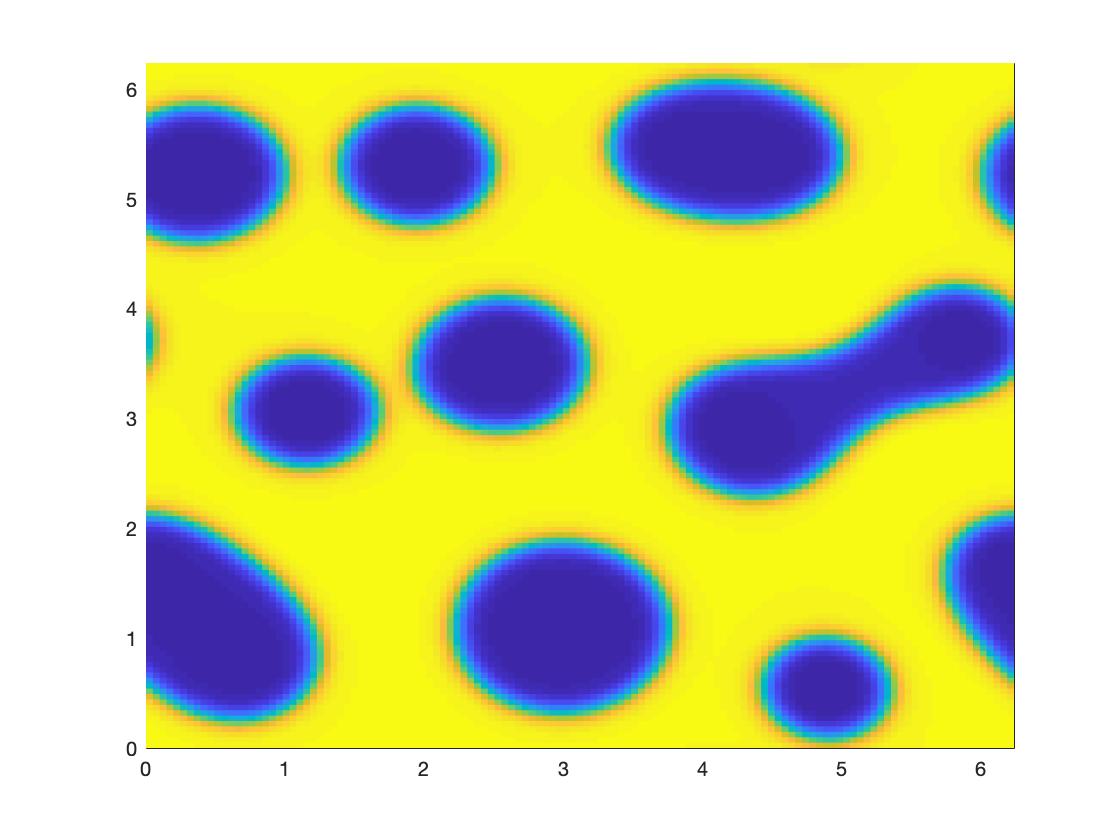}}
\subfigure[$G=e^{\frac{x}{c}} $ ]{
\includegraphics[width=0.22\textwidth,clip==]{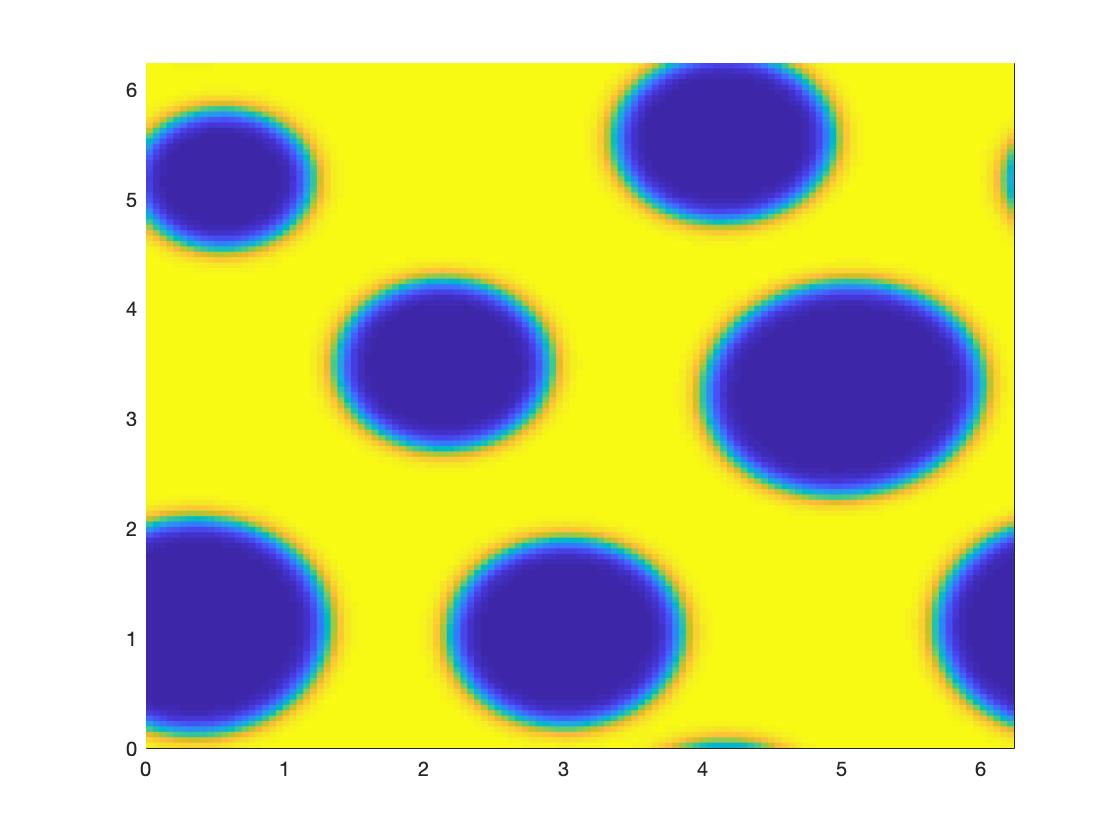}}
\subfigure[$ETDRK2$]{
\includegraphics[width=0.22\textwidth,clip==]{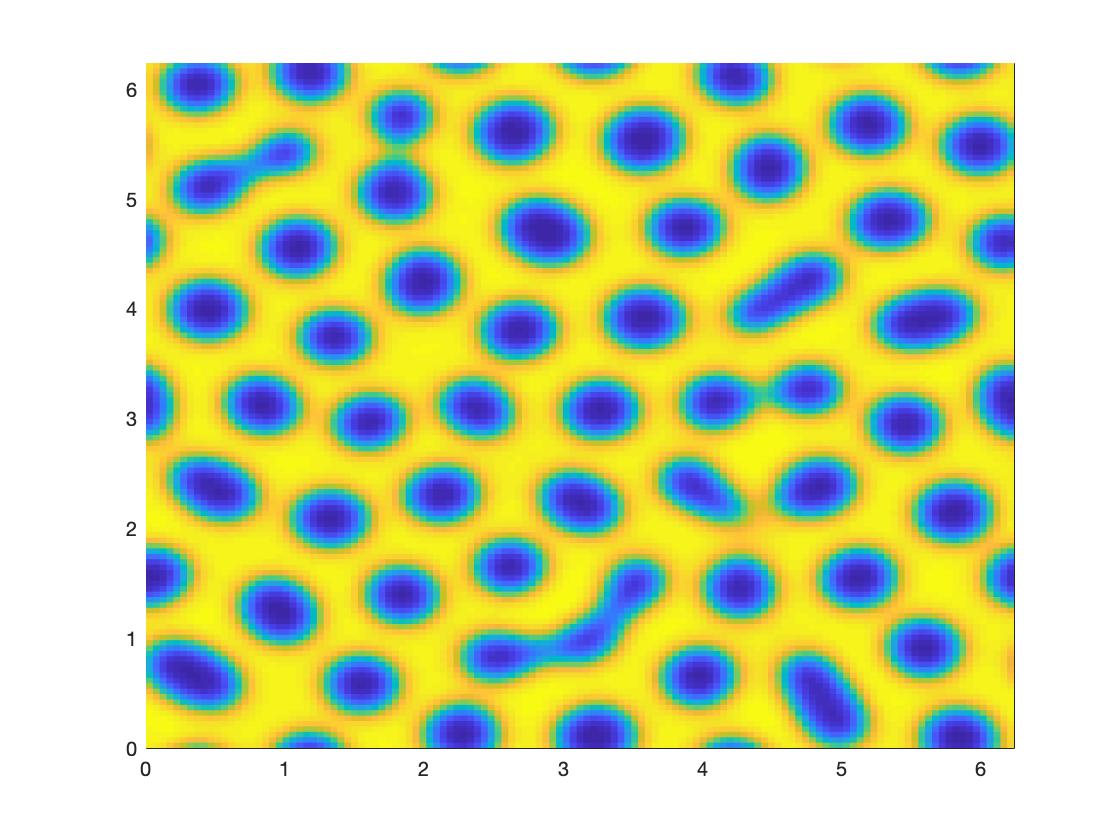}}
\subfigure[$ETDRK2 $ ]{
\includegraphics[width=0.22\textwidth,clip==]{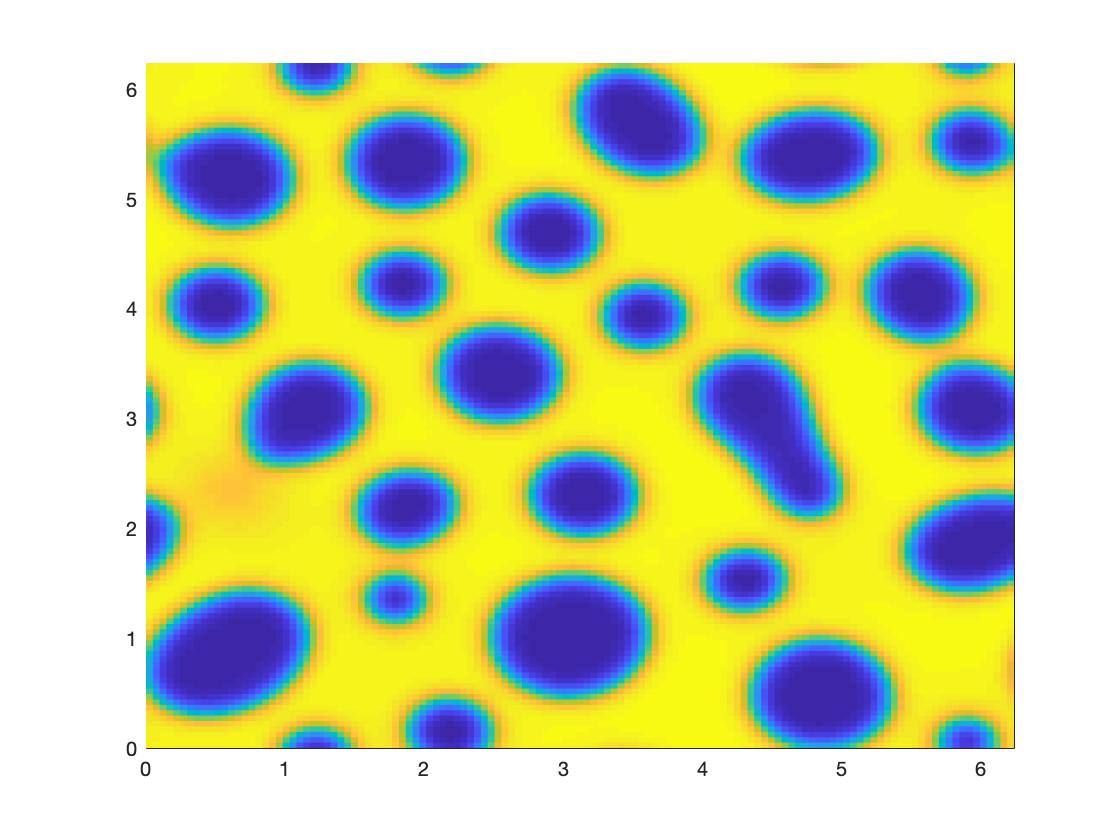}}
\subfigure[$ETDRK2  $ ]{
\includegraphics[width=0.22\textwidth,clip==]{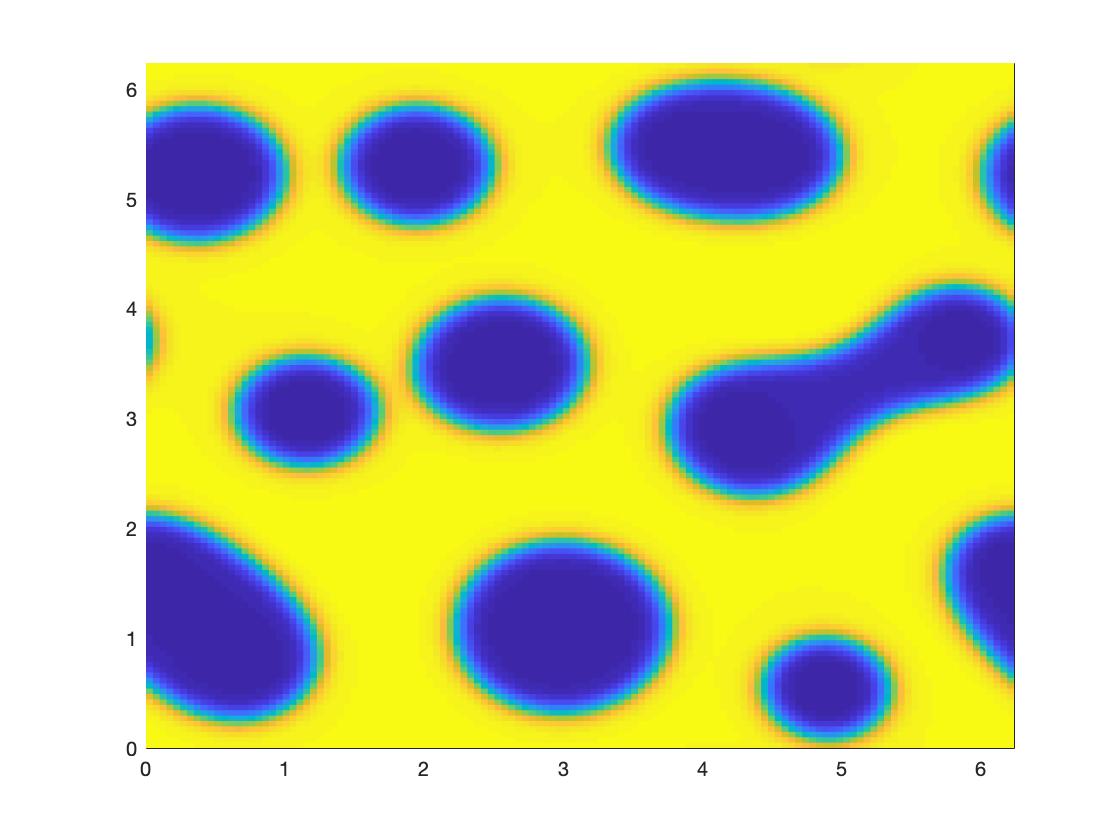}}
\subfigure[$ETDRK2$ ]{
\includegraphics[width=0.22\textwidth,clip==]{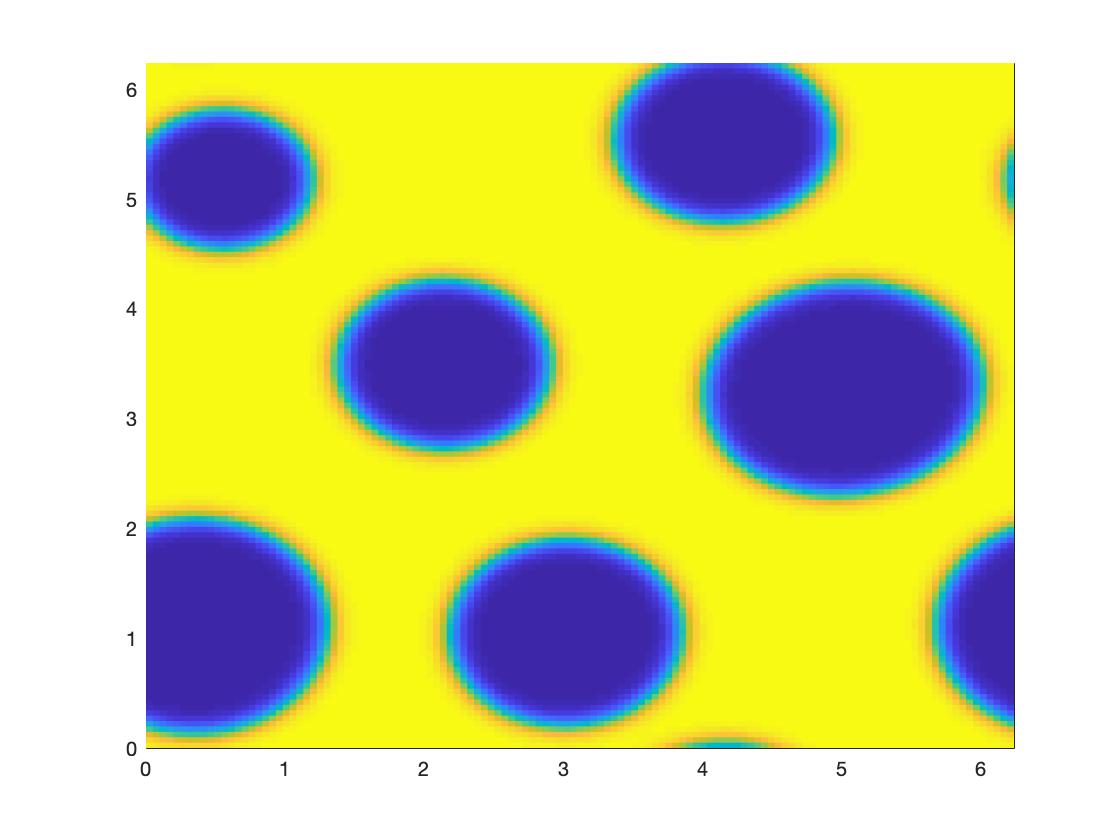}}
\caption{The 2D dynamical evolutions of the phase variable $\phi$ at $t=0.0025,0.01,0.04,0.1$ for the Cahn-Hilliard equation  with parameters $\eps^2=0.005$ computed by BDF2 scheme of  various no-iterative  SAV approaches and ETDRK2  with $\delta t=10^{-5}$.  The constant c is equal to  $10^4$ for $\tanh$ and mapped exponential SAV approaches. }\label{CH_contour}
\end{figure}

\begin{figure}[htbp]
\centering
\includegraphics[width=0.45\textwidth,clip==]{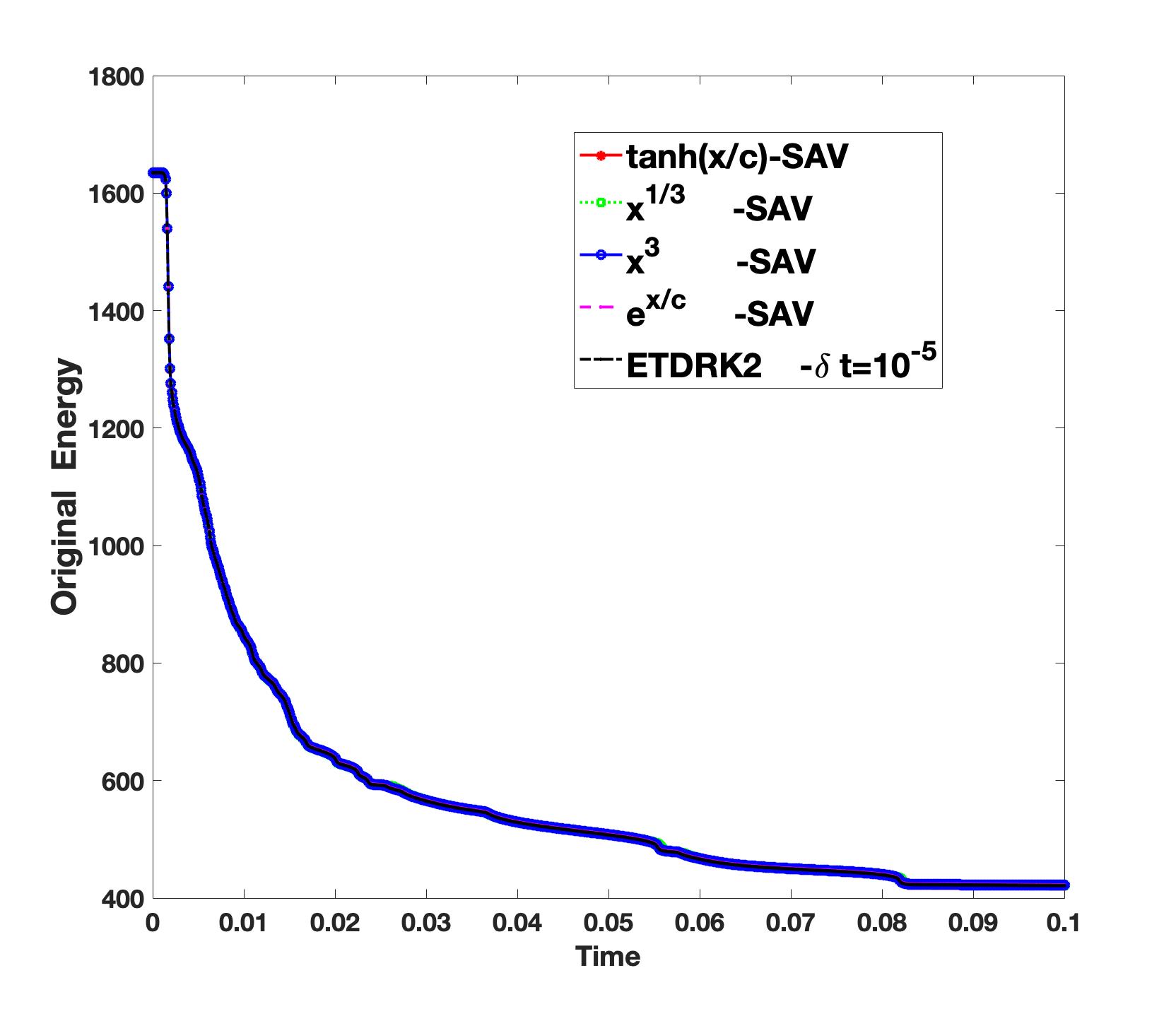}
\includegraphics[width=0.45\textwidth,clip==]{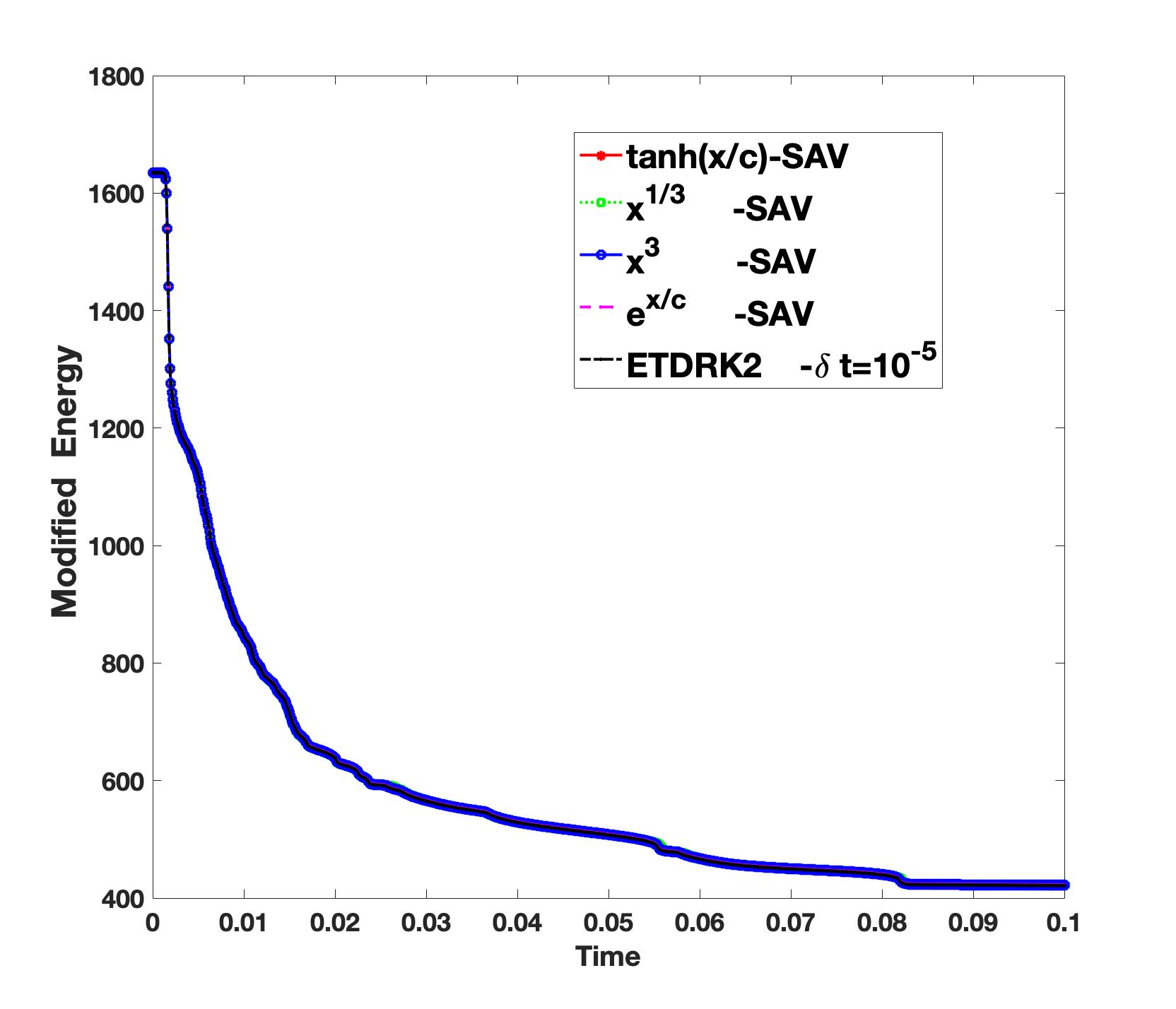}
%\hskip 0cm
\caption{Left: Evolutions of original energy by using various SAV approaches and ETDRK2  for Fig,\,\ref{CH_contour}; Right: Evolutions of the corresponding modified energy of various SAV approaches.}\label{G_SAV_energy}
\end{figure}

In the  Fig.\,\ref{CH_contour}, we plot the dynamic evolution of  phase separation for Cahn-Hilliard equation by using different BDF2 schemes of  the first approach  with $\delta t=10^{-5}$ and scheme  ETDRK2 \cite{cox2002exponential} with $\delta t=10^{-5}$.  We observe that all the  G-SAV of first approach   and ETDRK2 approaches lead to indistinguishable $\phi$ in Fig,\,\ref{CH_contour}.  This indicates that the accuracy of  BDF2 schemes by using G-SAV approaches is comparable with ETDRK2.

In the Fig.\,\ref{G_SAV_energy}, the evolutions of original and modified energy computed by schemes G-SAV and ETDRK2  are depicted.  From this figure we find no visible difference  are observed for those energy curves which are consistent with the numerical results in  Fig.\,\ref{CH_contour}.

\subsubsection{\bf Convergence rate with given exact solution}
We  test the  convergence rate of BDF2 and Crank-Nicolson schemes using various G-SAV approaches for 2D Allen-Cahn equation  with the exact solution
\begin{equation}\label{exact}
\phi(x,y,t)=(\frac{\sin(2x)\cos(2y)}{4}+0.48)(1-\frac{\sin^2(t)}{2}). 
\end{equation} 
 The  manufactured  exact  solution \eqref{exact} are obtained by adding a force term $f(\bx,t)$ into Allen-Cahn equation.   In table \ref{table11} and table \ref{table22},  we show the $L^{\infty}$ errors  of $\phi$ between numerical solution and the given exact  solution  with different time steps which are computed by BDF2 and Crank-Nicolson schemes of the first approach with $G=\tanh(\frac{x}{c})$  and $G=x^{3}$,  where the constant $c=10^4$.   Similar with table \ref{table11} and table \ref{table22},  we show the  
 the $L^{\infty}$ errors  of $\phi$ between numerical solution and the given exact  solution by using the second approach with same functions $G$ in table \ref{table1} and table \ref{table2}.  We observe that both first and second approaches obtain  second-order convergence rates in time.

\begin{table}[ht!]
\centering
\begin{tabular}{r||c|c|c|c}
\hline
$\delta t$            & {$BDF2-\tanh(\frac{x}{c})$}  & Order & {$CN2-\tanh(\frac{x}{c})$} & Order    \\ \hline
$8\times 10^{-4}$    &$5.41E(-4)$  & $-$   &$5.77E(-4)$ &$-$       \\\hline
$4\times 10^{-4}$    &$1.35E(-4)$  & $1.98$  &$1.46E(-4)$ &$2.00$       \\\hline
$2\times 10^{-4}$     &$3.37E(-5)$ & $1.99$&$3.66E(-5)$ &$2.00$    \\\hline
$1\times 10^{-4}$     &$8.41E(-6)$ &$1.99$  &$9.16E(-6)$ &$2.00$  \\\hline
$5\times 10^{-5}$  &$2.10E(-6)$ &$2.00$ &$2.29E(-6)$ &$2.00$   \\ \hline
$2.5\times 10^{-5}$  &$5.24E(-7)$&$1.99$ &$5.73E(-7)$ &$2.00$ \\\hline
$1.25\times 10^{-5}$  &$1.31E(-7)$&$2.00$ &$1.43E(-7)$ &$2.00$ \\\hline
\hline
\end{tabular}
\vskip 0.5cm
\caption{Accuracy test: with given exact solution for the Allen-Cahn equation. The $L^{\infty}$ errors at $t=0.1$ for the phase variables $\phi$  computed by the scheme based on schemes BDF2 and Crank-Nicolson  using  {\bf the first  approach} with $G=\tanh(\frac{x}{c})$, $c=10^4$. }\label{table11}
\end{table}

\begin{table}[ht!]
\centering
\begin{tabular}{r||c|c|c|c}
\hline
$\delta t$            & {$BDF2 - x^3$}  & Order & {$CN2 - x^3$} & Order    \\ \hline
$8\times 10^{-4}$    &$1.47E(-3)$  & $-$   &$1.42E(-3)$ &$-$       \\\hline
$4\times 10^{-4}$    &$3.74E(-4)$  & $1.97$  &$3.70E(-4)$ &$1.94$       \\\hline
$2\times 10^{-4}$     &$9.37E(-5)$ & $1.99$&$9.48E(-5)$ &$1.96$    \\\hline
$1\times 10^{-4}$     &$2.33E(-5)$ &$2.00$  &$2.39E(-5)$ &$1.98$  \\\hline
$5\times 10^{-5}$  &$5.84E(-6)$ &$1.99$ &$6.00E(-6)$ &$1.99$   \\ \hline
$2.5\times 10^{-5}$  &$1.45E(-6)$&$2.00$ &$1.50E(-6)$ &$2.00$ \\\hline
$1.25\times 10^{-5}$  &$3.64E(-7)$&$1.99$ &$3.77E(-7)$ &$1.99$ \\\hline
\hline
\end{tabular}
\vskip 0.5cm
\caption{Accuracy test: with given exact solution for the Allen-Cahn equation. The $L^{\infty}$ errors at $t=0.1$ for the phase variables $\phi$  computed by the scheme based on schemes BDF2 and Crank-Nicolson  using {\bf the first  approach} with $G=x^3$. }\label{table22}
\end{table}

\begin{table}[ht!]
\centering
\begin{tabular}{r||c|c|c|c}
\hline
$\delta t$            & {$BDF2-\tanh(\frac{x}{c})$}  & Order & {$CN2-\tanh(\frac{x}{c})$} & Order    \\ \hline
$8\times 10^{-4}$    &$6.39E(-5)$  & $-$   &$1.11E(-4)$ &$-$       \\\hline
$4\times 10^{-4}$    &$1.56E(-5)$  & $2.03$  &$2.78E(-5)$ &$2.00$       \\\hline
$2\times 10^{-4}$     &$3.85E(-6)$ & $2.02$&$6.94E(-6)$ &$2.00$    \\\hline
$1\times 10^{-4}$     &$9.55E(-7)$ &$2.01$  &$1.73E(-6)$ &$2.00$  \\\hline
$5\times 10^{-5}$  &$2.38E(-7)$ &$2.00$ &$4.32E(-7)$ &$2.00$   \\ \hline
$2.5\times 10^{-5}$  &$5.96E(-8)$&$2.00$ &$1.08E(-7)$ &$2.00$ \\\hline
$1.25\times 10^{-5}$  &$1.53E(-8)$&$1.96$ &$2.73E(-8)$ &$1.98$ \\\hline
\hline
\end{tabular}
\vskip 0.5cm
\caption{Accuracy test: with given exact solution for the Allen-Cahn equation. The $L^{\infty}$ errors at $t=0.1$ for the phase variables $\phi$  computed by the scheme based on schemes BDF2 and Crank-Nicolson  using  {\bf the second  approach} with $G=\tanh(\frac{x}{c})$, $c=10^4$. }\label{table1}
\end{table}

\begin{table}[ht!]
\centering
\begin{tabular}{r||c|c|c|c}
\hline
$\delta t$            & {$BDF2 - x^3$}  & Order & {$CN2 - x^3$} & Order    \\ \hline
$8\times 10^{-4}$    &$6.28E(-5)$  & $-$   &$1.10E(-4)$ &$-$       \\\hline
$4\times 10^{-4}$    &$1.53E(-5)$  & $2.04$  &$2.77E(-5)$ &$1.99$       \\\hline
$2\times 10^{-4}$     &$3.77E(-6)$ & $2.02$&$6.91E(-6)$ &$2.00$    \\\hline
$1\times 10^{-4}$     &$9.36E(-7)$ &$2.01$  &$1.72E(-6)$ &$2.00$  \\\hline
$5\times 10^{-5}$  &$2.33E(-7)$ &$2.01$ &$4.31E(-7)$ &$2.00$   \\ \hline
$2.5\times 10^{-5}$  &$5.84E(-8)$&$2.00$ &$1.07E(-7)$ &$2.01$ \\\hline
$1.25\times 10^{-5}$  &$1.50E(-8)$&$1.96$ &$2.71E(-8)$ &$1.98$ \\\hline
\hline
\end{tabular}
\vskip 0.5cm
\caption{Accuracy test: with given exact solution for the Allen-Cahn equation. The $L^{\infty}$ errors at $t=0.1$ for the phase variables $\phi$  computed by the scheme based on schemes BDF2 and Crank-Nicolson  using {\bf the second  approach} with $G=x^3$. }\label{table2}
\end{table}

\subsubsection{Cahn-Hilliard equation with singular potential}
In this subsection, we show the validation for Cahn-Hilliard equation with logarithmic (singular) potential \cite{chen2019positivity}. The total free energy is 
\begin{equation}
E(\phi)=\int_{\Omega}\frac{\eps^2}{2}|\Grad\phi|^2+F(\phi)d\bx,
\end{equation}
where the logarithmic Flory Huggins energy potential is  
\begin{equation}
F(\phi) = -\frac{\theta}{2}\phi^2 + (1+\phi)\ln(1+\phi) +(1-\phi)\ln(1-\phi).
\end{equation}
Parameter $\theta$ is a  positive constant which is associated with diffusive interface. The chemical potential for $H^{-1}$ gradient flow is 
\begin{equation}
\mu=\ln(1+\phi)-\ln(1-\phi) -\theta \phi- \eps^2\Delta \phi. 
\end{equation}
We consider the second approach by choosing  $G=e^{\frac{x}{c}}$ with $c=10^4$ and  define the new variable $r$ as 
\begin{equation}
r = e^{\frac{\int_{\Omega} -\frac{\theta}{2}\phi + (1+\phi)\ln(1+\phi) +(1-\phi)\ln(1-\phi)d\bx}{c}}=G(\frac{\int_{\Omega}F(\phi)d\bx}{c}),
\end{equation}
where $c$ is a positive constant.
Compared with original SAV approach and  IEQ approach,  the exponential-SAV approach  eliminates the constraint: $\int_{\Omega}F(\phi)d\bx+c>0$.  In Fig.\,\ref{spin_1}  we simulate the phase separation at various time  by choosing  phase parameters  $\theta=3$ and $\eps^2=0.002$ where initial condition is \eqref{ini_rand}.

\begin{figure}
\centering
\subfigure[$t=0$.]{\includegraphics[width=0.22\textwidth,clip==]{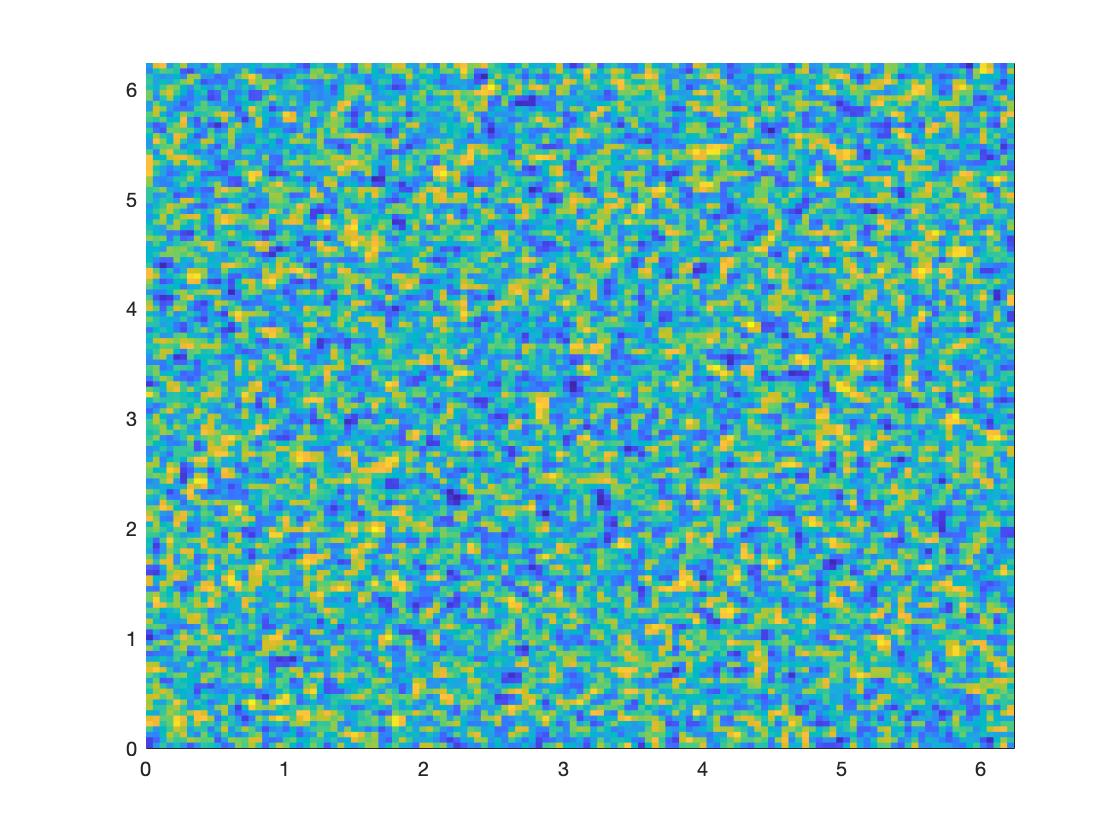}\hskip 0cm}
\subfigure[$t=0.025$.]{\includegraphics[width=0.22\textwidth,clip==]{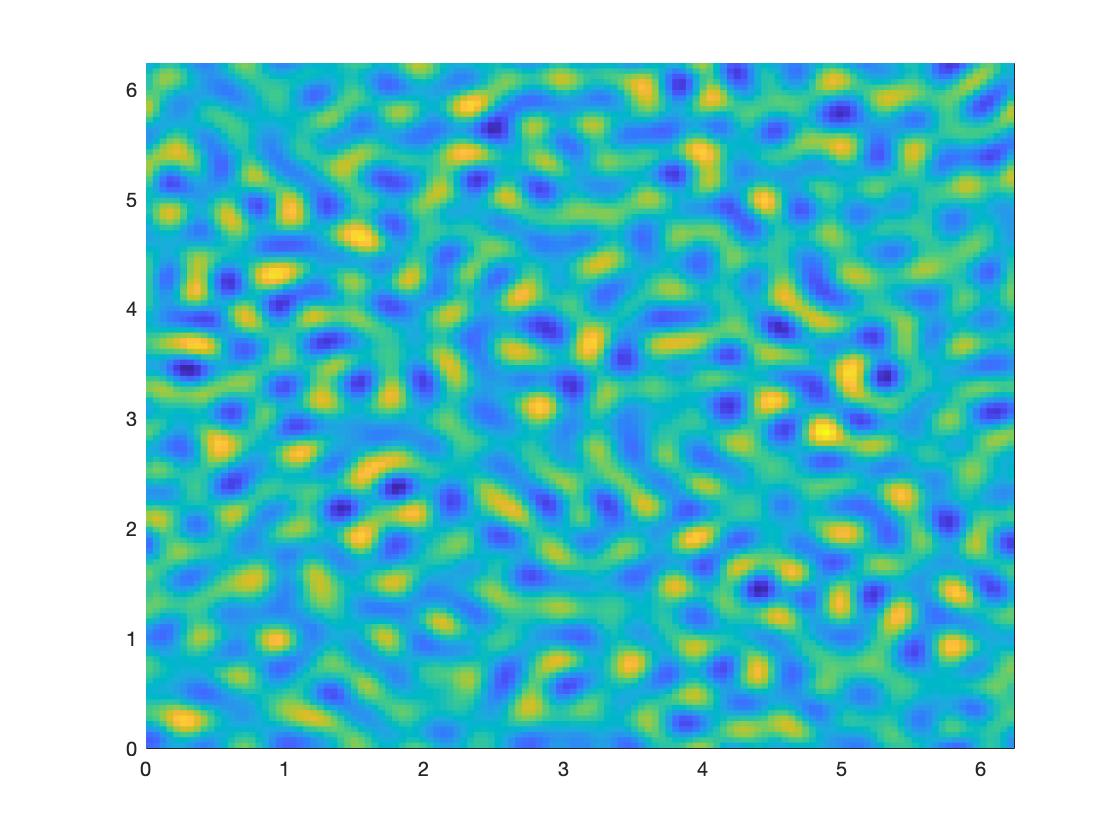}\hskip 0cm}
\subfigure[$t=0.05$.]{\includegraphics[width=0.22\textwidth,clip==]{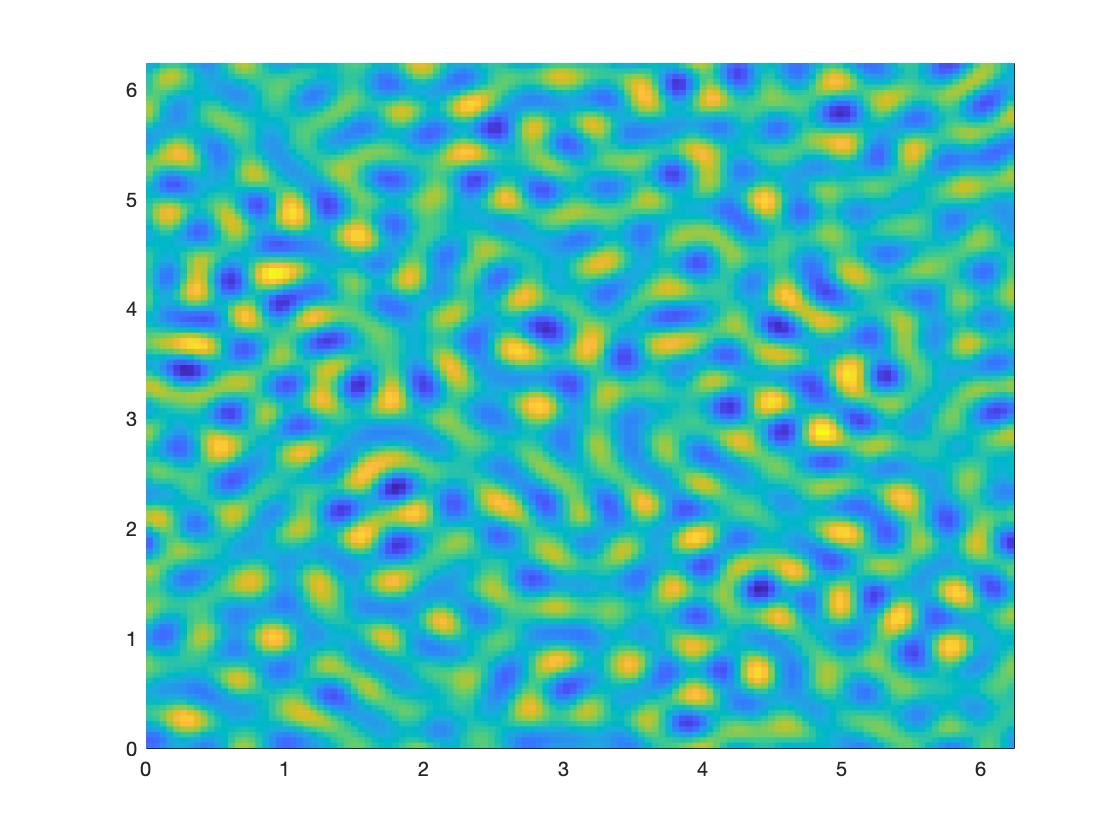}\hskip 0cm}
\subfigure[$t=0.1$.]{\includegraphics[width=0.22\textwidth,clip==]{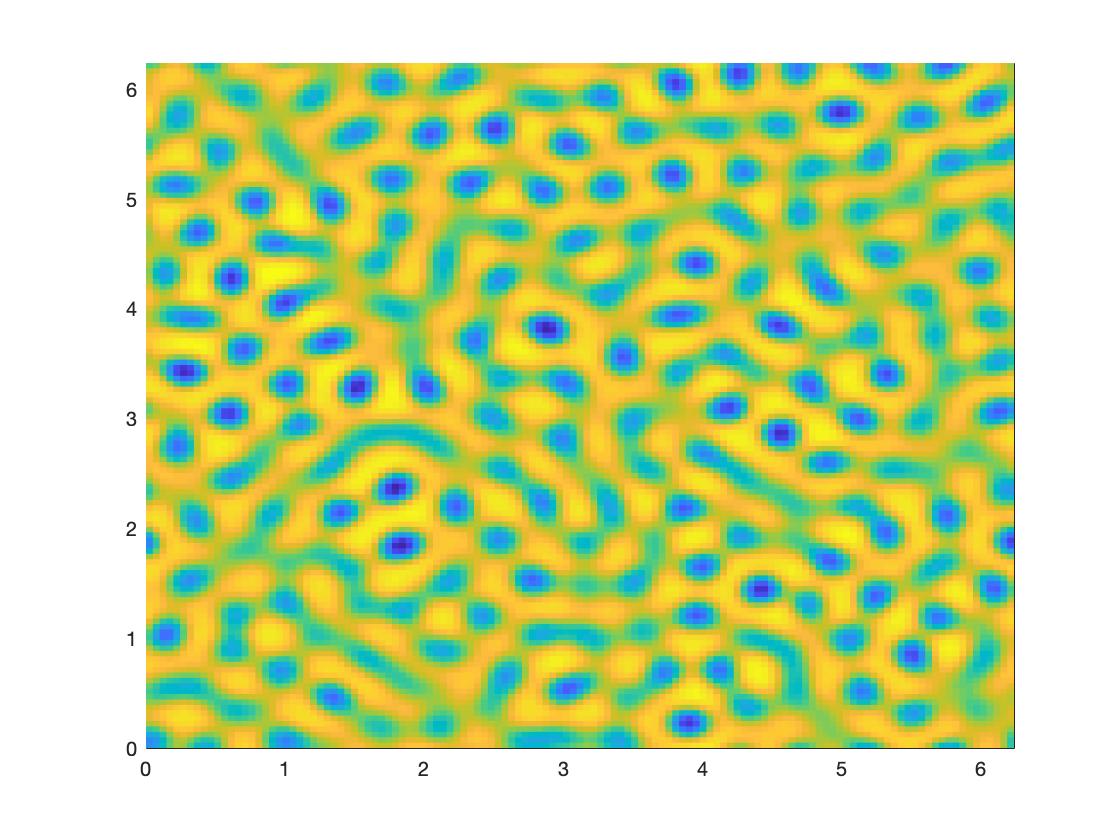}\hskip 0cm}
\subfigure[$t=0.4$.]{\includegraphics[width=0.22\textwidth,clip==]{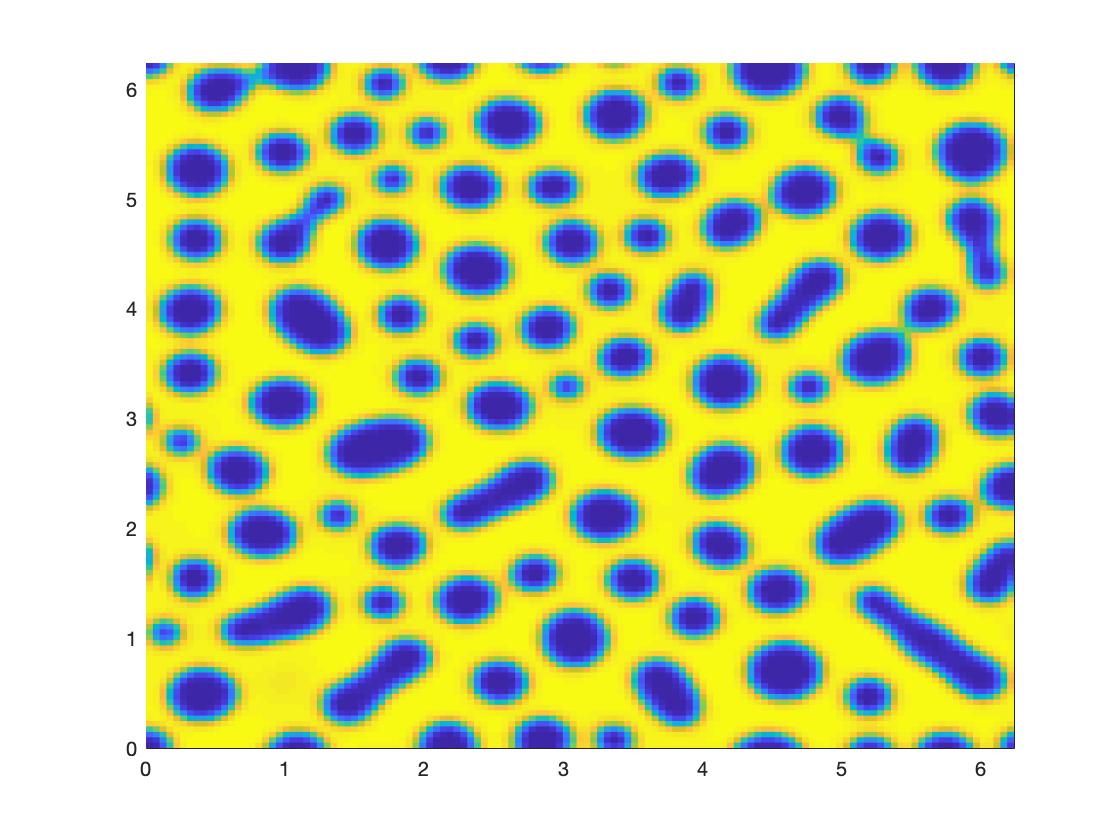}\hskip 0cm}
\subfigure[$t=3$.]{\includegraphics[width=0.22\textwidth,clip==]{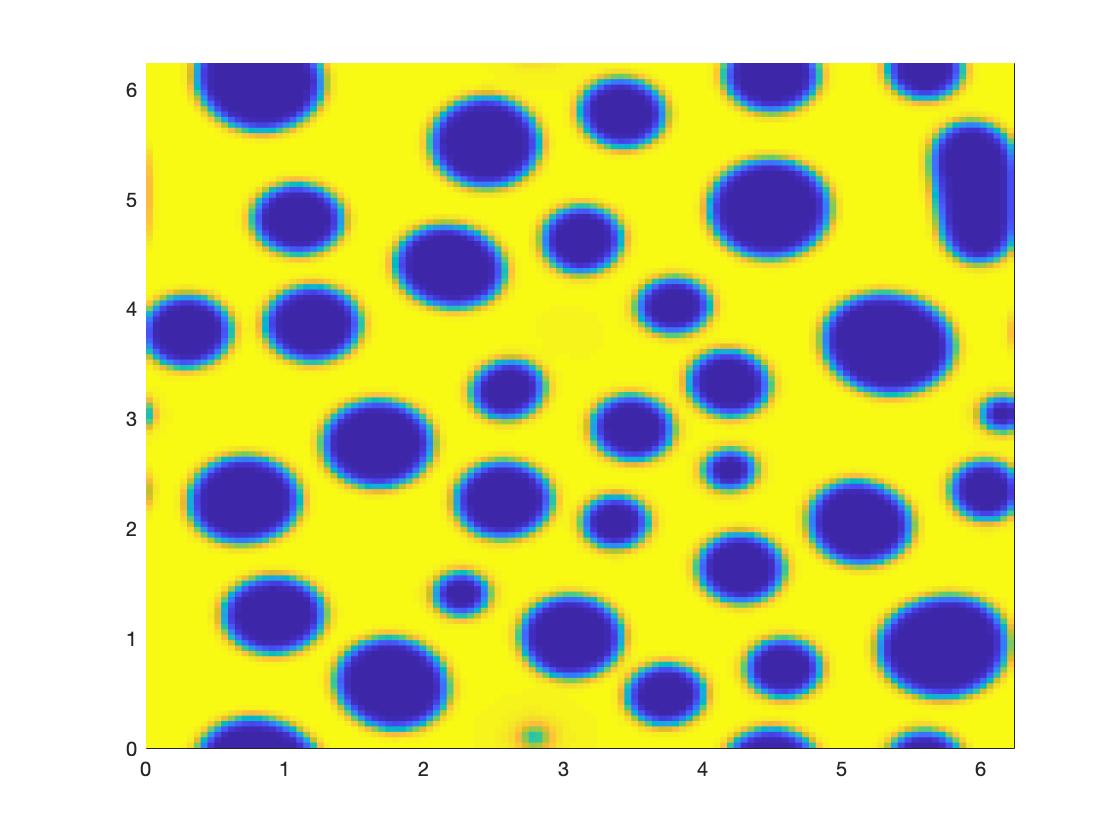}\hskip 0cm}
\subfigure[$t=6$.]{\includegraphics[width=0.22\textwidth,clip==]{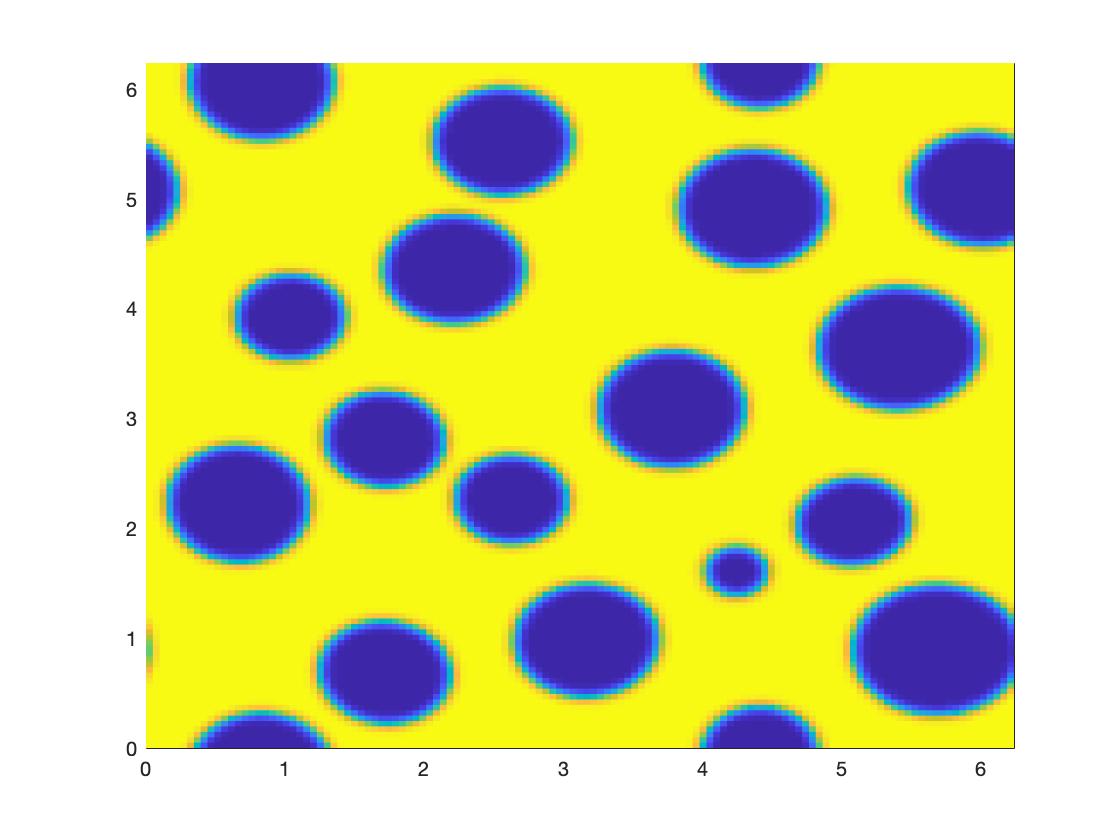}\hskip 0cm}
\subfigure[$t=10$.]{\includegraphics[width=0.22\textwidth,clip==]{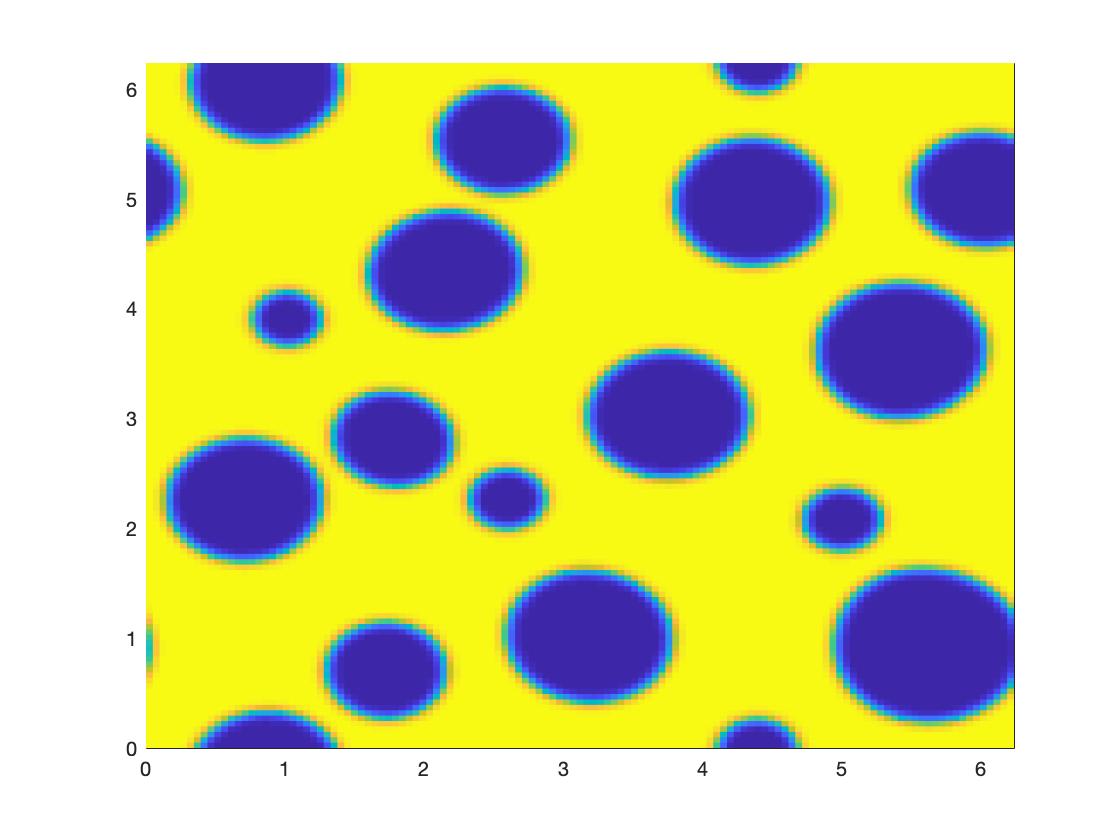}\hskip 0cm}
\caption{Dynamical evolution of the phase variables $\phi$ for the Cahn-Hilliard  model with logarithmic potential}\label{spin_1}
\end{figure}

\subsection{Molecular beam epitaxial (MBE) without slope selection}
As we mentioned above, the G-SAV approach takes big advantages of dealing with the model  where the nonlinear potential or  free energy are  unbounded from below . For example,  the molecular beam epitaxial (MBE)  without slope selection \cite{villain1991continuum}, where 
the total free energy is $E(\phi)=\int_{\Omega}\frac{\eps^2}{2}|\Delta \phi|^2+F(\phi)dx$, and 
the nonlinear potential is 
\begin{equation}\label{ori:ene2}
F(\phi)=-\frac{1}{2}\ln(1+|\nabla \phi|^2).
\end{equation}
Especially,  the $L^2$ gradient flow with respect to the free energy above is  
\begin{eqnarray}
&&\phi_t =-M\frac{\delta E(\phi)}{\delta \phi}= -M\big( \epsilon^2\Delta^2 \phi+F'(\phi)\big), \label{MBE:1}
\end{eqnarray}
 with periodic  boundary conditions:
where $\bf n$ is the unit outward normal on the boundary $\partial\Omega$.
  In the above,  $M$ is a mobility constant, and $F'(\phi)= \Grad \cdot\Big(\frac{\Grad \phi}{1+|\Grad \phi|^2}\big)$. 
  
  One has to use energy splitting method \cite{cheng2019highly} or stabilized method \cite{hou2019variant} to deal with the nonlinear part of energy. However,  G-SAV  approach can be directly applied since there is no requirement of free energy to be  bounded  from below.  
   For example, we consider the  $\tanh$-SAV approach,  the new variable $r$ is set to be 
  \begin{equation}
  r = \tanh(\frac{\int_{\Omega}-\frac{1}{2}\ln(1+|\nabla \phi|^2)d\bx}{c})=G(\frac{\int_{\Omega}F(\phi)d\bx}{c}),
  \end{equation}
where $c$ is a positive constant. Notice that we can choose any invertible  $G=x^3,x^{\frac{1}{3}},e^{\frac{x}{c}},\cdots$,  whose domain are  $(-\infty,\infty)$.
For example,  a second-order BDF2 scheme based on the $\tanh$-SAV  approach is:
\begin{eqnarray}
&&\frac{3\phi^{n+1}-4\phi^n+\phi^{n-1}}{2\delta t} +M\big( \epsilon^2\Delta^2 \phi^{n+1}+\frac{r^{\dagger,n}}{G(\int_{\Omega}F(\phi^{\dagger,n})d\bx)}F'( \phi^{\dagger,n})\big)=0,\label{mbe:bdf1}\\
&&\frac{3(G^{-1}(r))^{n+1}-4(G^{-1}(r))^n+(G^{-1}(r))^{n-1}}{2\delta t}\nonumber \\&&\hskip 1cm=\frac{r^{\dagger,n}}{G(\int_{\Omega}F(\phi^{\dagger,n})d\bx)}(F'(\phi^{\dagger,n}),\frac{3\phi^{n+1}-4\phi^n+\phi^{n-1}}{2\delta t}),\label{mbe:bdf2}
\end{eqnarray}
where $f^{\dagger,n}=2f^n-f^{n-1}$ for any sequence $\{f^n\}$.
Similarly with the proof of Theorem \ref{stable:1}, we can easily show that the above equation is unconditionally energy stable.  Meanwhile,  the scheme \eqref{mbe:bdf1}-\eqref{mbe:bdf2} can also be implemented very efficiently.

We simulate the coarsening dynamical process of MBE model \eqref{mbe:bdf1}-\eqref{mbe:bdf2},  where  a random initial condition  varying from $-0.001$ to $0.001$.
The phase parameters  are  set to be 
\begin{eqnarray}\label{Coarse_pa}
\epsilon= 0.03, \; \delta t=10^{-2},\;  M=1.
\end{eqnarray}
The computational domain is  $\Omega=[0,2\pi]^2$.  The total free energy will decay with  time growth where the decay rate behaves like  $-88\log(t)-124$  which are depicted in Fig. \ref{MBE_coarse}.  The coarse dynamic process  at various time  are  observed in Fig. \ref{NoSlopeCoarse} which are consistent with numerical results in \cite{cheng2019highly,chen2014linear}.

\begin{figure}
\centering
\subfigure[$t=0.005$.]{
\includegraphics[width=0.23\textwidth,clip==]
{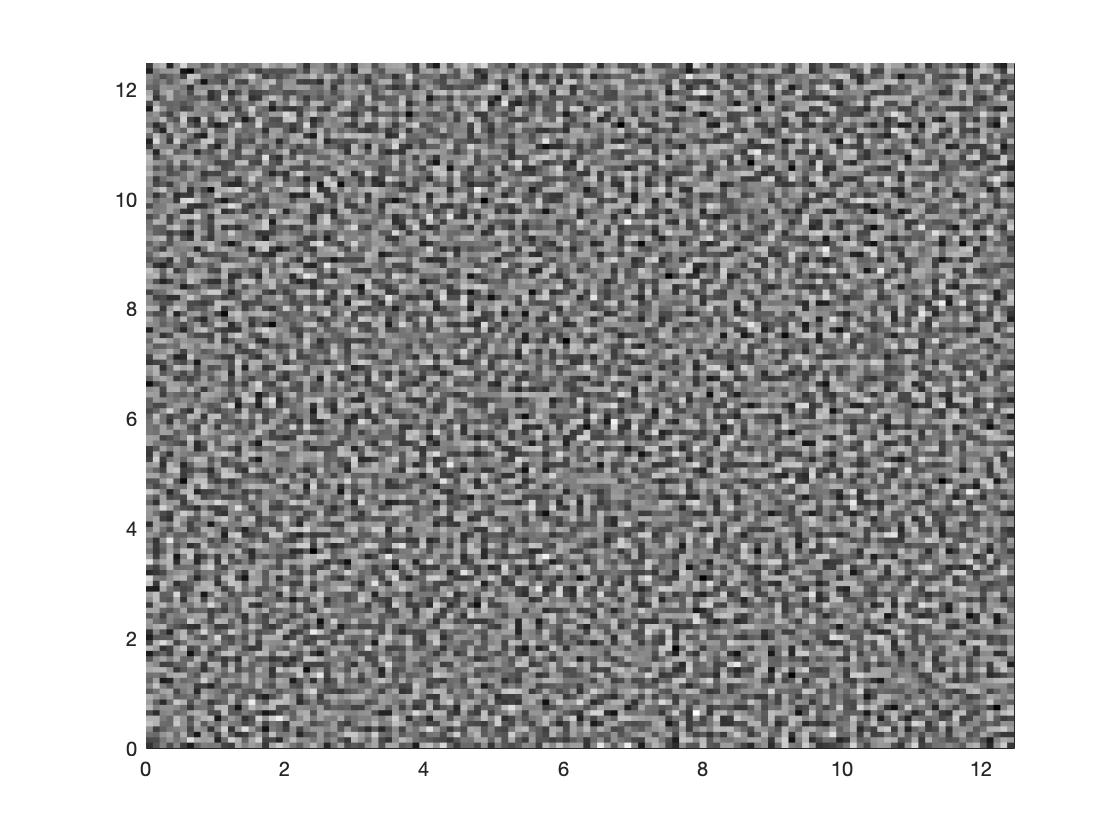}\hskip 0cm
\includegraphics[width=0.23\textwidth,clip==]{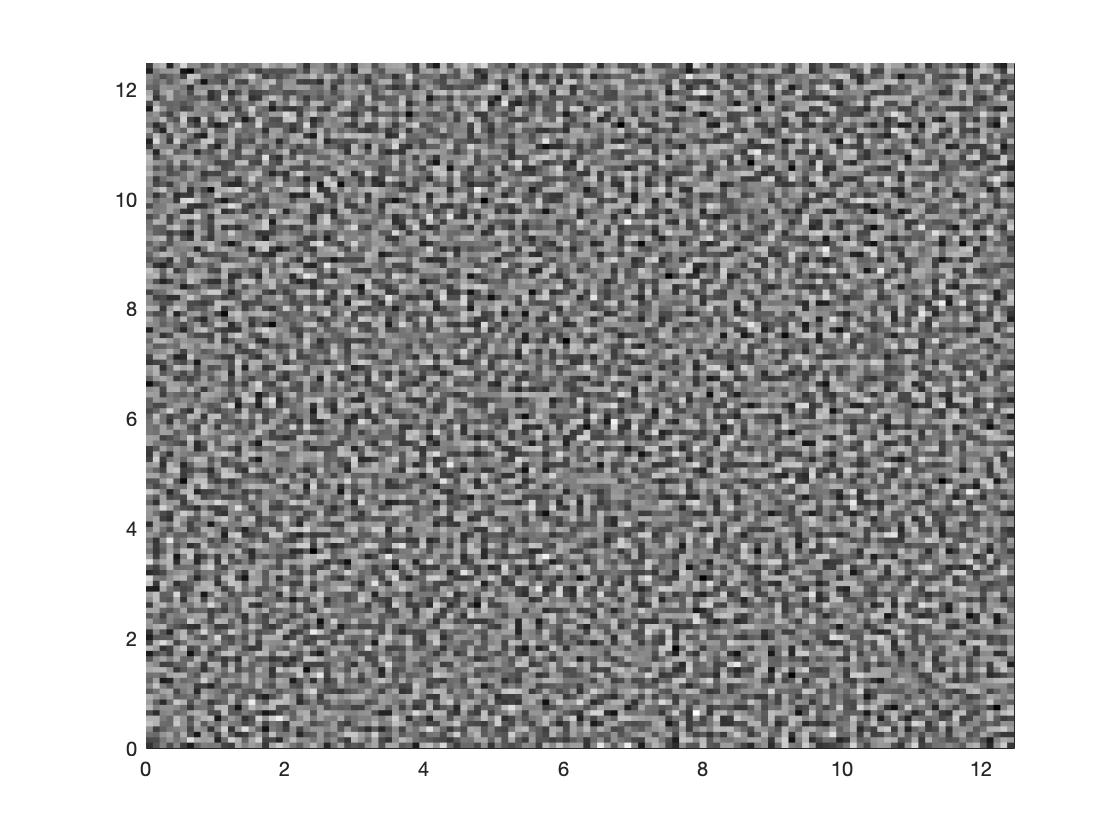}}
\subfigure[$t=1$.]{
\includegraphics[width=0.23\textwidth,clip==]{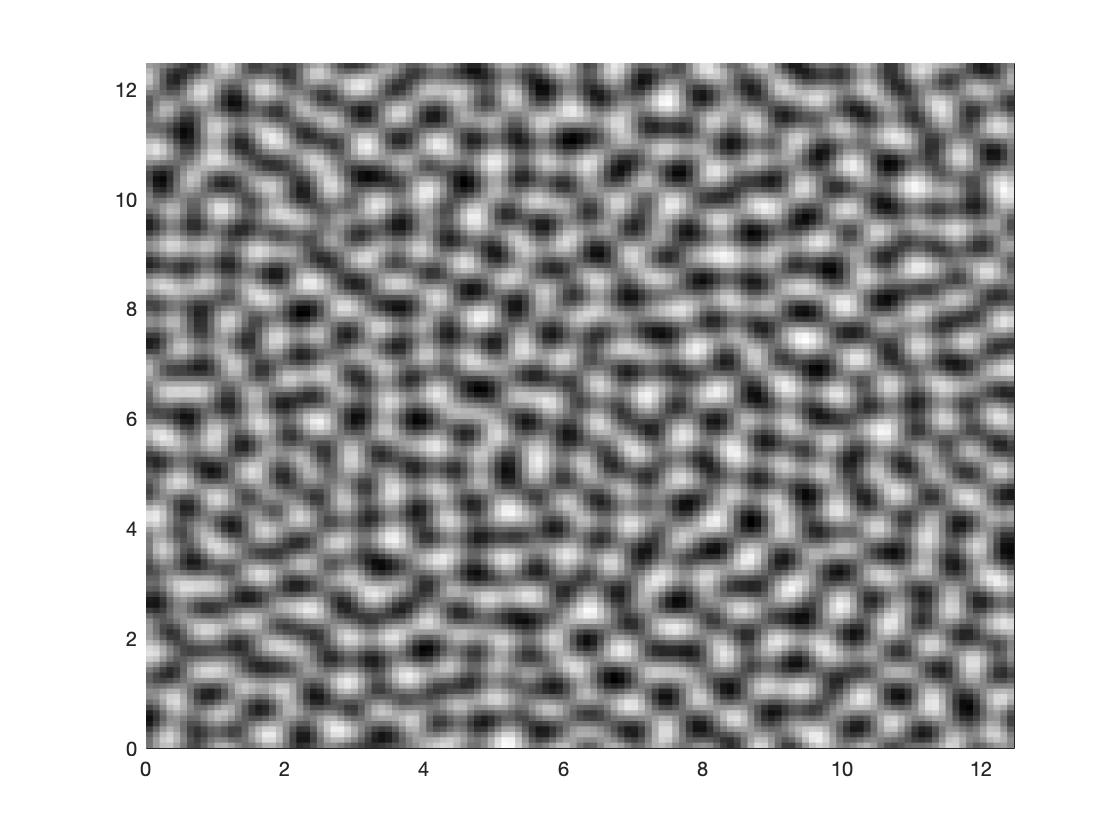}\hskip 0cm
\includegraphics[width=0.23\textwidth,clip==]{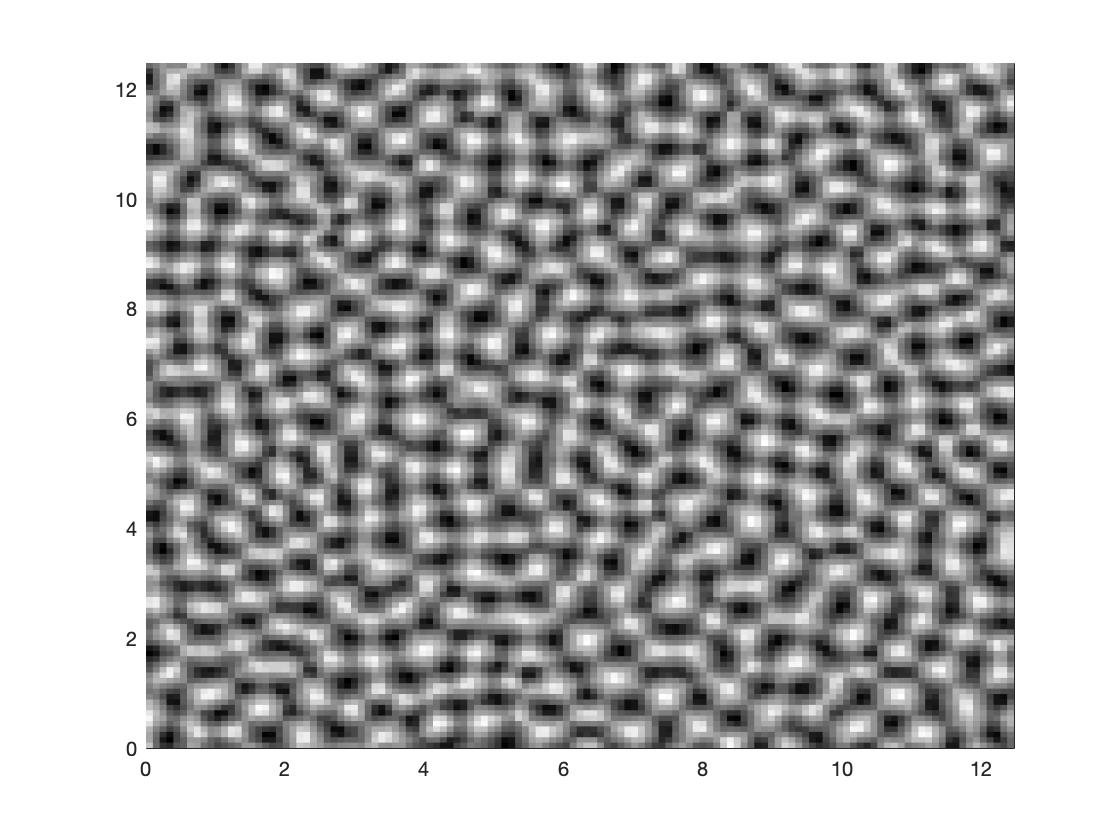}}
\subfigure[$t=6$.]{
\includegraphics[width=0.23\textwidth,clip==]{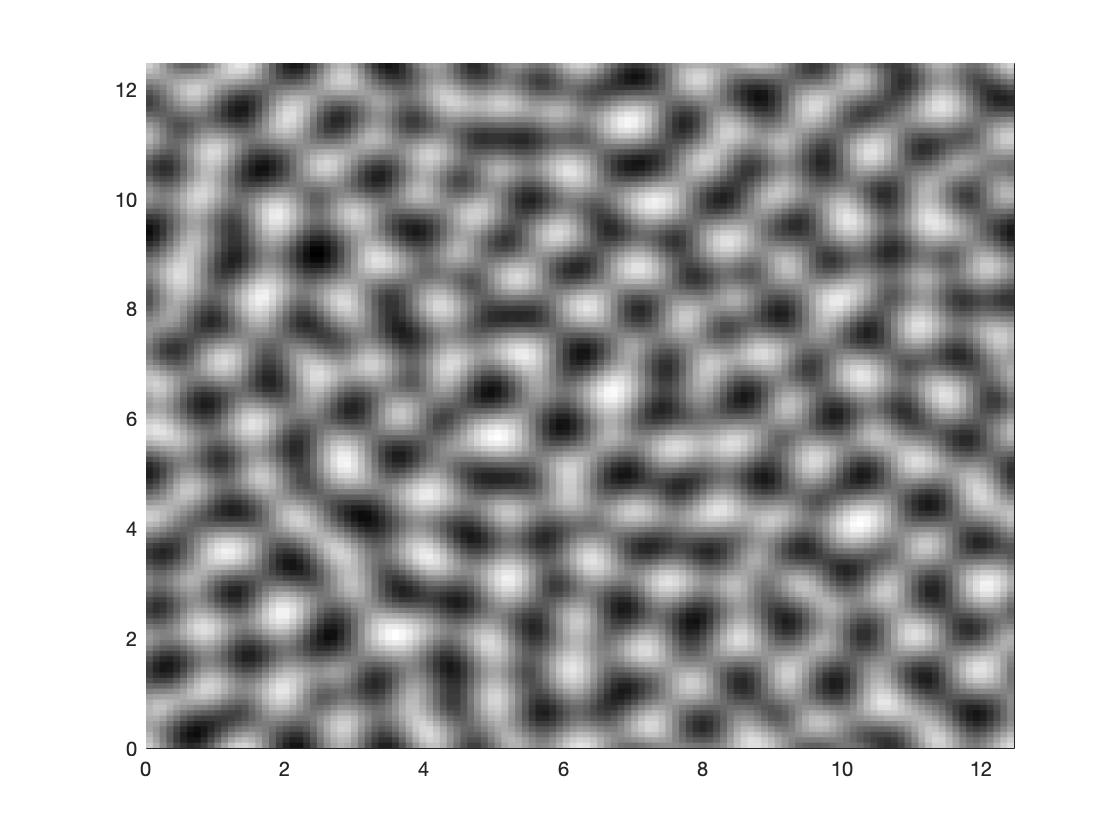}\hskip 0cm
\includegraphics[width=0.23\textwidth,clip==]{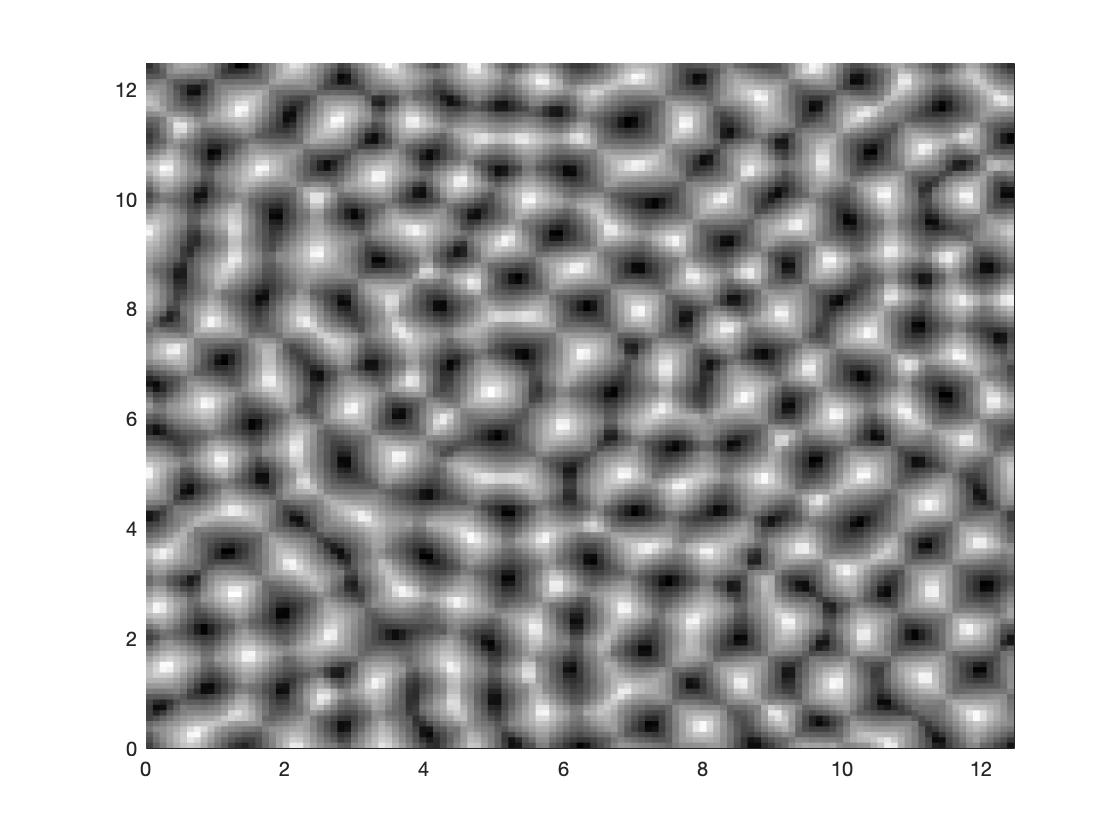}}
\subfigure[$t=20$.]{
\includegraphics[width=0.23\textwidth,clip==]{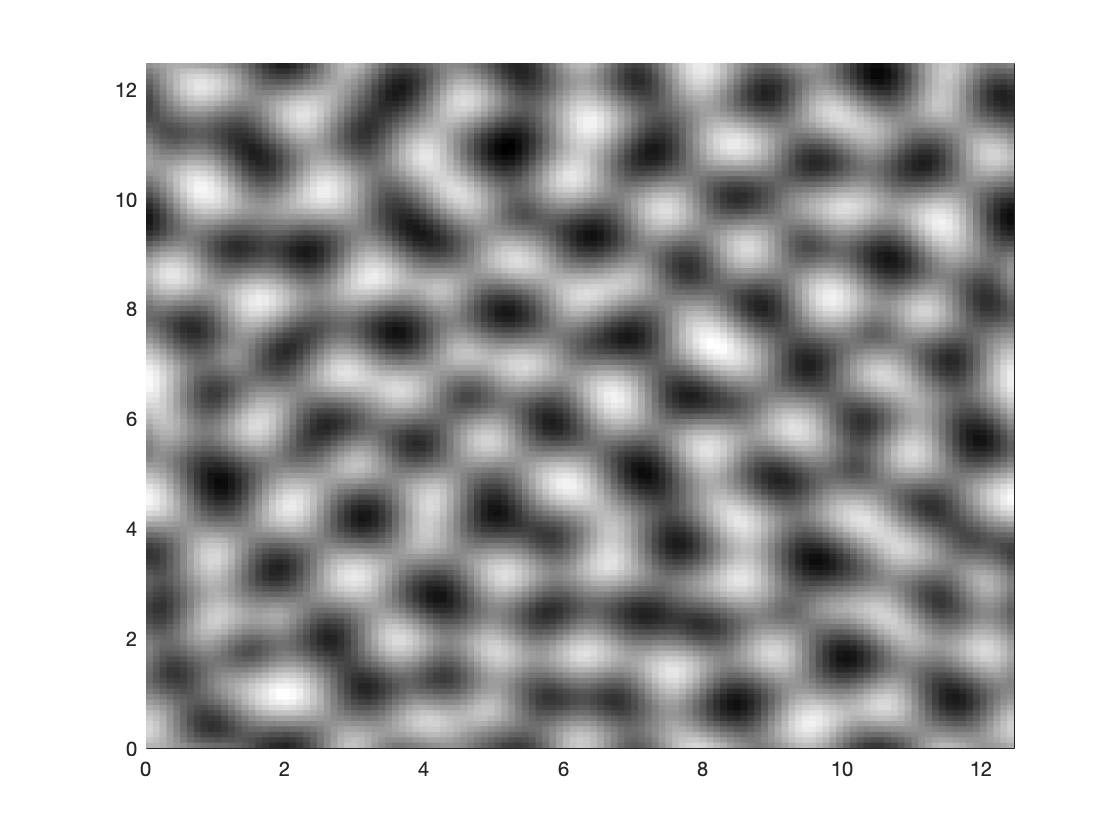}\hskip 0cm
\includegraphics[width=0.23\textwidth,clip==]{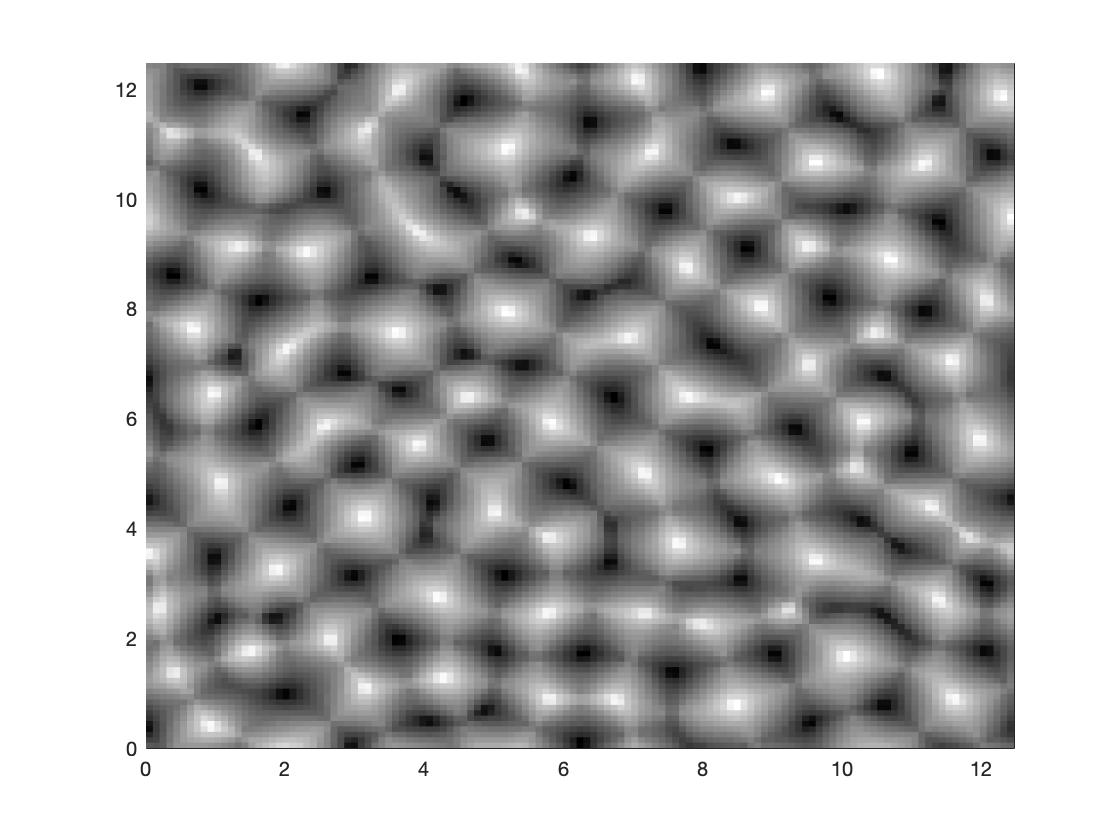}}
\subfigure[$t=100$.]{
\includegraphics[width=0.23\textwidth,clip==]{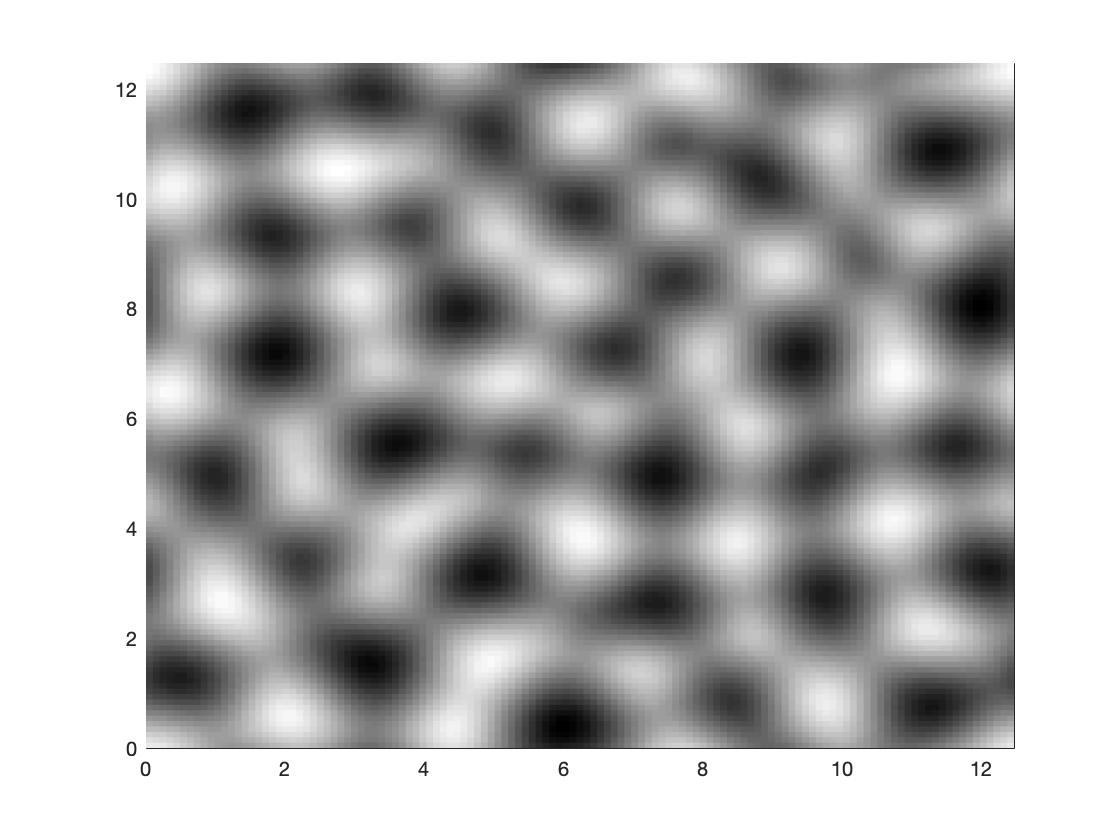}\hskip 0cm
\includegraphics[width=0.23\textwidth,clip==]{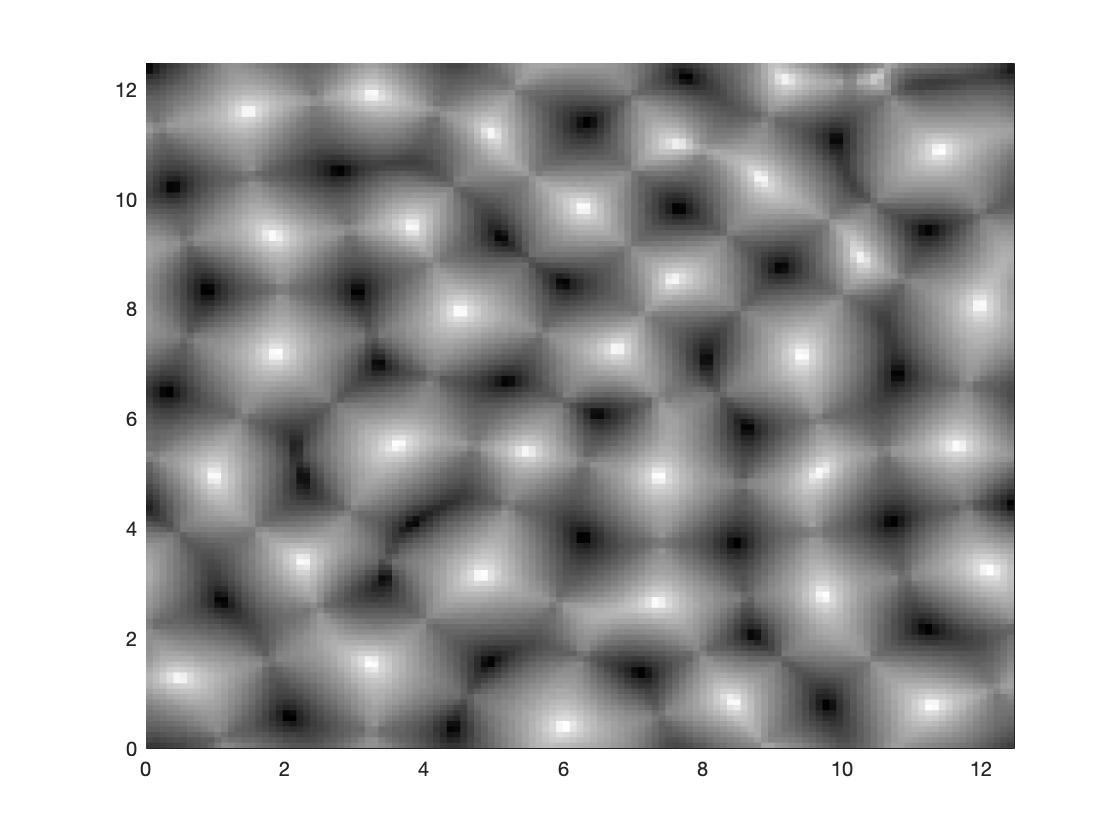}}
\subfigure[$t=600$.]{
\includegraphics[width=0.23\textwidth,clip==]{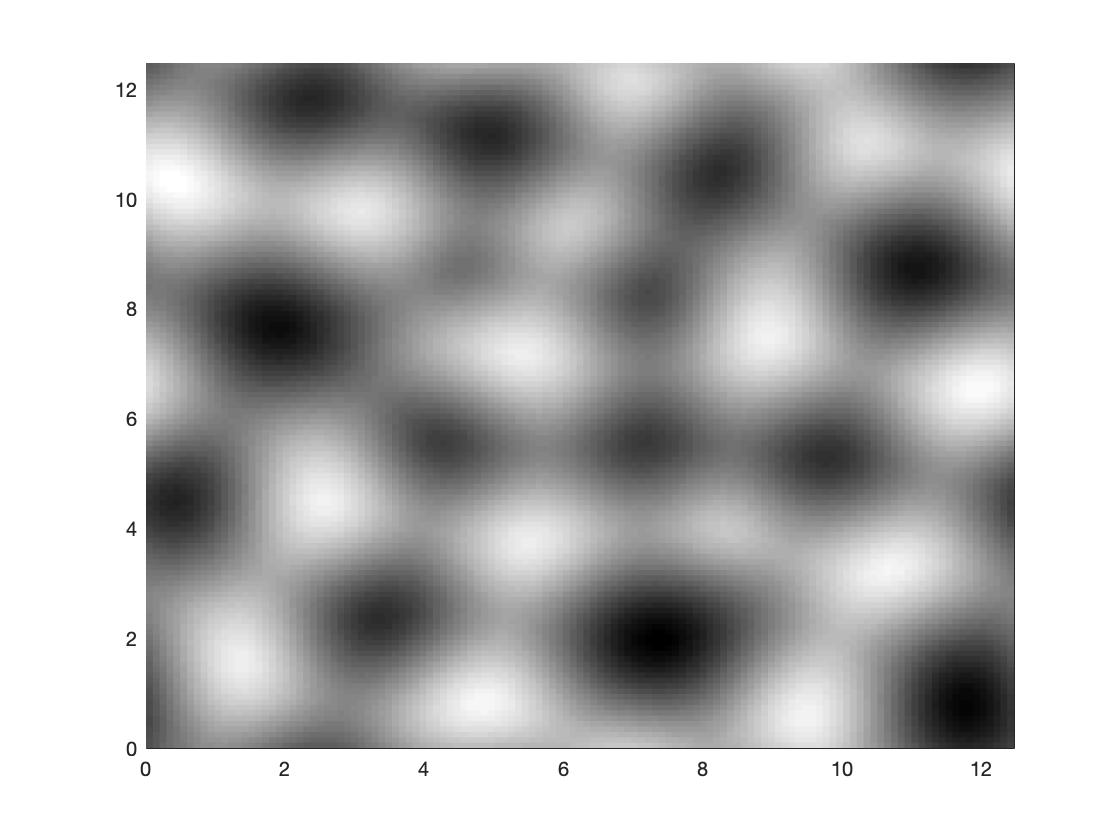}\hskip 0cm
\includegraphics[width=0.23\textwidth,clip==]{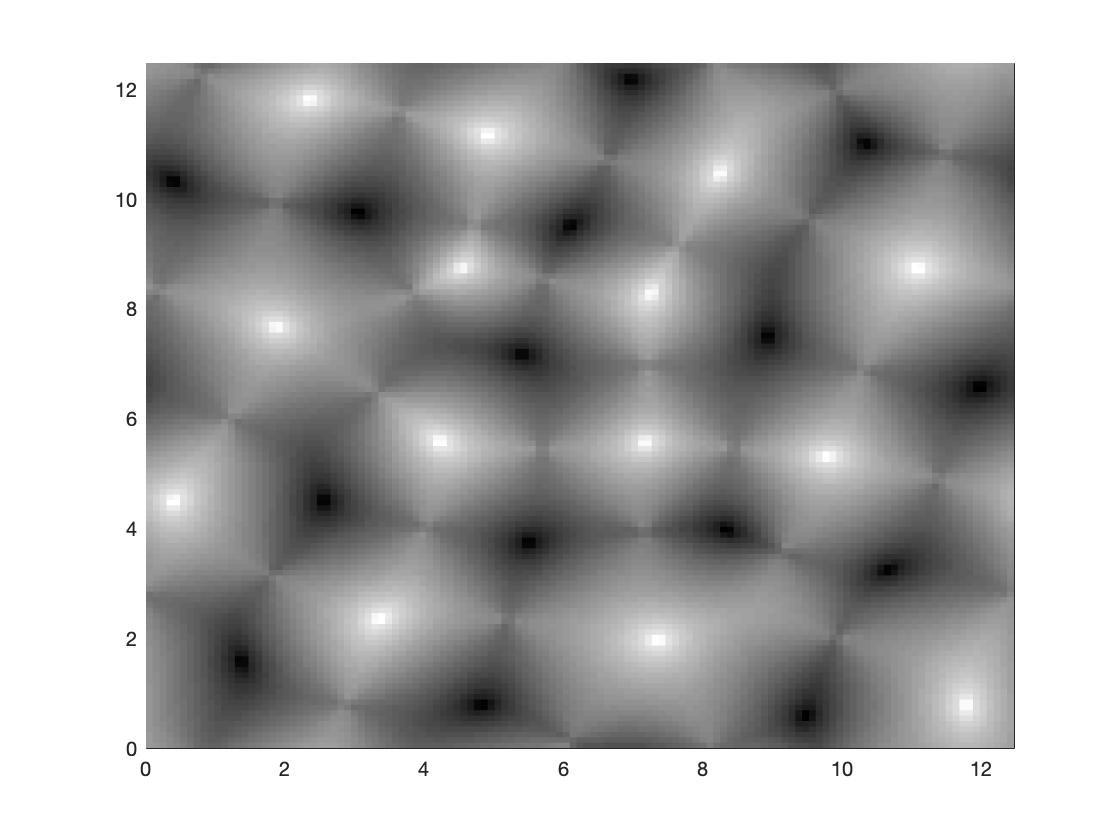}}
\caption{The isolines of the numerical solutions of the height function $\phi$ and its Laplacian $\Delta \phi$ for the model without slope selection with random initial condition. For each subfigure, the left is $\phi$ and the right is $\Delta \phi$ . Snapshots are taken at $t = 0.005 , 1 , 6 , 20 , 100 , 600$, respectively.}\label{NoSlopeCoarse}
\end{figure}

\begin{figure}[htbp]
\centering
\includegraphics[width=0.45\textwidth,clip==]{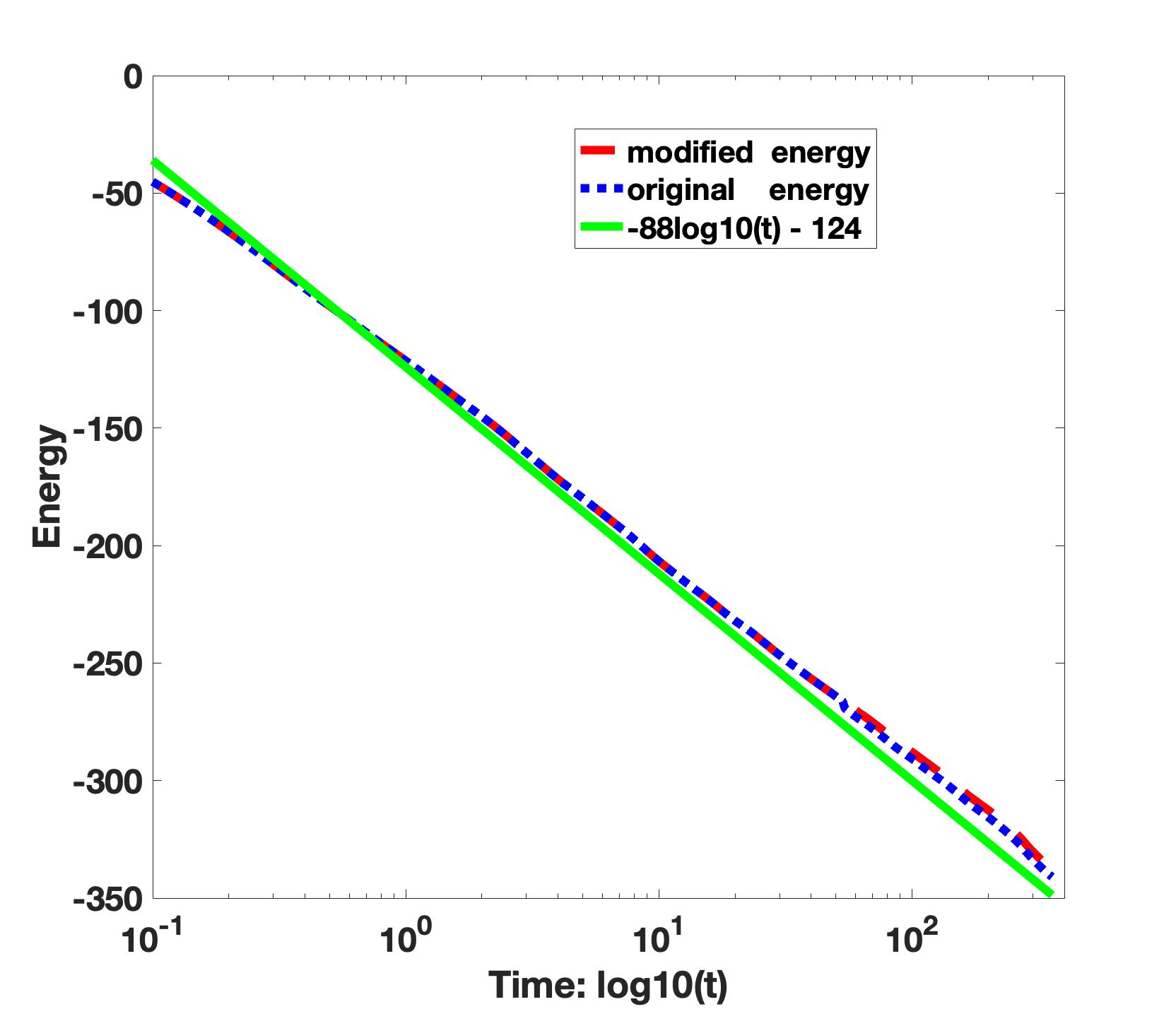}
\caption{The coarse dynamic  evolution of energy  with time for MBE model without slope selection.  }\label{MBE_coarse}
\end{figure}

%\end{comment}

\subsection{BCP model}
The (block copolymer) BCP or coupled Cahn-Hilliard system \cite{avalos2016frustrated,avalos2018transformation}  can be interpreted as a gradient flow as follows
\begin{eqnarray}
&&u_t=M_u\Delta\frac{\delta E(u,v)}{\delta u},  \label{ch2:1a}\\
&&v_t=M_v\Delta \frac{\delta E(u,v)}{\delta v},\label{ch2:3a} 
\end{eqnarray}
 with
 the  total free energy
 \begin{equation}\label{BCP:e}
 E_{\eps_u,\eps_v}(u,v)=\int_{\Omega}\frac{\eps^2_u}{2}|\Grad u|^2+\frac{\eps^2_v}{2}|\Grad v|^2+W(u,v)+\frac{\sigma}{2}|(-\Delta)^{-\frac 12}(v-\overline v)|^2d\bx,
 \end{equation}
  where
$$W(u,v)=\frac{(u^2-1)^2}{4}+\frac{(v^2-1)^2}{4}+\alpha uv+\beta uv^2+\gamma u^2v.$$
Indeed, one can easily check that 
$$\frac{\delta E(u,v)}{\delta u}=-\eps_u^2\Delta u +\frac{\delta W}{\delta u}=\mu_u,\quad \frac{\delta E(u,v)}{\delta v}=-\eps_u^2\Delta v +\frac{\delta W}{\delta v}-\sigma\Delta^{-1}(v-\overline v)=\mu_v.$$
The coupled Cahn-Hilliard equation describes a BCP particle is surrounded with homopolymer.  The order parameter $u$ describes these two components in the interval $[-1,1]$,   $-1$ represents the homopolymer rich domain and $1$ is the BCP-rich domain. Order parameter $v$  describes micro separation inside the BCP
domain which also acquires values from interval $[-1,1]$ with the end points corresponding to A-type BCP and B-type BCP.  $\eps_u$ and $\eps_v$ are corresponding diffusive interface parameters and $M_u$ and $M_v$ are mobility constants.

Now  we construct a second-order  numerical scheme by using G-SAV approach and  define  a new variable $r=G(\int_{\Omega}W(u,v)d\bx)$ where $G$ is a reversible function. Then the new total free energy is rewritten as 
\begin{equation}
 E_{\eps_u,\eps_v}(u,v)=\int_{\Omega}\frac{\eps^2_u}{2}|\Grad u|^2+\frac{\eps^2_v}{2}|\Grad v|^2+\frac{\sigma}{2}|(-\Delta)^{-\frac 12}(v-\overline v)|^2d\bx+G^{-1}\{G(\int_{\Omega}W(u,v)d\bx)\}.
 \end{equation}
As in the previous section, we can construct a second-order  G-SAV  scheme based on BDF2 version  for the above system.  
 Assuming that $u^{n-1}$, $u^{n}$ and $v^{n-1}$, $v^{n}$ are known, we find
$u^{n+1}$ and  $v^{n+1}$ as follows:
\begin{eqnarray}
&&\frac{3u^{n+1}-4u^n+u^{n-1}}{2\delta t}=M_u\Delta\mu^{n+1}_u, \label{ch:bdf2:1}\\
&&\mu^{n+1}_u=-\eps_u^2\Delta u^{n+1} +(\frac{\delta W}{\delta u})^{\dagger,n}\frac{r^{n+1}}{G(\int_{\Omega}W(u^{\dagger,n},v^{\star,n})d\bx)},\label{ch:bdf2:2}\\
&&\frac{3v^{n+1}-4v^n+v^{n-1}}{2\delta t}=M_v\Delta\mu^{n+1}_v,\label{ch:bdf2:3} \\
&&\mu^{n+1}_v=-\eps_v^2\Delta v^{n+1} +(\frac{\delta W}{\delta v})^{\dagger,n}\frac{r^{n+1}}{G(\int_{\Omega}W(u^{\dagger,n},v^{\dagger,n})d\bx)}\nonumber\\&& \hskip 1cm -\sigma\Delta^{-1}(v^{n+1}-\overline v)\label{ch:bdf2:4},\\
&&3G^{-1}(r^{n+1})-4G^{-1}(r^{n})+G^{-1}(r^{n-1})\label{ch:bdf2:5},\\
&&\hskip 1cm = \frac{r^{n+1}}{G(\int_{\Omega}W(u^{\dagger,n},v^{\dagger,n})d\bx)}\{((\frac{\delta W}{\delta u})^{\dagger,n},3u^{n+1}-4u^n+u^{n-1})\nonumber\\&&\hskip 1cm+((\frac{\delta W}{\delta v})^{\dagger,n},3v^{n+1}-4v^n+v^{n-1})\}.\nonumber
\end{eqnarray}
Where $f^{\dagger,n}=2f^n-f^{n-1}$ for any function $f$, and the boundary conditions are periodic. 
The G-SAV scheme \eqref{ch:bdf2:1}-\eqref{ch:bdf2:5} is energy stable and  can be solved following the  G-SAV  scheme \eqref{scheme:mod:three:CH} with multi-components.

Now we make  numerical experiments  to simulate the annealing process  \cite{avalos2016frustrated,avalos2018transformation}  of block copolymer  and  detect its  morphology transformations.  Phase variables for coupled Cahn-Hilliard equation  are chosen as 
\begin{equation}
\eps_u=0.075 \quad \eps_v=0.05 \quad \sigma =10 \quad \alpha=0.1 \quad \beta =-0.75 \quad \gamma=0,
\end{equation}
and  domain is set to be  $[-1,1]$. 
The initial conditions are taken as a randomly perturbed concentration field as follows:
\begin{eqnarray}\label{initial}
&&u(t=0)=\,{\rm rand}(x,y),\\
&&v(t=0)=\,{\rm rand}(x,y),
\end{eqnarray}
where the ${\rm rand}(x,y)$ is a uniformly distributed random function in $[-1,1]^2$ with zero mean.
Assuming that the yellow bulk presents  A-BCP particle and bule bulk presents  as B-BCP particle.  The numerical solutions in Fig.\,\ref{BCP} and Fig.\,\ref{BCP_2}  are computed by G-SAV scheme \eqref{ch:bdf2:1}-\eqref{ch:bdf2:5} with $G=e^{x/c}$ and $G=\sqrt{x+c}$.
From this two figures, it is observed that the contours of numerical solution are indistinguishable by using different functions $G$.  Phase variable $u$ presents  clearly the confined surface  for BCP particles. While variable $v$ describes the dynamic process of morphological transformation for BCP particles.
From Fig.\,\ref{BCP} and Fig.\,\ref{BCP_2}  the striped ellipsoids gradually  appears  at final stage as Fig.\,\ref{BCP}.  The steady morphology coincides  with the experimental results depicted  in Fig.\,\ref{refer_1} which also are observed in \cite{avalos2016frustrated,avalos2018transformation}.

\begin{figure}
\centering
\includegraphics[width=0.7\textwidth,clip==]{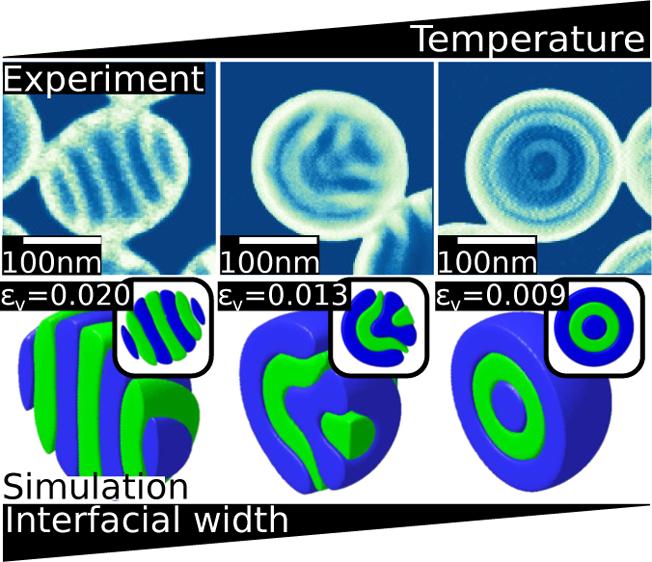}\hskip 0cm
\caption{Experimental results  and simulations at various temperature in \cite{avalos2016frustrated,avalos2018transformation}.}
\label{refer_1}
\end{figure}

\begin{figure}
\centering
\subfigure[$\bu:t=0.2$.]{\includegraphics[width=0.22\textwidth,clip==]{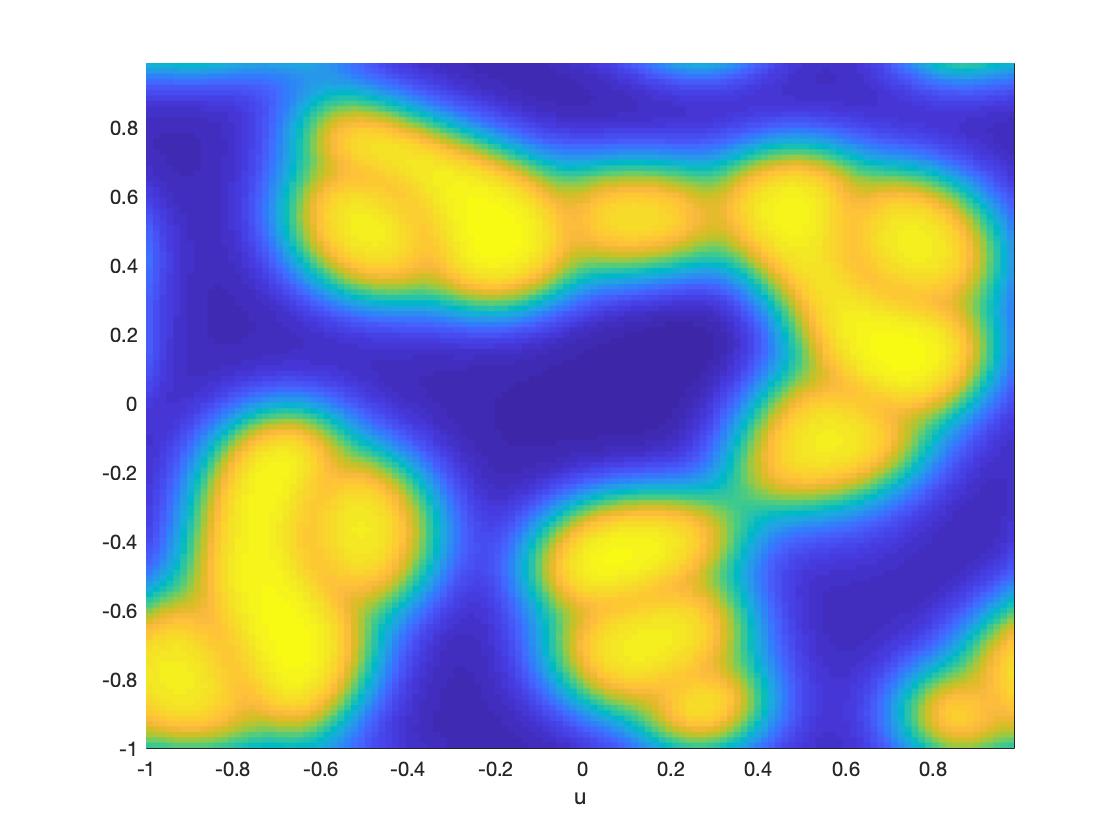}\hskip 0cm}
\subfigure[$\bv:t=0.2$.]{\includegraphics[width=0.22\textwidth,clip==]{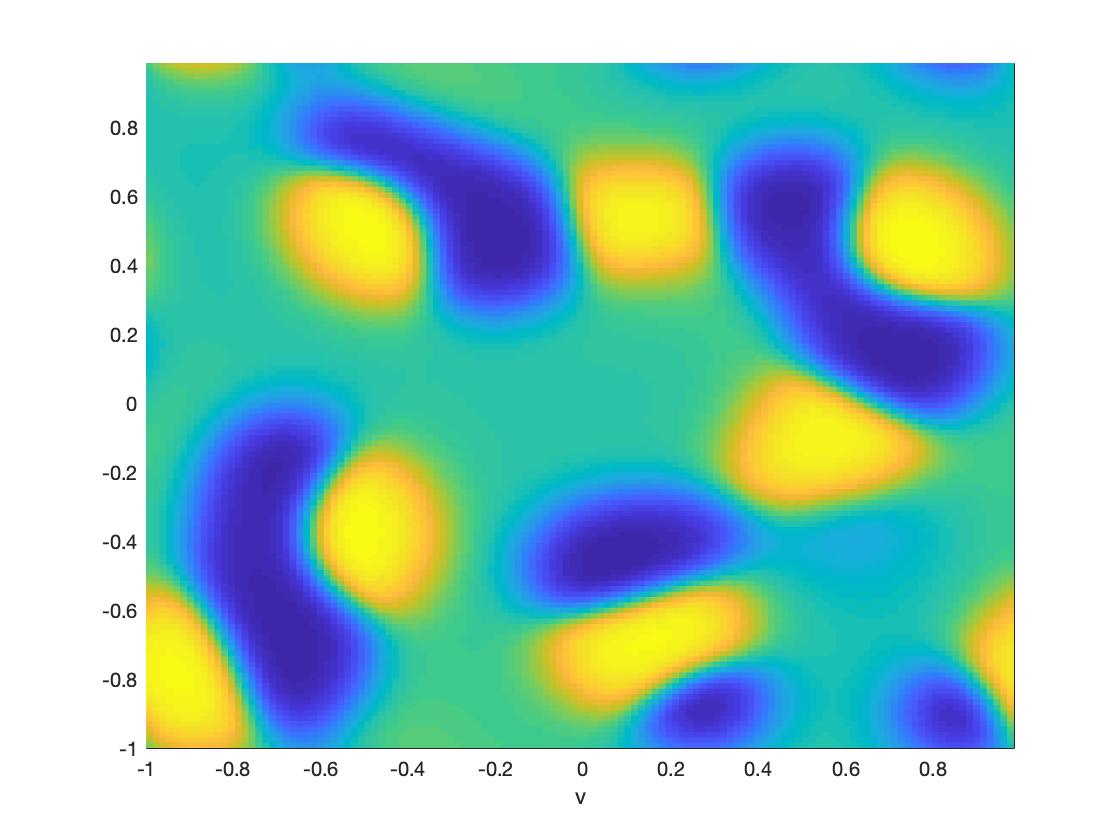}\hskip 0cm}
\subfigure[$\bu:t=0.5$.]{\includegraphics[width=0.22\textwidth,clip==]{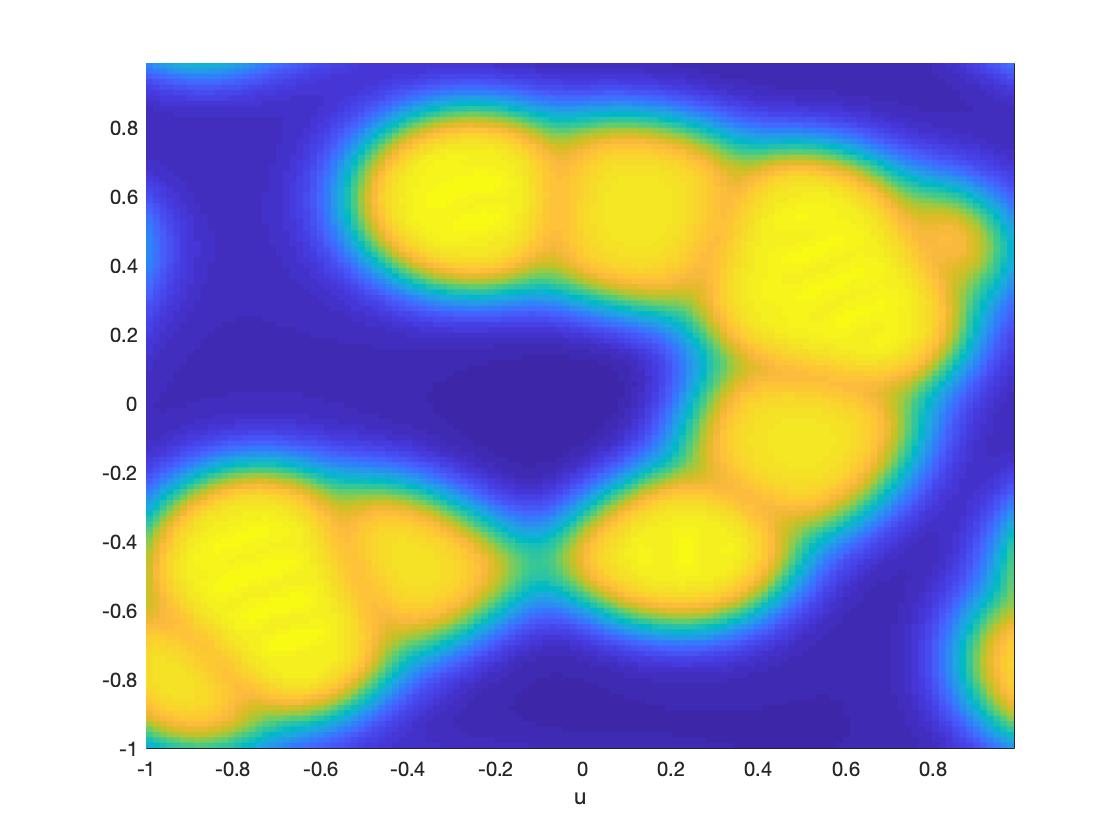}\hskip 0cm}
\subfigure[$\bv:t=0.5$.]{\includegraphics[width=0.22\textwidth,clip==]{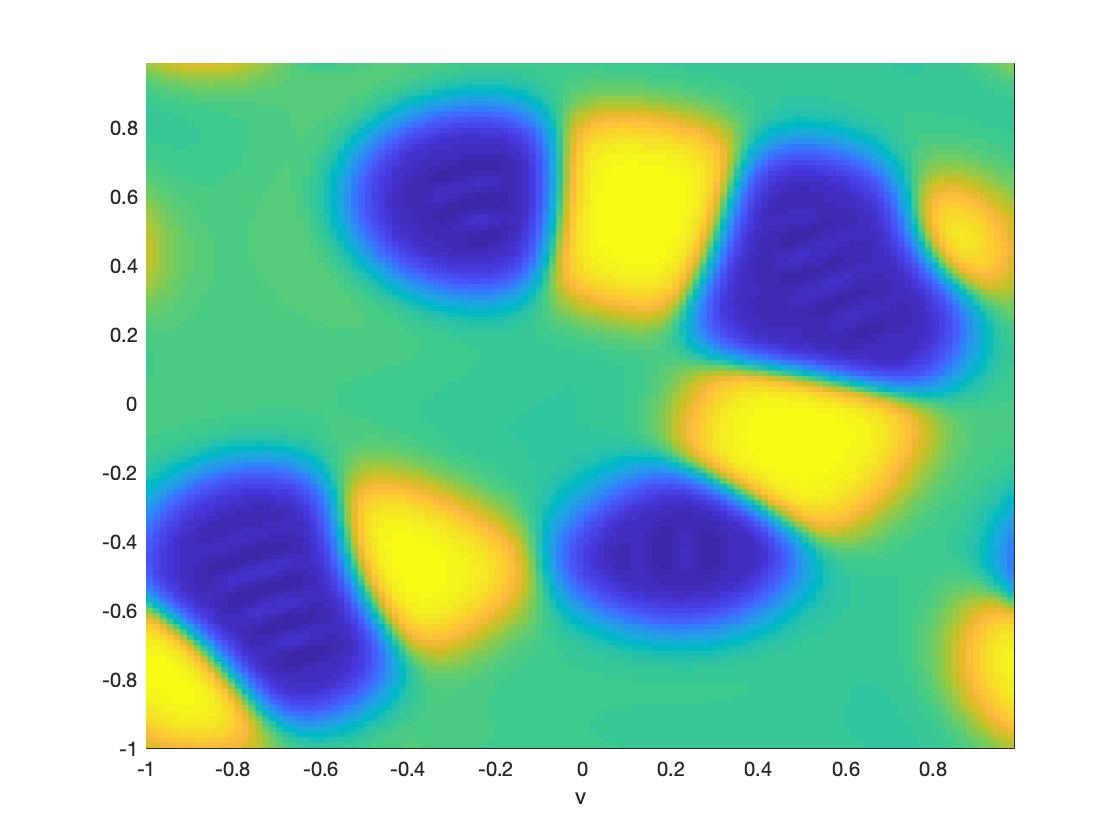}\hskip 0cm}
\subfigure[$\bu:t=1$.]{\includegraphics[width=0.22\textwidth,clip==]{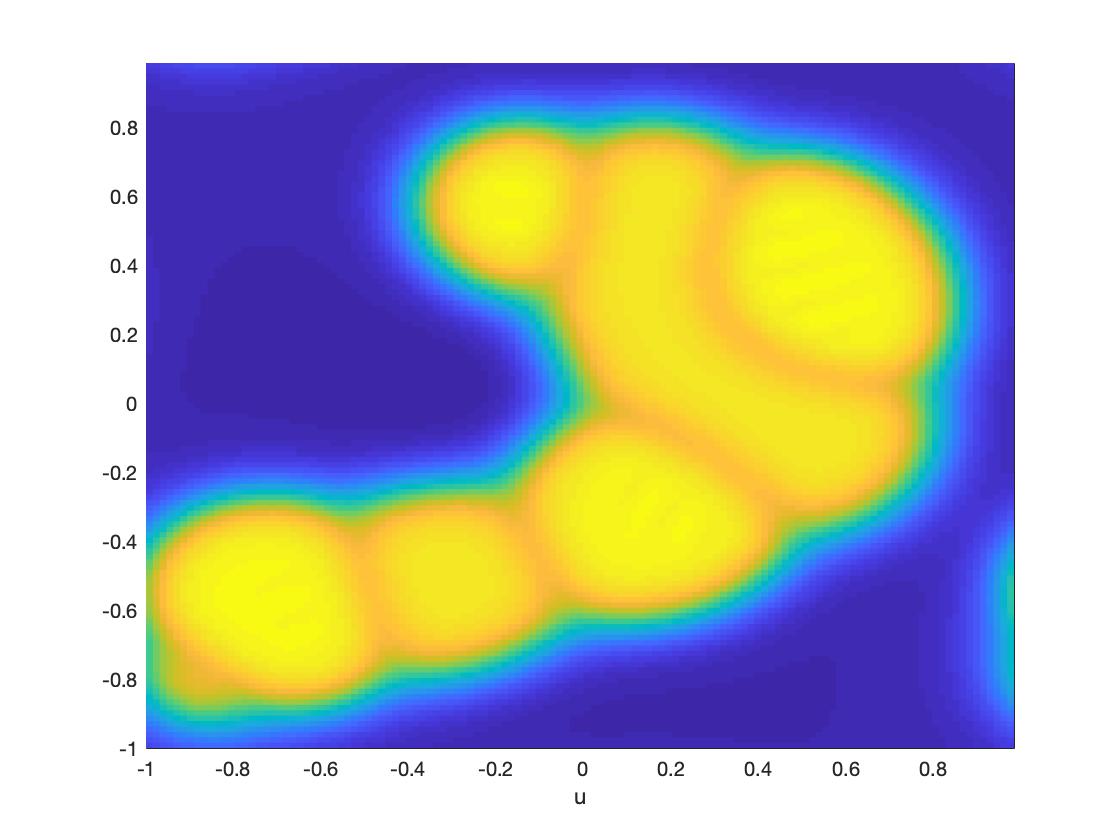}\hskip 0cm}
\subfigure[$\bv:t=1$.]{\includegraphics[width=0.22\textwidth,clip==]{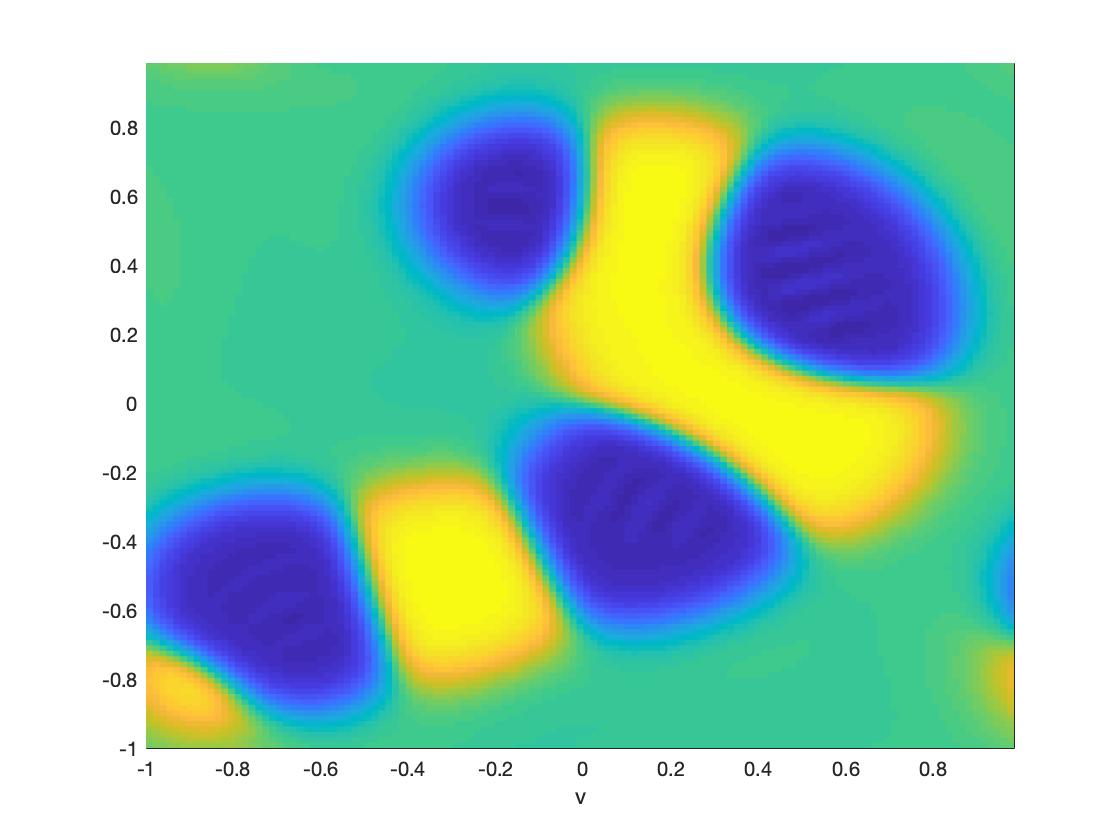}\hskip 0cm}
\subfigure[$\bu:t=2$.]{\includegraphics[width=0.22\textwidth,clip==]{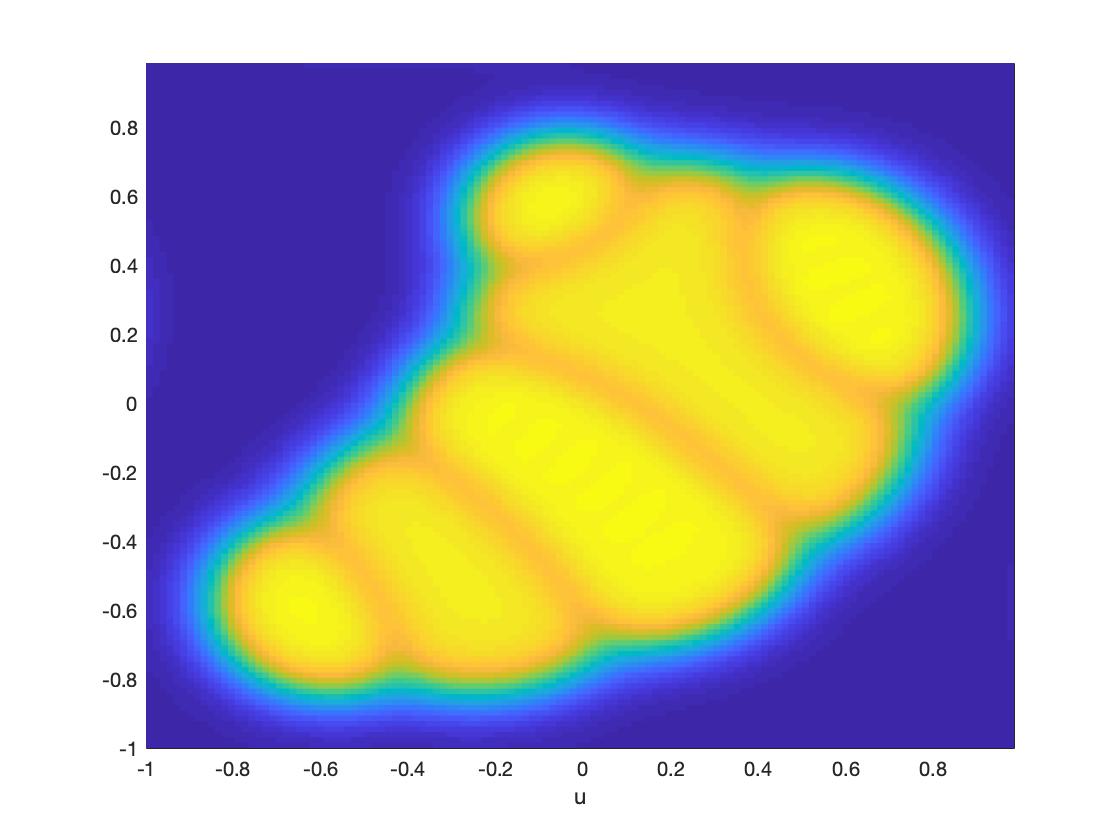}\hskip 0cm}
\subfigure[$\bv:t=2$.]{\includegraphics[width=0.22\textwidth,clip==]{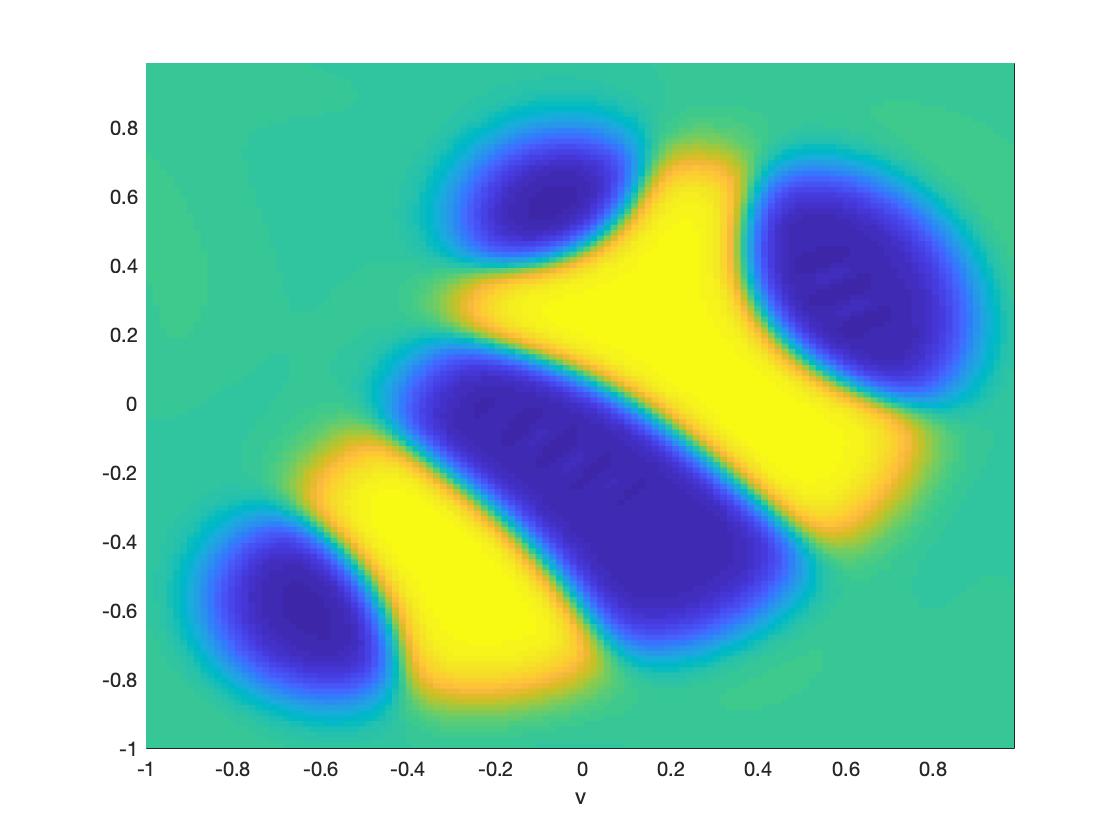}\hskip 0cm}
\subfigure[$\bu:t=3$.]{\includegraphics[width=0.22\textwidth,clip==]{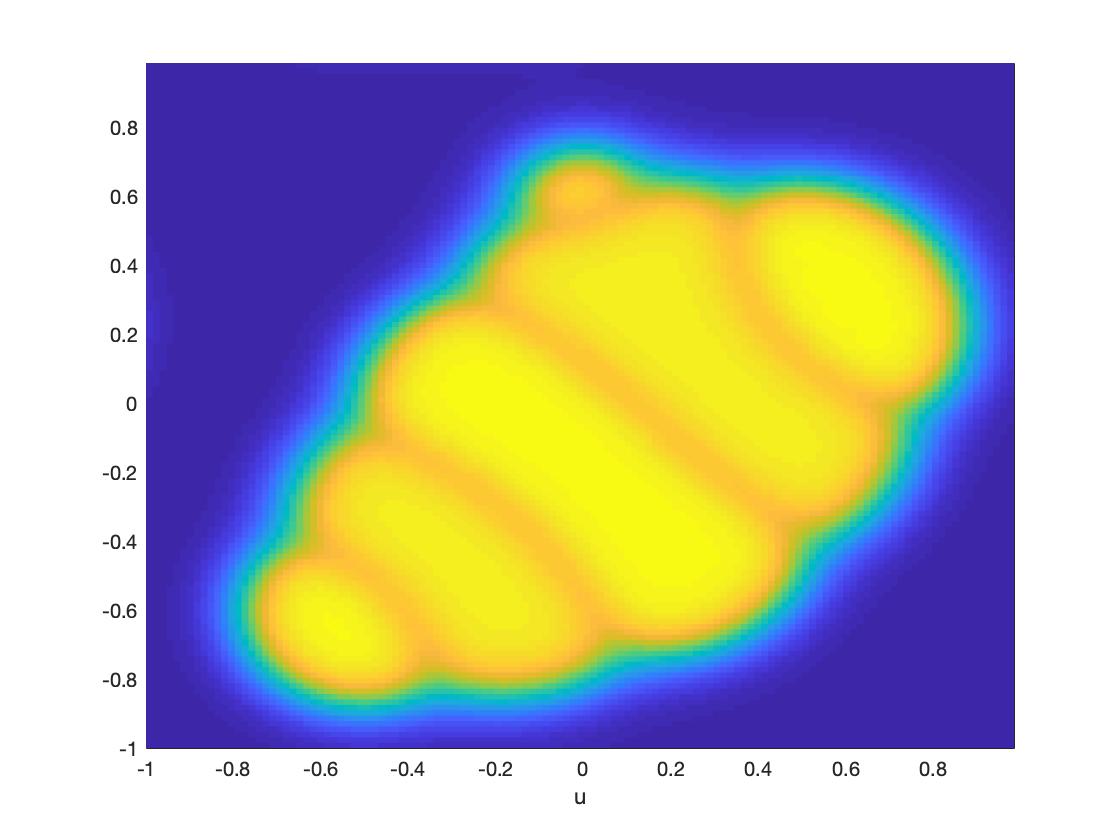}\hskip 0cm}
\subfigure[$\bv:t=3$.]{\includegraphics[width=0.22\textwidth,clip==]{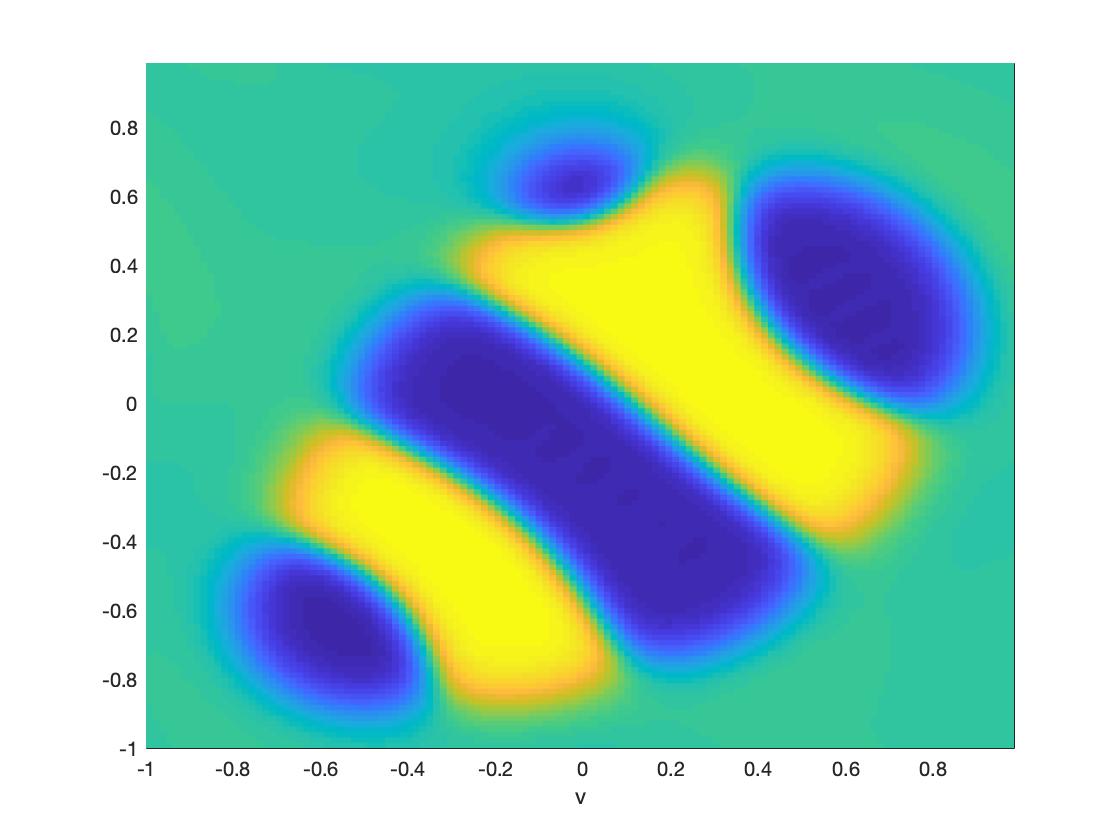}\hskip 0cm}
\subfigure[$\bu:t=4$.]{\includegraphics[width=0.22\textwidth,clip==]{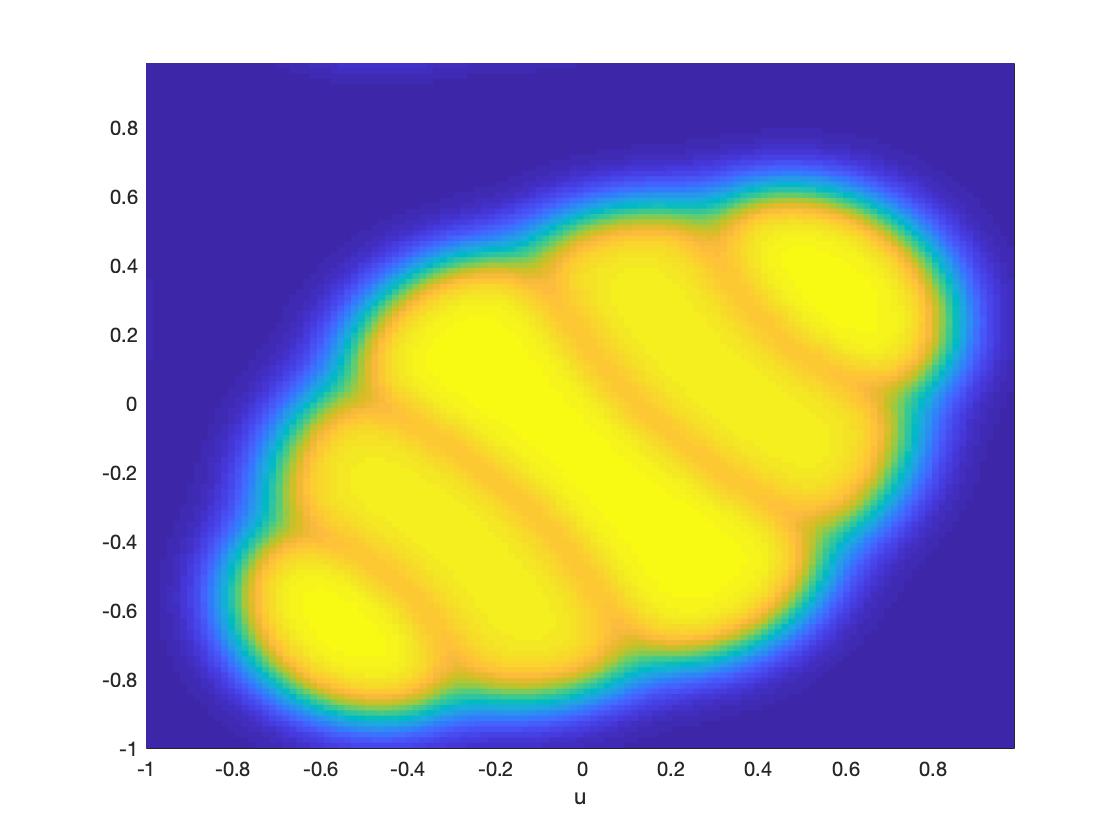}\hskip 0cm}
\subfigure[$\bv:t=4$.]{\includegraphics[width=0.22\textwidth,clip==]{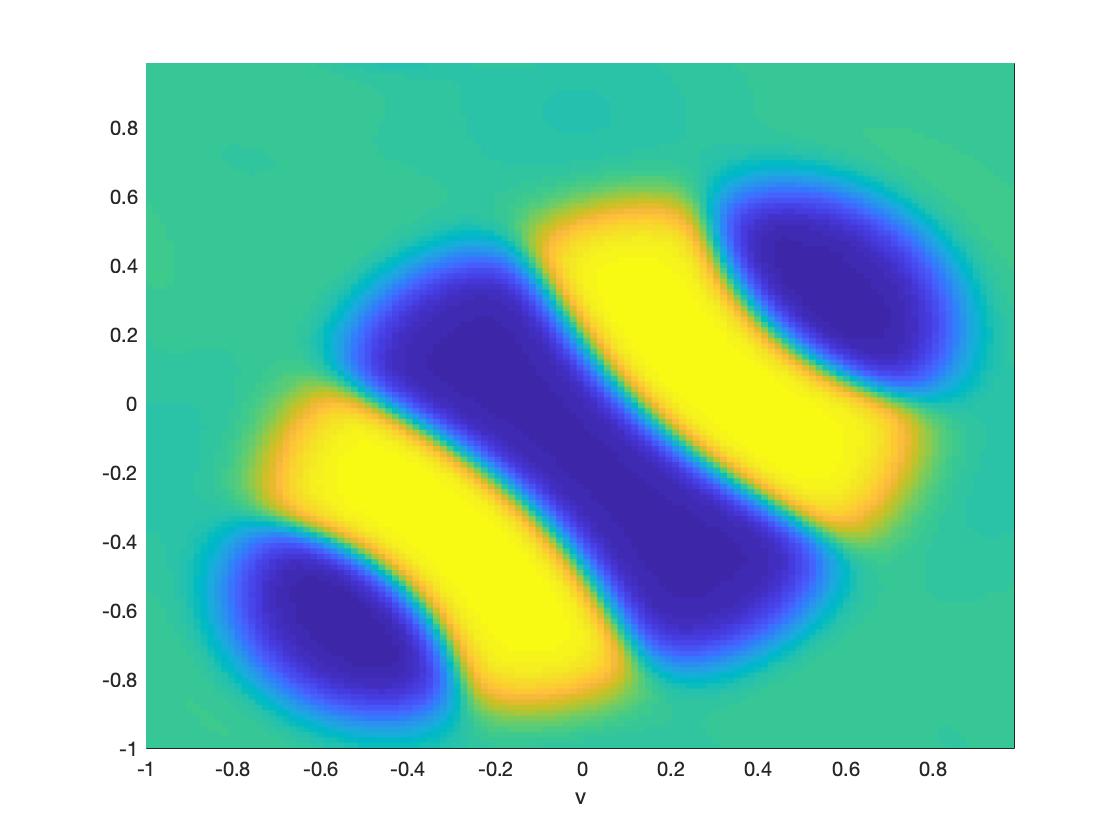}\hskip 0cm}
\caption{The 2D dynamical evolution of the phase variable $\bu,\bv$ for the Coupled-BCP model with the initial condition \eqref{initial} and $G=e^{x/c}$ with $c=10^4$.}\label{BCP}
\end{figure}

\begin{figure}
\centering
\subfigure[$\bu:t=0.2$.]{\includegraphics[width=0.22\textwidth,clip==]{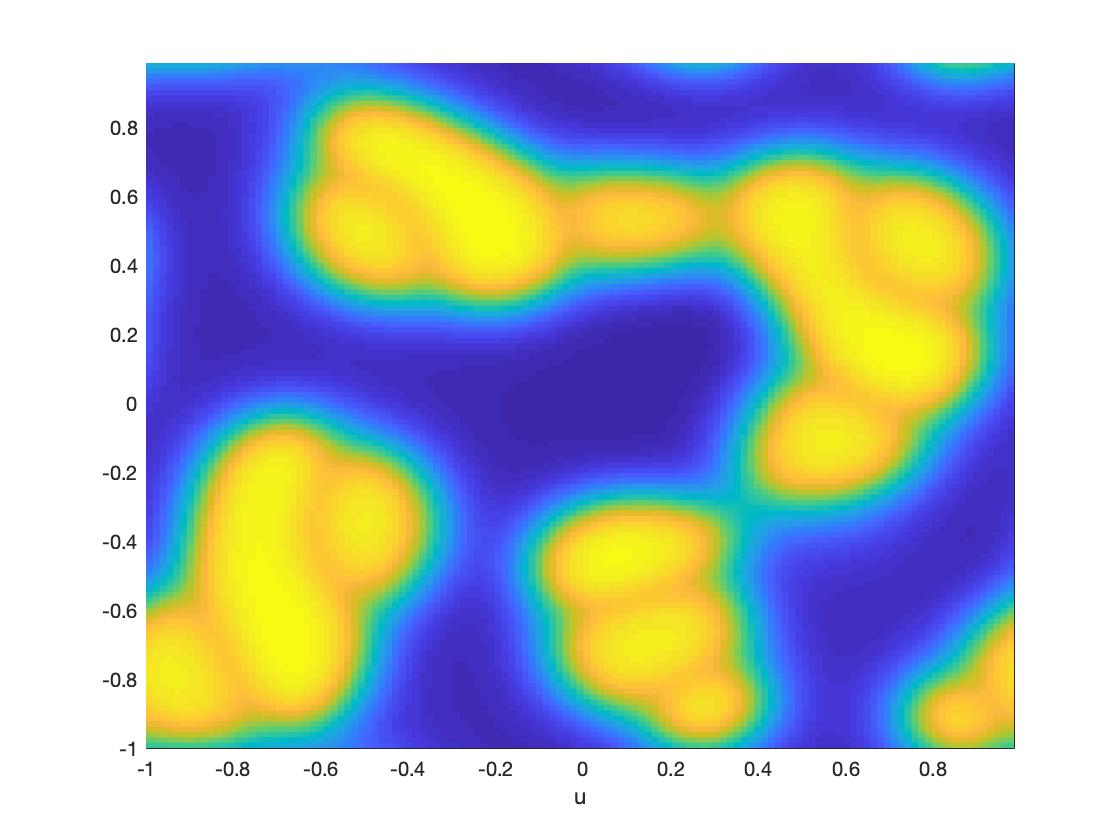}\hskip 0cm}
\subfigure[$\bv:t=0.2$.]{\includegraphics[width=0.22\textwidth,clip==]{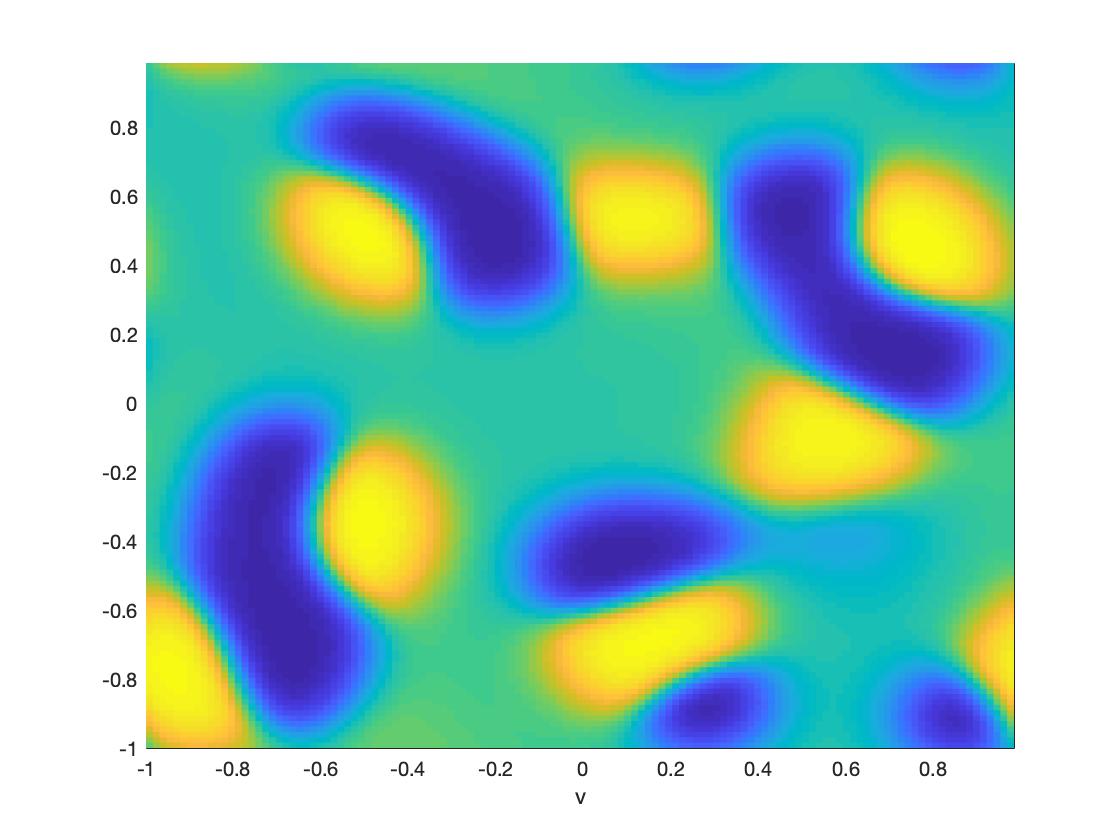}\hskip 0cm}
\subfigure[$\bu:t=0.5$.]{\includegraphics[width=0.22\textwidth,clip==]{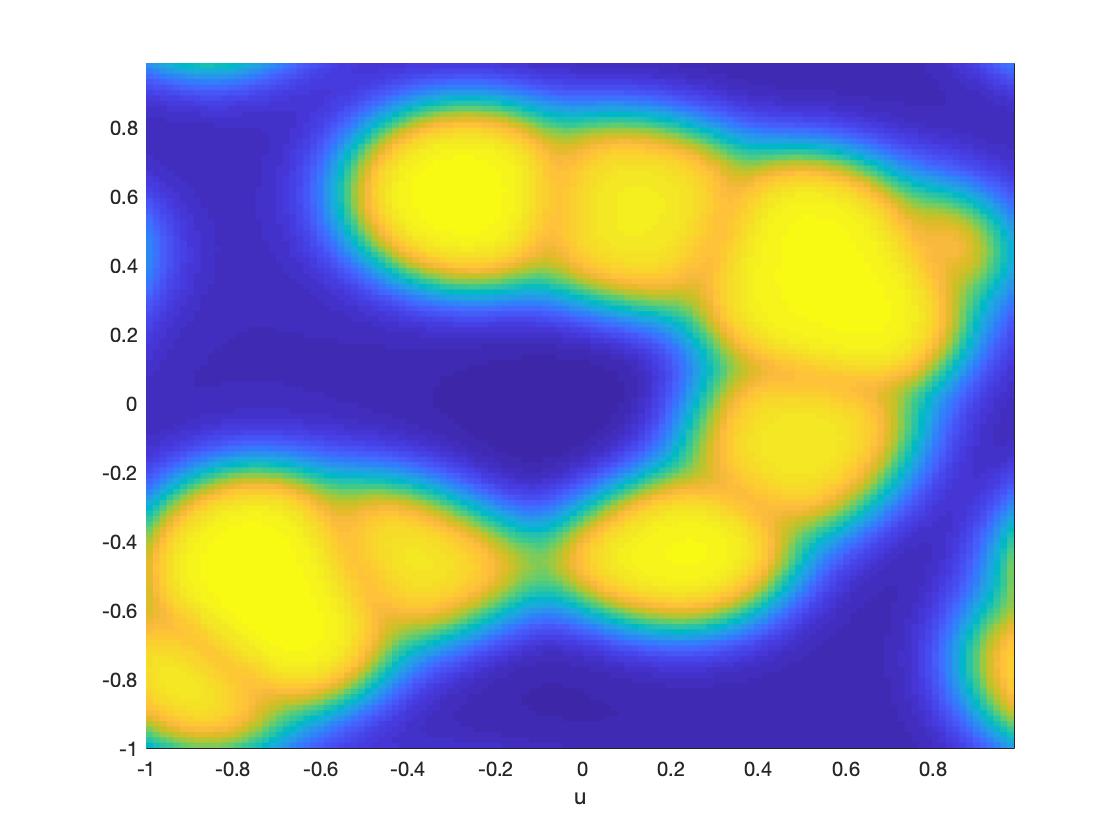}\hskip 0cm}
\subfigure[$\bv:t=0.5$.]{\includegraphics[width=0.22\textwidth,clip==]{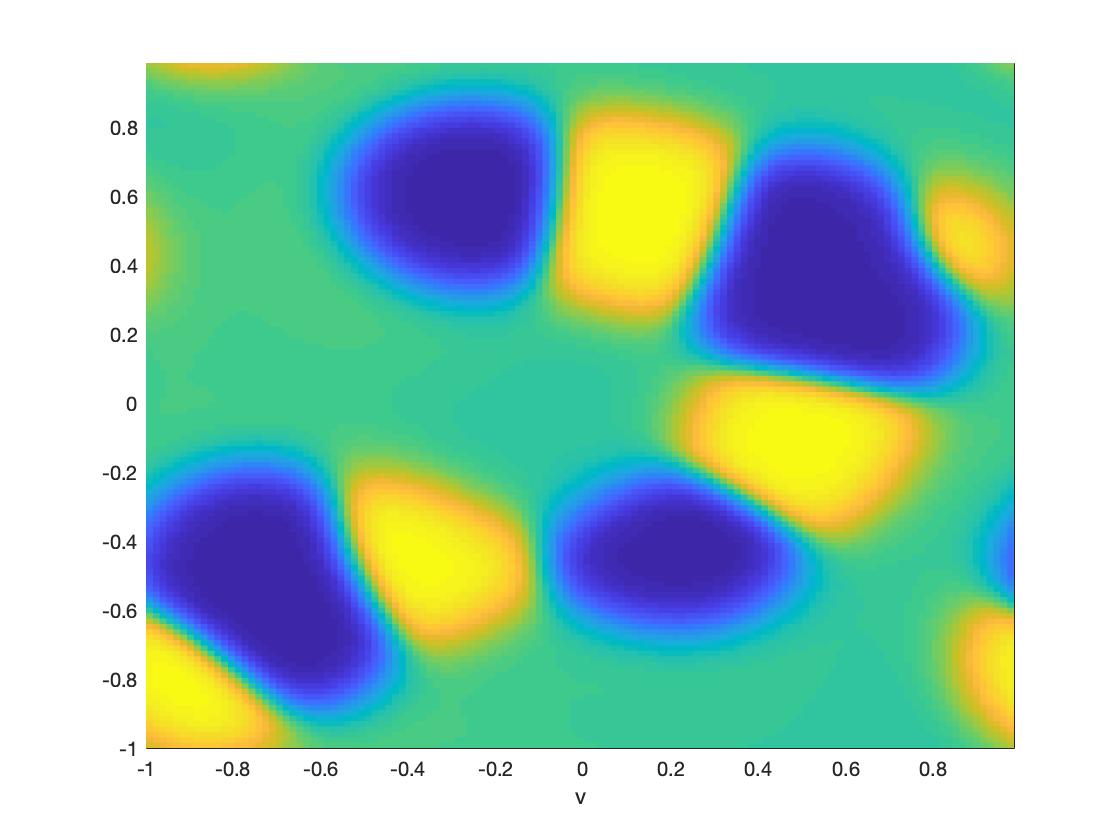}\hskip 0cm}
\subfigure[$\bu:t=1$.]{\includegraphics[width=0.22\textwidth,clip==]{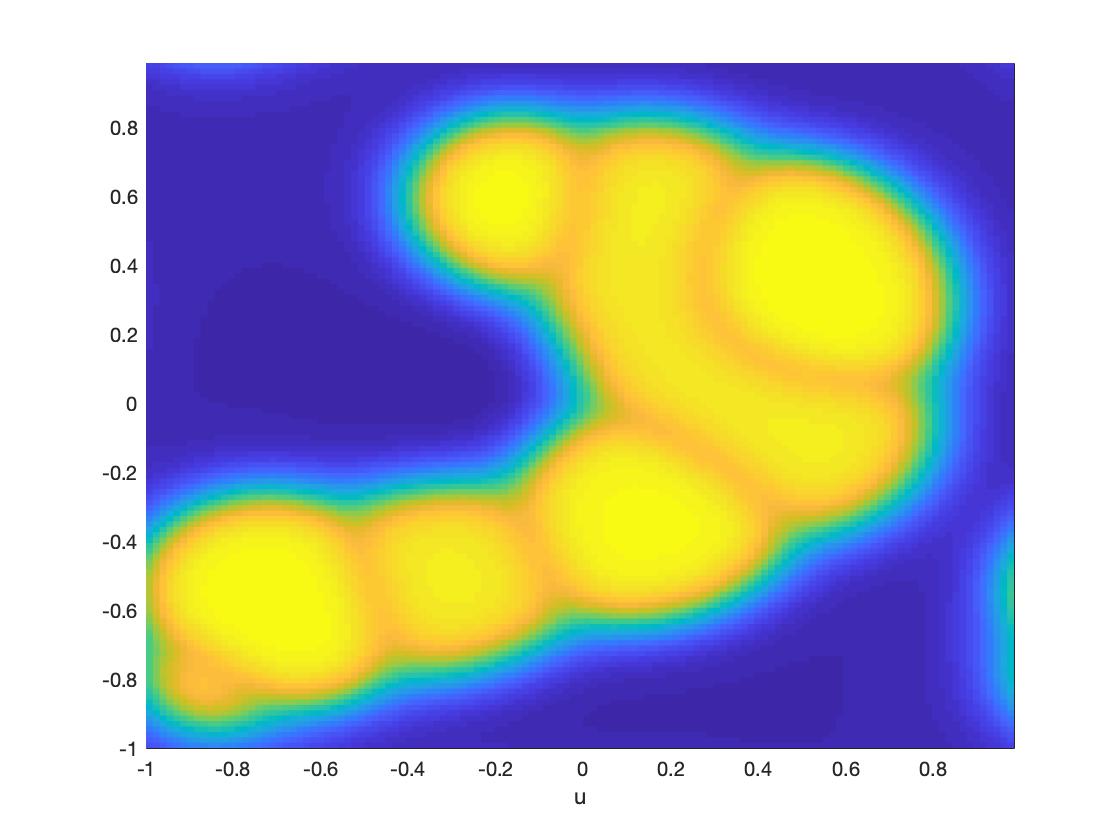}\hskip 0cm}
\subfigure[$\bv:t=1$.]{\includegraphics[width=0.22\textwidth,clip==]{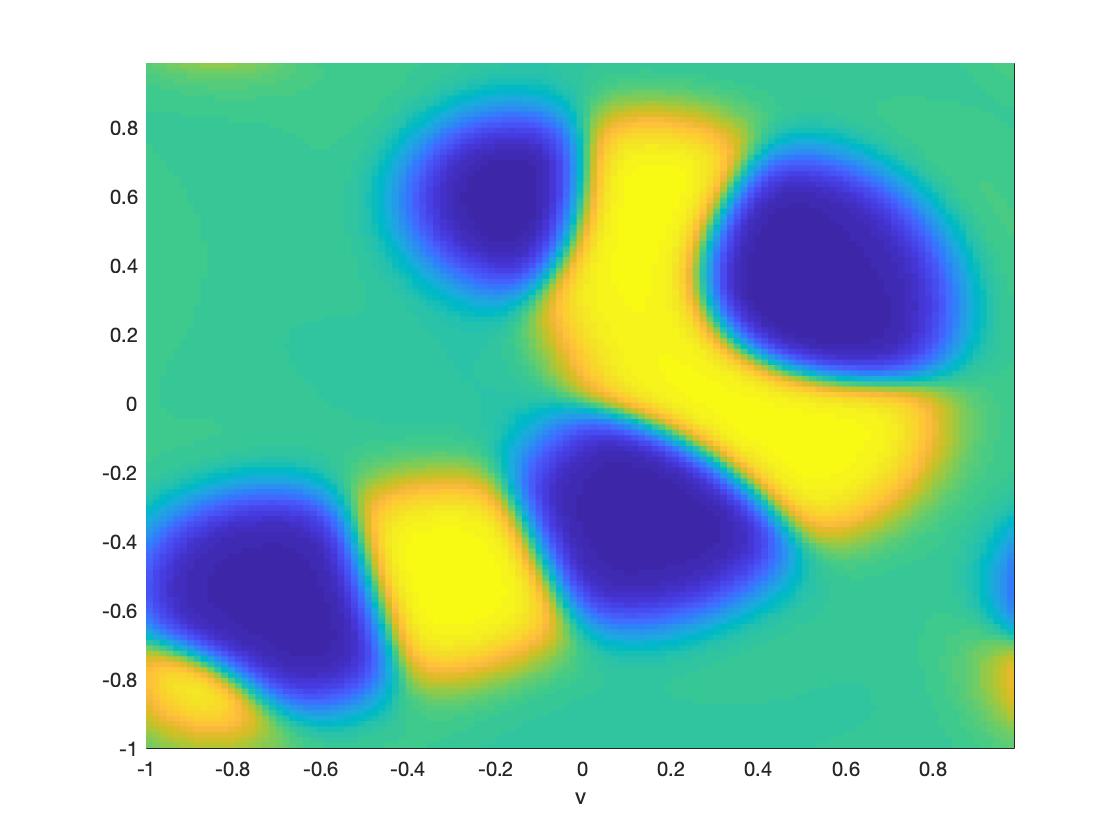}\hskip 0cm}
\subfigure[$\bu:t=2$.]{\includegraphics[width=0.22\textwidth,clip==]{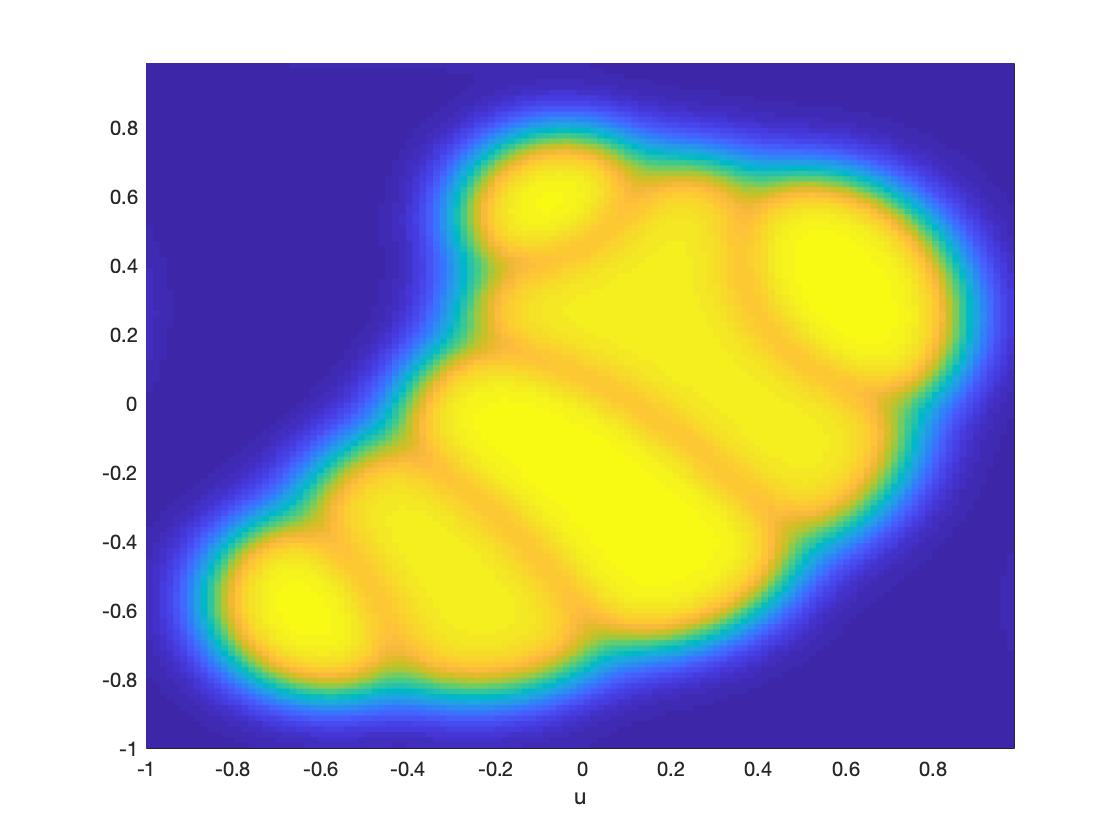}\hskip 0cm}
\subfigure[$\bv:t=2$.]{\includegraphics[width=0.22\textwidth,clip==]{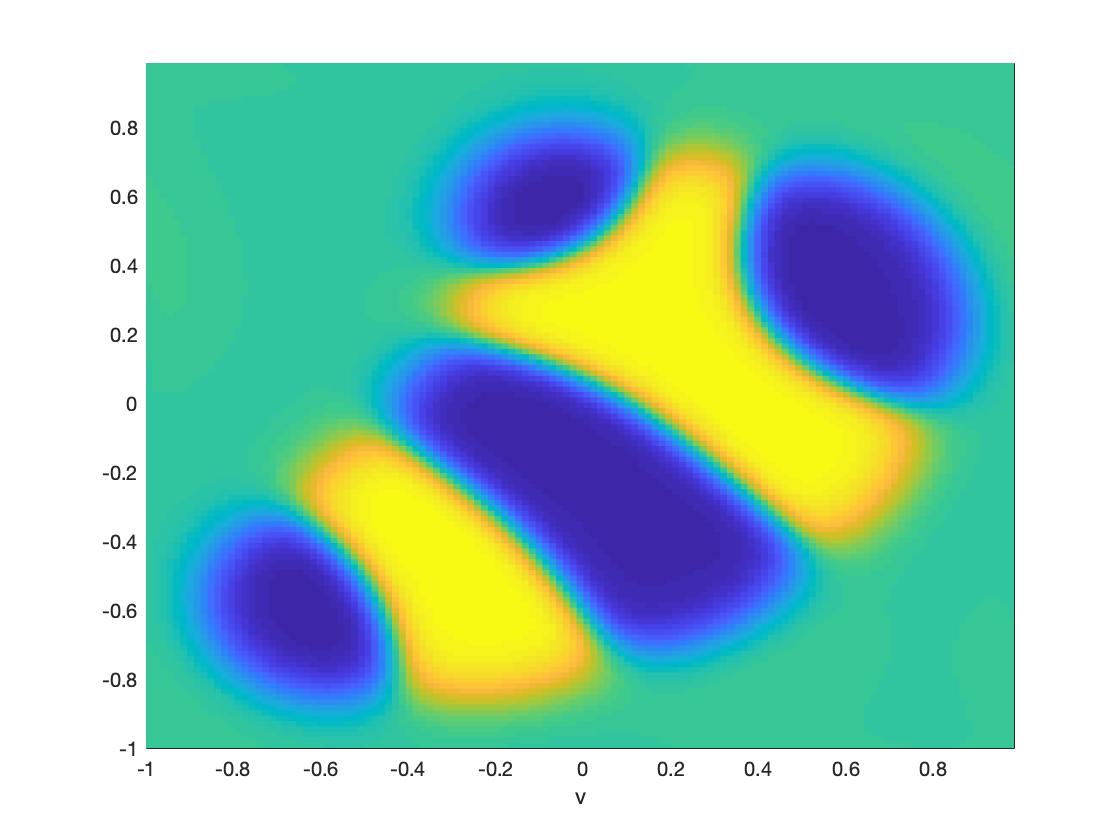}\hskip 0cm}
\subfigure[$\bu:t=3$.]{\includegraphics[width=0.22\textwidth,clip==]{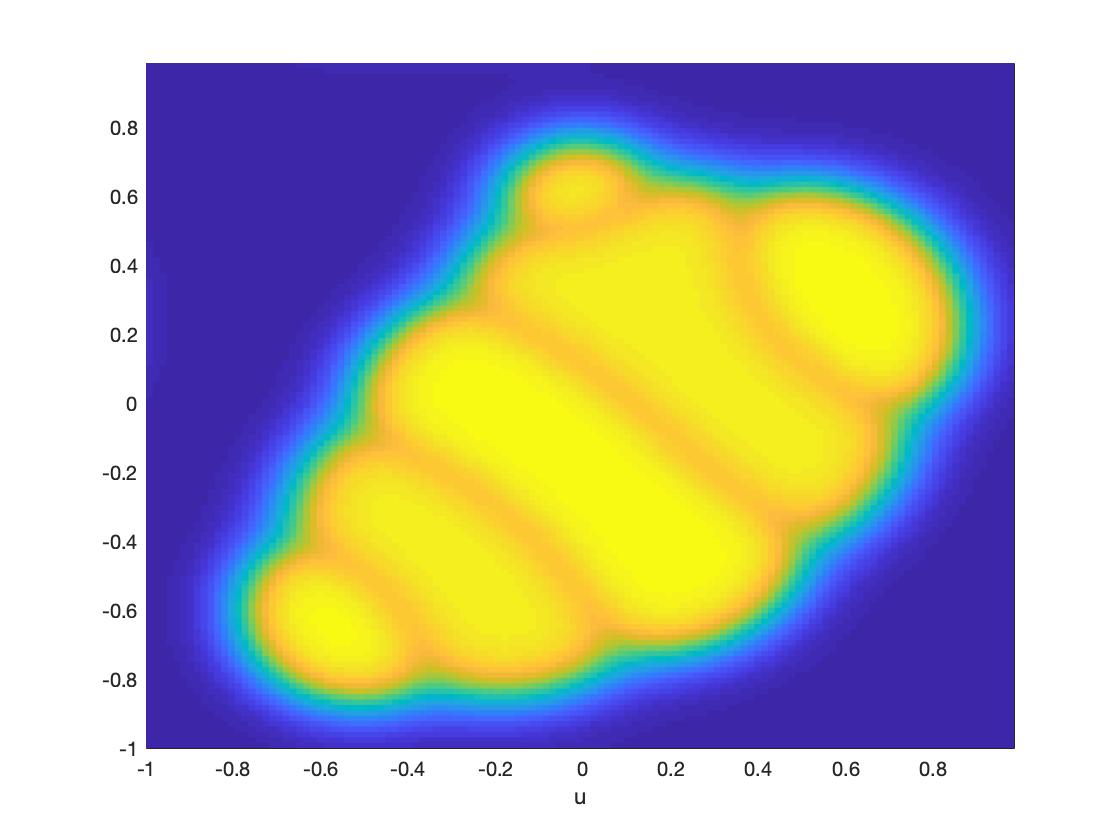}\hskip 0cm}
\subfigure[$\bv:t=3$.]{\includegraphics[width=0.22\textwidth,clip==]{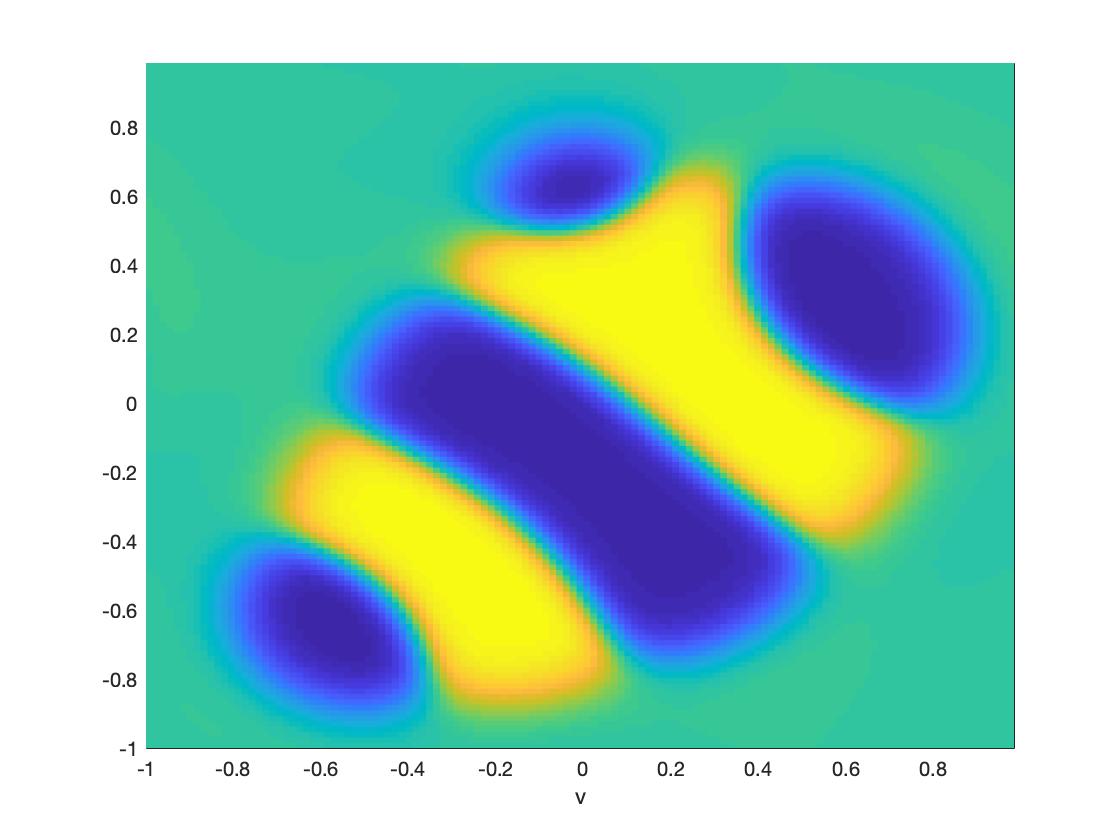}\hskip 0cm}
\subfigure[$\bu:t=4$.]{\includegraphics[width=0.22\textwidth,clip==]{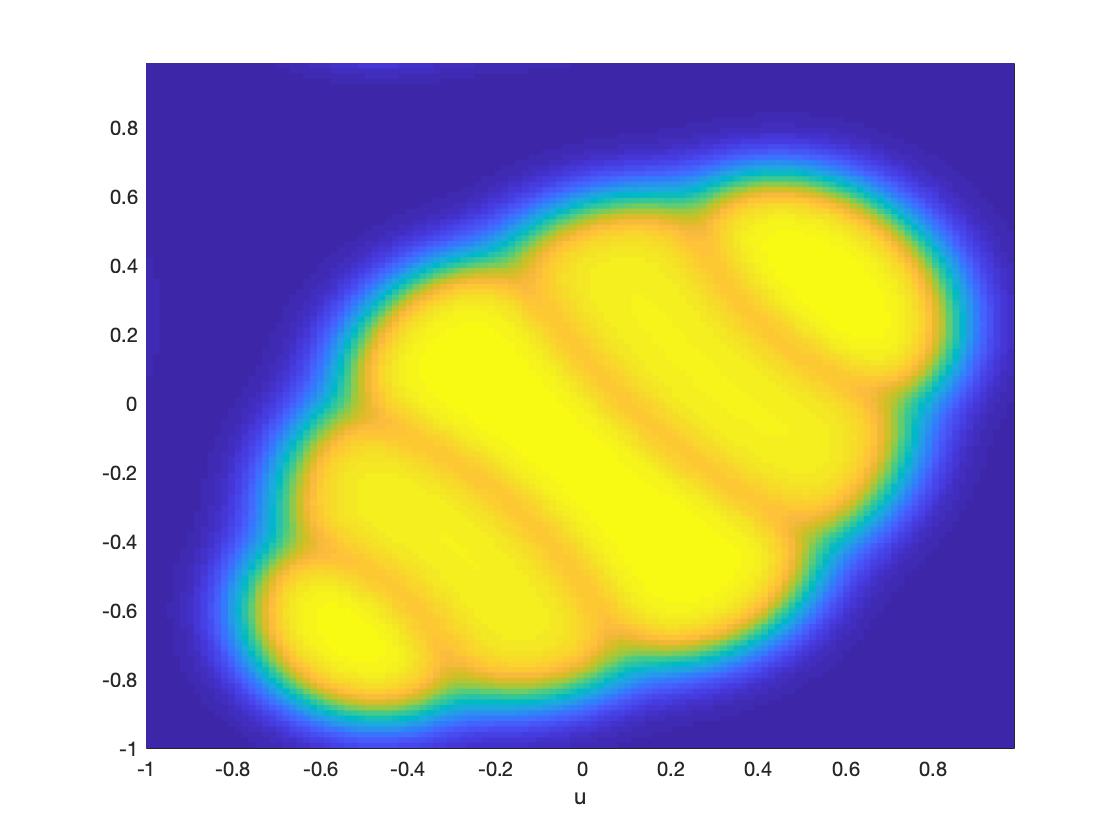}\hskip 0cm}
\subfigure[$\bv:t=4$.]{\includegraphics[width=0.22\textwidth,clip==]{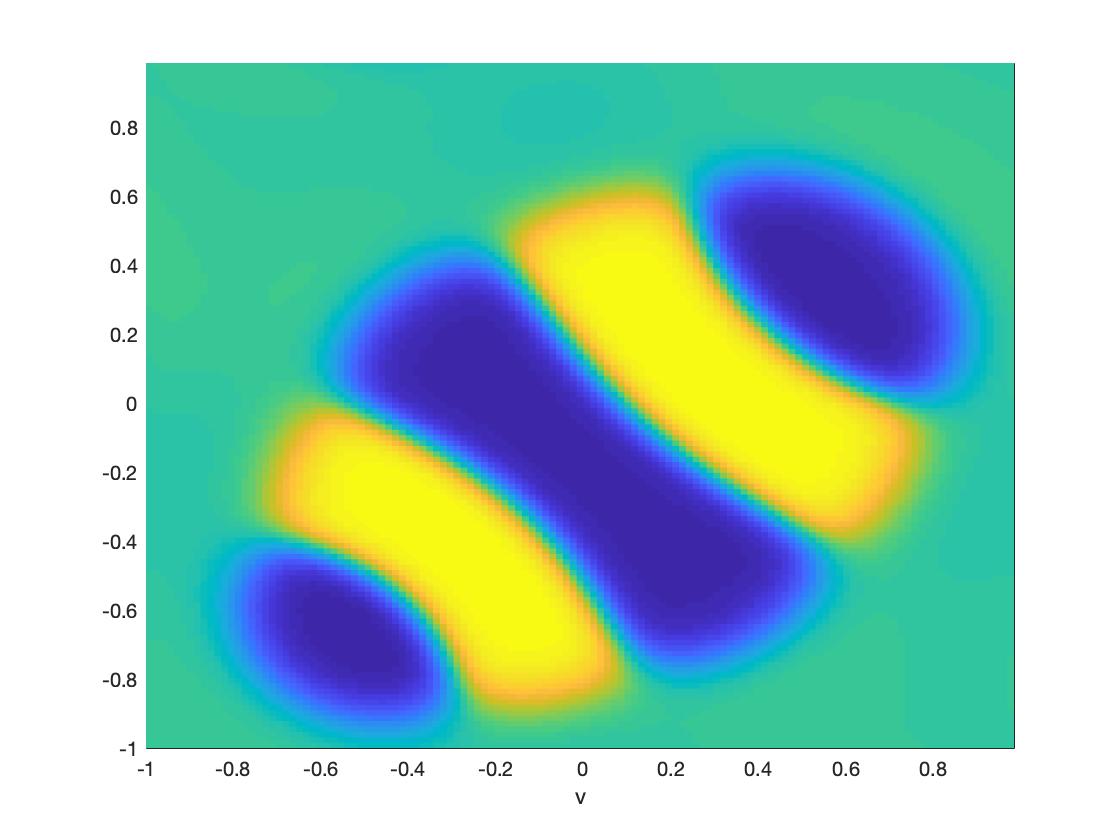}\hskip 0cm}
\caption{The 2D dynamical evolution of the phase variable $\bu,\bv$ for the Coupled-BCP model with the initial condition \eqref{initial} and $G=\sqrt{x+c}$ with $c=10$.}\label{BCP_2}
\end{figure}

\section{Concluding remarks}
How to construct efficient, accurate numerical schemes for gradient flows is a challenging task. The newly IEQ and  SAV approaches  are  developed  in recent years by introducing auxiliary variables.  But the form of auxiliary variables can only be defined as the square root function with respect to nonlinear part of energy or nonlinear potential. 
We  remove the definition restriction that auxiliary variables can only be square root function and develop three classes of  generalized-SAV approach.  Numerical schemes based on these three numerical approaches are efficient as the SAV schemes  i.e., only require solving linear  equations with constant coefficients at each time step.  The small price to pay for the first and third approaches is to solve  an additional nonlinear algebraic system which can be solved at negligible cost.  For the second approach the auxiliary variable can be guaranteed to be positive by choosing  $\tanh$ function or exponential function which IEQ and SAV approaches can not preserve.
Moreover,  all three  approaches lead to  schemes which are unconditionally energy stable.  
We present  ample numerical results to show the efficiency and accuracy of numerical approaches we proposed.  Our numerical results indicate that  the proposed approaches can achieve accurate results which are comparable with ETDRK2 scheme and original SAV schemes. Numerical simulations from coupled Cahn-Hilliard model show that the first approach is more robust and accurate than the second approach.  
 
Although we consider  only time-discretization schemes in this paper,  they can be combined with any consistent finite dimensional Galerkin type approximations in practice,    since the stability  proofs are all based on  variational  formulations with all test functions in the same space as the trial functions.

\bibliographystyle{siamplain}
\bibliography{references}
\end{document}